\documentclass[a4paper,12pt,reqno]{amsart}
\usepackage{graphicx}

\usepackage{color}%for comments

\usepackage{amsmath,amstext,amssymb,amsopn,amsthm}
\usepackage{url}

\usepackage[margin=30mm]{geometry}
\usepackage{eucal,mathrsfs,dsfont}%gives nicer set-names. Use \mathds instead of \mathds
\usepackage{soul}

\newtheorem{theorem}{Theorem}[section]
\newtheorem{corollary}[theorem]{Corollary}
\newtheorem{lemma}[theorem]{Lemma}
\newtheorem{proposition}[theorem]{Proposition}

\theoremstyle{definition}

\theoremstyle{remark}
\newtheorem{remark}[theorem]{Remark}

\newcommand{\eps}{\varepsilon}
\newcommand{\calA}{\mathcal{A}}
\newcommand{\Aa}{\mathcal{A}_{\alpha}}
\newcommand{\calB}{\mathcal{B}}
\newcommand{\calF}{\mathcal{F}}
\newcommand{\calG}{\mathcal{G}}
\newcommand{\calK}{\mathcal{K}}
\newcommand{\calL}{\mathcal{L}}
\newcommand{\calN}{\mathcal{N}}
\newcommand{\calR}{\mathcal{R}}

\newcommand{\R}{\mathds{R}}
\newcommand{\Rd}{\mathds{R}^d}
\newcommand{\N}{{\mathds{N}}}

\newcommand{\Bb}{\calB_b(\Rd)}

\newcommand{\tbfs}[1]{\tilde{b}_{#1 1}(x,y)}
\newcommand{\E}{\mathbb{E}}
\newcommand{\p}{\mathbb{P}}
\newcommand{\scalp}[2]{#1 #2}

\DeclareMathOperator{\supp}{supp}
\DeclareMathOperator{\dist}{dist}

\DeclareMathOperator{\Arg}{Arg}

\title[SDEs driven by multiplicative cylindrical stable noise]{Strong Feller property for SDEs driven by multiplicative cylindrical stable noise}

\author[T. Kulczycki]{Tadeusz Kulczycki}
\author[M. Ryznar]{Micha{\l} Ryznar}
\author[P. Sztonyk]{Pawe{\l} Sztonyk}
\thanks{T. Kulczycki was supported in part by the National Science Centre, Poland, grant no. 2015/17/B/ST1/01233, M. Ryznar was supported in part by the National Science Centre, Poland, grant no. 2015/17/B/ST1/01043, P. Sztonyk was supported in part by the National Science Centre, Poland, grant no. 2017/27/B/ST1/01339.}
\address{Faculty of Pure and Applied Mathematics, Wroc{\l}aw University of Science and Technology, Wyb. Wyspia{\'n}skiego 27, 50-370 Wroc{\l}aw, Poland.}
\email{Tadeusz.Kulczycki@pwr.edu.pl}
\email{Michal.Ryznar@pwr.edu.pl}
\email{Pawel.Sztonyk@pwr.edu.pl}

\begin{document}

\begin{abstract} We consider the stochastic differential equation $dX_t = A(X_{t-}) \, dZ_t$, $ X_0 = x$,
driven by cylindrical $\alpha$-stable process $Z_t$ in $\R^d$, where $\alpha \in (0,1)$ and $d \ge 2$. We assume that the determinant of $A(x) = (a_{ij}(x))$ is bounded away from zero, and $a_{ij}(x)$ are bounded and Lipschitz continuous. We show that for any fixed $\gamma \in (0,\alpha)$ the semigroup $P_t$ of the process $X_t$ satisfies $|P_t f(x) - P_t f(y)| \le c t^{-\gamma/\alpha} |x - y|^{\gamma} ||f||_\infty$ for arbitrary bounded Borel function $f$. Our approach is based on  Levi's method.
\end{abstract}

\maketitle

\section{Introduction}
Let $Z_t = (Z_t^{(1)},\ldots,Z_t^{(d)})^T$ be a cylindrical $\alpha$-stable process, that is $Z_t^{(1)}, \ldots, Z_t^{(d)}$ are independent one-dimensional symmetric  standard  $\alpha$-stable processes of index $\alpha \in (0,1)$, $d \in \N$, $d \ge 2$. We consider the stochastic differential equations 
\begin{equation}
\label{main}
dX_t = A(X_{t-}) \, dZ_t, \quad X_0 = x \in \R^d, 
\end{equation}
where $A(x) = (a_{ij}(x))$ is a $d \times d$ matrix and there are constants $\eta_1, \eta_2, \eta_3 > 0$, such that for any $x, y \in \R^d$, $i, j \in \{1,\ldots,d\}$
\begin{equation}
\label{bounded}
|a_{ij}(x)| \le \eta_1,
\end{equation}
\begin{equation}
\label{determinant}
\det(A(x)) \ge \eta_2,
\end{equation}
\begin{equation}
\label{Lipschitz}
|a_{ij}(x) - a_{ij}(y)| \le \eta_3 |x - y|.
\end{equation}

It is well known that SDEs (\ref{main}) has a unique strong solution $X_t$, see e.g. \cite[Theorem 34.7 and Corollary 35.3]{M1982}. By \cite[Corollary 3.3]{SS2010} $X_t$ is a Feller process.

Let $\E^x$ denote the expected value of the process $X$ starting from $x$ and $\Bb$ denote the set of all Borel bounded functions $f: \R^d \to \R$. For any $t \ge 0$, $x \in \R^d$ and $f \in \Bb$ we put 
\begin{equation}
\label{semigroup}
P_t f(x) = \E^x f(X_t).
\end{equation}

The main result of this paper is the following theorem, which gives the strong Feller property of the semigroup $P_t$.
\begin{theorem} 
\label{mainthm} 
For any $\gamma \in (0,\alpha)$, $\tau > 0$, $t \in (0,\tau]$, $x, y \in \R^d$ and $f \in \Bb$ we have
\begin{equation}
\label{Holder}
|P_t f(x) - P_t f(y)| \le c t^{-\gamma/\alpha} \, |x - y|^{\gamma} \, ||f||_\infty,
\end{equation}
where $c$ depends on $\tau, \alpha, d, \eta_1, \eta_2, \eta_3, \gamma$.
\end{theorem}

Strong Feller property for SDEs driven by additive cylindrical L{\'e}vy processes have been intensively studied recently (see e.g. \cite{PZ2011,WZ2015,DPSZ2016}). The SDE (\ref{main}) (with multiplicative noice) was studied by Bass and Chen in \cite{BC2006}. They proved existence and uniqueness of weak solutions of SDE (\ref{main}) under very mild assumptions on matrices $A(x)$ (i.e. they assumed that $A(x)$ are continuous and bounded in $x$ and nondegenerate for each $x$). In \cite{KR2017} SDE (\ref{main}) was studied for diagonal matrices $A(x)$, which diagonal coefficients are bounded away from zero, from infinity and H{\"o}lder continuous. Under these assumptions the corresponding transition density $p^A(t,x,y)$ was constructed and H{\"o}lder estimates $x \to p^A(t,x,y)$ were obtained. These estimates imply strong Feller property of the corresponding semigroup.

The case of non-diagonal matrices $A(x)$, treated in this paper, is much more difficult. The strong Feller property for semigroups generated by solutions to SDEs is often obtained by suitable versions of the Bismut-Elworthy-Li formula. We were not able to get such formula but  we use instead Levi's method to construct the semigroup $P_t$ and to obtain Theorem \ref{mainthm}. However there are many problems in applying this method to the case of non-diagonal matrices $A(x)$. Therefore we had to introduce some new ideas. Below we briefly describe the main steps in our approach.

The first problem with Levi's method in our case is that the standard approximation of the transition density (the so-called ``frozen density'') does not have good integrability properties. To overcome this we truncate the L{\'e}vy measure of the process $Z_t$ in a convenient way.  Then, using Levi's method, we construct the transition density (denoted by $u(t,x,y)$) of the solution of (\ref{main}) driven by this truncated process. As usual we represent $u(t,x,y)$ as a series $\sum_{n=0}^{\infty} q_n(t,x,y)$. Typically, in many papers using Levi's method, the first step was to obtain precise bounds for $q_0(t,x,y)$ which allow to estimate $q_n(t,x,y)$ inductively 
point-wise. In our case it seems impossible to obtain such precise bounds, hence  we prove (see Proposition \ref{integralq0}) some different kind of results for $q_0(t,x,y)$, which are sufficient for our purposes. The main tools to prove Proposition \ref{integralq0} are Lemma \ref{intA} and the estimates (\ref{gdelta3}). These key estimates (\ref{gdelta3}) are proven using the techniques and results from \cite{KR2016}, \cite{KR2018} and \cite{S2017}. After constructing the transition density $u(t,x,y)$ we use the technique developed by Knopova and Kulik \cite{KK2018} to show that $u(t,x,y)$ satisfies the appropriate heat equation in the so-called approximate setting. In the next step we construct the semigroup $T_t$ for the solution of SDE (\ref{main}) (driven by the not truncated process). Roughly speaking, this construction is based on adding long jumps to the truncated process. Next we show that $u(t,x) := T_t f(x)$ satisfies the appropriate heat equation in the approximate setting (see Lemma \ref{heat_p}),  which allows to prove that the constructed semigroup $T_t$ is in fact the semigroup $P_t$.

Our current technique is restricted to the case $\alpha \in (0,1)$. The main difficulty for $\alpha \in [1,2)$ is that in such case one has to effectively estimate the expression
\begin{equation}
\label{alpha2}
p_y(t,x+a_i(x)w) + p_y(t,x-a_i(x)w) - 2p_y(t,x)
\end{equation} 
instead of 
\begin{equation}
\label{alpha1}
p_y(t,x+a_i(x)w) -p_y(t,x),
\end{equation} 
where $p_y(t,x)$ is the frozen density for the truncated process (see Section 3 for the precise definition of $p_y(t,x)$) and $a_i(x) = (a_{1i}(x),\ldots, a_{di}(x))$. Our crucial estimate (\ref{gdelta3}) allows suitable  estimate of (\ref{alpha1}) but fails  to bound (\ref{alpha2}) in a way sufficient for our purpose. 

It is worth mentioning  that strong Feller property and gradient estimates for the semigroups associated to SDEs driven by L{\'e}vy processes in $\R^d$ with jumps, with absolutely continuous  L{\'e}vy measures,  have been studied for many years (see e.g. \cite{T2010,SSW2012,KM2014,Z2013,WXZ2015,XZ2017}). 

One may ask about further regularity properties of the semigroup $P_t$, in particular about boundedness of the operators $P_t: L^1(\R^d) \to L^{\infty}(\R^d)$, which is related to the boundedness of the transition densities of $P_t$.   It turns out that for some choices of matrices $A(x)$ (satisfying (\ref{bounded}), (\ref{determinant}), (\ref{Lipschitz})) and for some $t > 0$ the operators $P_t: L^1(\R^d) \to L^{\infty}(\R^d)$ are not bounded (see Remark \ref{example1} and Remark \ref{example1Uwaga}). Nevertheless we have the following regularity result.
\begin{theorem} 
\label{PtL1Linfty} 
For any $\gamma \in (0,\alpha/d)$, $\tau > 0$, $t \in (0,\tau]$, $x \in \R^d$ and $f \in L^1(\R^d) \cap L^{\infty}(\R^d)$ we have
\begin{equation*}
|P_t f(x)| \le c t^{-\gamma d/\alpha} \, \|f\|_\infty^{1-\gamma} \, \|f\|_1^{\gamma},
\end{equation*}
where $c$ depends on $\tau, \alpha, d, \eta_1, \eta_2, \eta_3, \gamma$.
\end{theorem}

The paper is organized as follows. In Section 2 we study properties of the transition density of a suitably truncated one-dimensional stable process. These properties are crucial in the sequel. In Section 3 we construct the transition density $u(t,x,y)$ of the solution of (\ref{main}) driven by the truncated process. We also show that it satisfies the appropriate equation in the approximate setting. In Section 4 we construct the transition semigroup of the solution of (\ref{main}). We also prove Theorems \ref{mainthm} and \ref{PtL1Linfty}.

\section{Preliminaries}
All constants appearing in this paper are positive and finite. In the whole paper we fix $\tau > 0$, $\alpha \in (0,1)$, $d \in \N$, $d \ge 2$, $\eta_1, \eta_2, \eta_3$, where $\eta_1, \eta_2, \eta_3$ appear in (\ref{bounded}), (\ref{determinant}) and (\ref{Lipschitz}). We adopt the convention that constants denoted by $c$ (or $c_1, c_2, \ldots$) may change their value from one use to the next. In the whole paper, unless is explicitly stated otherwise, we understand that constants denoted by $c$ (or $c_1, c_2, \ldots$) depend on $\tau, \alpha, d, \eta_1, \eta_2, \eta_3$. We  also understand that they may depend on the choice of the constants $\eps$ and $\gamma$. We write $f(x) \approx g(x)$ for $x \in A$ if $f, g \ge 0$ on $A$ and there is a constant $c \ge 1$ such that $c^{-1} f(x) \le g(x) \le c f(x)$ for $x \in A$. The standard inner product for $x, y \in \R^d$ we denote by $xy$. 

For any $t > 0$, $x \in \Rd$ we define the measure $\sigma_t(x,\cdot)$ by 
\begin{equation}
\label{sigmatx}
\sigma_t(x,A) = \p^x(X_t \in A),
\end{equation}
for any Borel set $A \subset \R^d$. $\p^x$ denotes the distribution of the process $X$ starting from $x\in \R^d$. For any $t > 0$, $x \in \Rd$ we have 
\begin{equation}
\label{semigroup_p}
P_t f(x) = \int_{\R^d} f(y) \sigma_t(x,dy), \quad f \in \Bb.
\end{equation}

It is well known that the density of the L{\'e}vy measure of the one-dimensional symmetric standard $\alpha$-stable process is given by $\calA_{\alpha} |x|^{-1-\alpha}$, where $\calA_{\alpha} = 2^{\alpha} \Gamma((1+\alpha)/2)/(\pi^{1/2} |\Gamma(-\alpha/2)|)$. In the sequel we will need to truncate this density. The truncated density will be denoted by $\mu^{(\delta)}(x)$.  Let $\mu^{(\delta)}: \R \setminus \{0\} \to [0,\infty)$ where $\delta \in (0,1]$. For $x \in (0,\delta]$ we put $\mu^{(\delta)}(x) = \Aa |x|^{-1-\alpha}$, for $x \in (\delta, 2 \delta)$ we put $\mu^{(\delta)}(x) \in (0,\Aa |x|^{-1-\alpha})$ and for $x \ge 2 \delta$ we put $\mu^{(\delta)}(x) = 0$. Moreover, $\mu^{(\delta)}$ is defined so that it is weakly decreasing, weakly convex and $C^1$ on $(0,\infty)$ and satisfies $\mu^{(\delta)}(-x) =\mu^{(\delta)}(x)$ for $x \in (0,\infty)$. 

We also define
$$
\calG^{(\delta)} f(x) = \lim_{\zeta \to 0^+} \int_{|w| > \zeta} (f(x+w) - f(x)) \mu^{(\delta)}(w) \, dw.
$$
By $g_t^{(\delta)}$ we denote the heat kernel corresponding to $\calG^{(\delta)}$ that is 
$$
\frac{\partial}{\partial t} g_t^{(\delta)}(x) = \calG^{(\delta)} g_t^{(\delta)}(x), \quad t > 0, \, x \in \R,
$$
$$
\int_{R} g_t^{(\delta)}(x) \, dx = 1, \quad t > 0.
$$
It is well known that $g_t^{(\delta)}$ belongs to $C^1((0,\infty))$ as a function of $t$ and  belongs to $C^2(\R)$ as a function of $x$.  

For any $\eps \in (0,1], \tau >0 $, $t \in (0,\infty)$ and $x \in \R$ we define
\begin{equation}
\label{hdefinition}
h_t^{(\eps)}(x) = \left\{              
\begin{array}{lll}                   
\frac t{(|x|+t^{1/\alpha})^{1+\alpha}}&\text{for}& |x| < \eps, \\  
c_{\eps} t^{1+ (d-1)/\alpha} e^{-|x|} &\text{for}& |x| \ge \eps,
\end{array}       
\right. 
\end{equation}
where $c_{\eps} =  \frac{e^\eps}{(1+\tau^{1/\alpha})^{1+\alpha}\tau^{(d-1)/\alpha}}$. 
%where $c_{\eps} = \left[\left(\min_{x \in [\eps,\infty)} \frac{e^x}{x^{1 + \alpha}}\right) \wedge e^{\eps}\right]  
%\left[\frac{1}{\tau^{(d-1)/\alpha}} \wedge \frac{1}{\tau^{1+ d/\alpha}}\right]$.
 The constant $c_{\eps}$ is chosen so that for any $t \in (0,\tau]$ the function $x \to h_t^{(\eps)}(x)$ is nonincreasing on $[0,\infty)$.

\begin{lemma}
\label{gtht}
For any $\eps \in (0,1],$ there exist $c$ such that for 
$\delta =\varepsilon\min\{ \frac{\alpha}{8(\alpha+d+2)},\frac{1}{4d(\eta_1\vee 1)^2}\}$, and any $t \in (0,\tau]$, $x,y \in \R$, we have
\begin{align}
\label{gdelta1}
g_t^{(\delta)}(x) 
& \le c h_t^{(\eps)}(|x|),\\
%\label{gdelta2}
|g_t^{(\delta)}(x) - g_t^{(\delta)}(y)| 
%&\le c \frac{|x-y|}{t^{1/\alpha} + (|x|\vee |y|)}h_t^{(\eps)}(|x|\wedge |y|)\\
\label{gdelta3}
& \le c|x-y| \left(\frac{h_t^{(\eps)}(|x|)}{t^{1/\alpha} + |x|}+ \frac{h_t^{(\eps)}(|y|)}{t^{1/\alpha} + |y|}\right).
\end{align}
\end{lemma}
\begin{proof}

First we consider a general case of heat kernels $g^{(\delta,n)}$ on $\R^n$ for $\delta \in (0,1]$,  
$n\in\{1,2,...\}$ such that
$$
  g_t^{(\delta,n)}(x) = \frac{1}{(2\pi)^n} \int_{\R^n} e^{-i\scalp{x}{u}} e^{-t\Phi_\delta^{(n)}(u)}\, du,
$$
where
$$
  \Phi_\delta^{(n)} (u) 
	=  \int_{\R^n} (1 - \cos(uy)) \mu^{(\delta,n)}(y)\, dy,\quad u \in\R^n,
$$
and $\mu^{(\delta,n)}(y)=\mu^{(\delta,n)}(|y|)$ 
is isotropic unimodal L{\'e}vy density such that $\mu^{(\delta,n)}(y) \approx |y|^{-n-\alpha}$
for $|y|\leq \delta$, and $\mu^{(\delta,n)}=0$ for $|y|\geq 2\delta$. In the proof of this lemma we assume that constants $c$ may additionally depend on $n$. 
%We note that the function $\mu^{(\delta)}$ is symmetric and unimodal so 
It follows from Lemma 1 in \cite{BGR2014} that
$$
  \tfrac{2}{n\pi^2} H(|u|) \leq \Phi_\delta^{(n)} (u) \leq 2 H(|u|),\quad u\in\R^n,
$$
where
$$
  H(r) = \int_{\R^n} (1\wedge (r|y|)^2 ) \mu^{(\delta,n)}(y)\, dy,\quad r\geq 0,
$$
hence we easily obtain
$$
  \Phi_\delta^{(n)} (u) 
	 \geq c 
	|u|^2 \int_0^{\frac{1}{|u|}\wedge \delta} r^{1-\alpha}\, dr 
	 = c \tfrac{1}{\delta^{\alpha}} \left((\delta|u|)^2\wedge (\delta|u|)^\alpha\right), \quad u\in\R^n.
$$
Similarly
$$
  \Phi_\delta^{(n)}(u) \leq c
	\tfrac{1}{\delta^{\alpha}} \left((\delta|u|)^2\wedge (\delta|u|)^\alpha\right), \quad u\in\R^n.
$$
In particular the symbol $\Phi_\delta^{(n)}$ has global weak lower scaling property with index $\alpha$
(see \cite{BGR2014}).
This yields, by Theorem 21 of  \cite{BGR2014},
\begin{equation*}%\label{ondiagestg}
  g_t^{(\delta,n)}(x) \leq c \min\{(H^{-1}(1/t))^n, tH(1/|x|) |x|^{-n}\}  ,\quad t>0, x\in\R^n.
\end{equation*}
Observing that $H(r)\le c r^\alpha$ for  $r\ge 0 $ and $H^{-1}(1/t)\le  c \frac 1{t^{1/\alpha}}$, for $t\le \tau $, we have

\begin{equation}\label{stableestg}
g_t^{(\delta,n)}(x) \leq c \frac t{(|x|+t^{1/\alpha})^{n+\alpha}}\quad  x\in\R^n, t\in(0,\tau].
\end{equation}

Let $t\le 1\wedge\tau$.
Using Lemma 4.2 from \cite{S2017} we get 
\begin{align*}
  g_t^{(\delta,n)}(x) 
	& \leq  e^{\frac{-|x|}{8\delta}\log\left(\frac{\delta |x|}{tm_0}\right)} g_t^{(\delta,n)}(0)   =   
	        \left(\tfrac{tm_0}{\delta|x|}\right)^{\frac{|x|}{8\delta}} g_{t}^{(\delta,n)}(0) \\ 
	& \leq  c  t^{\frac{|x|}{8\delta}-\frac{n}{\alpha}}   e^{\frac{-|x|}{8\delta}\log\left(\frac{\delta |x|}{m_0}\right)}, \quad |x| \geq \tfrac{em_0}{\delta} t,
\end{align*}
where $m_0 = \int_{\R^n} |y|^2  \mu^{(\delta,n)}(y)\, dy\approx \delta^{2-\alpha}.$ This yields
$$
  g_t^{(\delta,n)}(x) \leq c   t^{1+\frac{d-1}{\alpha}} e^{\frac{-|x|}{8\delta}\log\left(\frac{\delta |x|}{m_0}\right)},
$$
provided $|x| \geq \max\{8\delta(1+\frac{n+d-1}{\alpha}),\tfrac{em_0}{\delta} t\}.$ We observe that there exists $c_1=c_1(\delta,\alpha,d,n)$ such that
$$
  e^{\frac{-|x|}{8\delta}\log\left(\frac{\delta |x|}{m_0}\right)} \leq c_1 e^{-(\alpha+d)|x|/\alpha}, \quad x\in\R^n,
$$
so  we obtain
\begin{equation}\label{estexpg}
  g_t^{(\delta,n)}(x) \leq c_2 t^{1+\frac{d-1}{\alpha}} e^{-(\alpha+d)|x|/\alpha}, \quad |x| \geq \max\{8\delta(1+\tfrac{n+d-1}{\alpha}),\tfrac{em_0}{\delta} t\},
\end{equation}
with $c_2=c_2(\delta,\alpha,d,n).$

Let $1\le t\le \tau$.
Using again  Lemma 4.2 from \cite{S2017} we get 
\begin{align}
  g_t^{(\delta,n)}(x) 
	& \leq  e^{\frac{-|x|}{8\delta}\log\left(\frac{\delta |x|}{tm_0}\right)} g_t^{(\delta,n)}(0) \leq   e^{\frac{-|x|}{8\delta}}g_1^{(\delta,n)}(0)\nonumber\\  
	        %\left(\tfrac{tm_0}{\delta|x|}\right)^{\frac{|x|}{8\delta}} g_{t}^{(\delta,n)}(0) \\ 
	& \leq  c t^{1+\frac{d-1}{\alpha}}  e^{\frac{-|x|}{8\delta}}, \quad |x| \geq \tfrac{em_0}{\delta} t.
\label{estexpg2}\end{align}
 %Taking $\delta= \frac{\varepsilon\alpha}{8(\alpha+d)}$ we get 
%$g_t^{(\delta)}(x) \leq c t^{1+\frac{d-1}{\alpha}} e^{-(\alpha+d)|x|/\alpha}\leq c t^{1+\frac{d-1}{\alpha}} e^{-|x|}$ 
%for $|x| \geq \max\{\varepsilon,\tfrac{em_0}{\delta} t\}.$

Let $\delta < \frac{\varepsilon\alpha}{8(\alpha+d+2)}.$
Now, (\ref{gdelta1}) follows from \eqref{stableestg}, \eqref{estexpg} and \eqref{estexpg2} with $n=1$.

The function $\mu^{(\delta)}$ satisfies the assumption of Theorem 1.5 in \cite{KR2016} which yields that there exists a L\'evy process $X_t^{(3)}$ in $\R^3$ with
the characteristic exponent $\Phi^{(3)}_\delta(u) = \Phi^{(3)}_\delta(|u|),$ $u\in\R^3$ and the radial, radially nonincreasing transition density
$g_t^{(\delta,3)}(x) = g_t^{(\delta,3)}(|x|)$ satisfying
\begin{displaymath}
	g_t^{(\delta,3)}(r) = \frac{-1}{2\pi r}\frac{d}{dr} g_t^{(\delta,1)}(r),\quad r>0.
\end{displaymath}
Furthermore it follows from the proof of Theorem 1.5 in \cite{KR2016} that the L\'evy measure of the process $X_t^{(3)}$ is given by 
$\mu^{(\delta,3)}(dx) = \mu^{(\delta,3)}(x)\, dx $, where $\mu^{(\delta,3)}(R) = \frac{-1}{2\pi R} \frac{d\mu^{(\delta,1)}}{dR}(R).$ In particular $\mu^{(\delta,3)}(R)$ is nonincreasing on $(0,\infty)$ and we have 
$\mu^{(\delta,3)}(x) = \frac{\Aa}{2\pi(2+\alpha)}|x|^{-3-\alpha}$ for $|x|\leq \delta$, $\mu^{(\delta,3)}(x)\leq \frac{\Aa}{2\pi(2+\alpha)}\delta^{-3-\alpha},$
for $\delta \leq |x| \leq 2\delta,$ and $\mu^{(\delta,3)}(x)=0$ for $|x|>2\delta.$

By \eqref{stableestg}, \eqref{estexpg} and \eqref{estexpg2}, with $n=3$, we obtain
$$
  g_t^{(\delta,3)}(x) \leq  c \frac t{(|x|+t^{1/\alpha})^{3+\alpha}}\quad  x\in\R^3, t\in(0,\tau]
$$
and
$$
  g_t^{(\delta,3)}(x) \leq c t^{1+\frac{d-1}{\alpha}}e^{-(\alpha+d)|x|/\alpha}\le c \frac{ t^{1+\frac{d-1}{\alpha}}}{(|x|+t^{1/\alpha})^2}e^{-|x|}, \quad |x|>\max\{\varepsilon,c_3 t\},  t\in (0,\tau].
$$
The  above two inequalities  yield
$$
  g_t^{(\delta,3)}(x) \leq c \frac{h_t^{(\varepsilon)}(|x|)}{(|x|+t^{1/\alpha})^2},\quad x\in\R^3,t\in(0,\tau],
$$
and
$$
  \left|\frac{d}{dr} g_t^{(\delta,1)}(r)\right| \leq c \frac{rh_t^{(\varepsilon)}(r)}{(r+t^{1/\alpha})^2} \leq c \frac{h_t^{(\varepsilon)}(r)}{r+t^{1/\alpha}},\quad r>0,t\in(0,\tau].
$$
 Since $h_t^{(\varepsilon)}$ is nonincreasing, by the Lagrange theorem, we get
$$
  | g_t^{(\delta,1)}(x) - g_t^{(\delta,1)}(y) | \leq c |x-y| \left( \frac{h_t^{(\varepsilon)}(|x|)}{|x| + t^{1/\alpha}}
	+ \frac{h_t^{(\varepsilon)}(|y|)}{|y| + t^{1/\alpha}}\right),\quad x,y\in\R.
$$
\end{proof}

\begin{lemma}
\label{xxprime}
Let $\eps \in (0,1]$. For any $t \in (0,\tau]$, 
$x, x' \in \R$ if $|x- x'| \le t^{1/\alpha}$ and $|x - x'| \le \eps/4$ then
$$
h_t^{(\eps)}(x') \le c h_t^{(\eps)}(x/2).
$$
\end{lemma}
\begin{proof}
Assume first that $|x| \wedge |x'| \le \eps/2$. Then by the definition of $h_t^{(\eps)}(x)$ we have
$$
h_t^{(\eps)}(x) = \frac{t}{(t^{1/\alpha} + |x|)^{1+\alpha}}
\ge \frac{t}{(t^{1/\alpha} + |x - x'| + |x'|)^{1+\alpha}}
\ge \frac{t}{(2 t^{1/\alpha} + 2|x'|)^{1+\alpha}}
\ge c h_t^{(\eps)}(x').
$$
Assume now that $|x| \wedge |x'| > \eps/2$. Then we have $|x'| \ge |x| - |x - x'| \ge |x| - |x|/2 = |x|/2$. Hence $h_t^{(\eps)}(x') \le 
h_t^{(\eps)}(x/2)$.
\end{proof}

\begin{lemma}
\label{product_gt}
Let $\eps \in (0,1]$, $\delta =\varepsilon\min\{ \frac{\alpha}{8(\alpha+d+2)},\frac{1}{4d(\eta_1\vee 1)^2}\}$. For any 
$t \in (0,\tau]$, $x, x' \in \R^d$ if $|x- x'| \le t^{1/\alpha}$ and $|x - x'| \le \delta$ then
\begin{equation}
\label{product_gt1}
\left| \prod_{i=1}^d  g_t^{(\delta)}\left({x_i}\right) - 
\prod_{i=1}^d  g_t^{(\delta)}\left({x'_i}\right)\right|
\le c  \left(\prod_{i=1}^d h_t^{(\eps)}(x_i/2) \right)\left[1\wedge \sum_{j = 1}^d t^{-1/\alpha} |x_j - x'_j|\right].
\end{equation}
\end{lemma}
\begin{proof}
By Lemma \ref{gtht} we get
\begin{eqnarray*}
&&
\left| \prod_{i=1}^{d}  g_t^{(\delta)}\left({x_i}\right) - 
\prod_{i=1}^{d}  g_t^{(\delta)}\left({x'_i}\right)\right|\\
&&
\le  \sum_{j=1}^d \left[\left|g_t^{(\delta)}\left(x_{j}\right) - g_t^{(\delta)}\left(x'_{j}\right)\right| 
\prod_{i\ne j, 1\le i\le d} g_t^{(\delta)}\left(|x_{i}|\wedge |x'_{i}|\right) \right]\\
&& \le c \left(\prod_{i=1}^d h_t^{(\eps)}(|x_i|\wedge |x'_i|) \right) \sum_{j = 1}^d  
\frac{|x_j-x'_j|}{t^{1/\alpha}}.
\end{eqnarray*}

Clearly we have
$$\left| \prod_{i=1}^{d}  g_t^{(\delta)}\left({x_i}\right) - 
\prod_{i=1}^{d}  g_t^{(\delta)}\left({x'_i}\right)\right|\le \prod_{i=1}^d g_t^{(\delta)}(|x_i|\wedge |x'_i|).$$

Now the assertion follows from Lemmas \ref{gtht} and \ref{xxprime}.
\end{proof}

\begin{lemma}
\label{integralgt}
Let $\eps \in (0,1]$, $\delta =\varepsilon\min\{ \frac{\alpha}{8(\alpha+d+2)},\frac{1}{4d(\eta_1\vee 1)^2}\}$ and let $a\in\R$ be such
that $|a|\leq \tfrac{\eps}{4\delta}$. Then there exists $c$ such that for any $t \in (0,\tau]$, $x \in \R$ we have
\begin{equation}
\label{gdelta31}
\int_{\R} |g_t^{(\delta)}(x + aw) - g_t^{(\delta)}(x)| \mu^{(\delta)}(w) \, dw \le \frac{c |a|^{\alpha} h_t^{(\eps)}\left(\tfrac{x}{2}\right)}{t}.
\end{equation}
\end{lemma}
\begin{proof}
First we note that
\begin{displaymath}
	\int_{\R} |g_t^{(\delta)}(x + aw) - g_t^{(\delta)}(x)| \mu^{(\delta)}(w) \, dw \leq \int_{|w|<2\delta} |g_t^{(\delta)}(x + aw) - g_t^{(\delta)}(x)| \Aa |w|^{-1-\alpha} \, dw,
\end{displaymath}
and by the substitution $s=aw$ we have
\begin{displaymath}
	\int_{|w|<2\delta} |g_t^{(\delta)}(x + aw) - g_t^{(\delta)}(x)| |w|^{-1-\alpha} \, dw 
	= |a|^{\alpha} \int_{|s|<2\delta|a|} |g_t^{(\delta)}(x + s) - g_t^{(\delta)}(x)| |s|^{-1-\alpha} \, ds.
\end{displaymath}

Now we estimate the latter integral. Let
\begin{align*}
	&\int_{|s|<2\delta|a|} |g_t^{(\delta)}(x + s) - g_t^{(\delta)}(x)| |s|^{-1-\alpha} \, ds 
	  =  \int_{|s|<t^{1/\alpha} \wedge (2\delta|a|)} |g_t^{(\delta)}(x + s) - g_t^{(\delta)}(x)| |s|^{-1-\alpha} \, ds \\
	  & \quad \quad \quad \quad \quad \quad \quad \quad \quad \quad \quad \quad \quad \quad  + \int_{t^{1/\alpha} \wedge (2\delta|a|) \leq |s|<2\delta|a|} |g_t^{(\delta)}(x + s) - g_t^{(\delta)}(x)| |s|^{-1-\alpha} \, ds \\
		& \quad \quad \quad \quad \quad \quad \quad \quad \quad \quad \quad \quad \quad \quad  =: I_1 + I_2.
\end{align*}

Using (\ref{gdelta3}) we get
\begin{align*}
  I_1
	& \leq  c \int_{|s|< t^{1/\alpha} \wedge (2\delta|a|) } |s|^{-\alpha} \left(\frac{h_t^{(\varepsilon)}(x+s)}{t^{1/\alpha} + |x+s|} + 
	  \frac{h_t^{(\varepsilon)}(x)}{t^{1/\alpha} + |x|}\right)\, ds \\
	& \leq  c t^{-1/\alpha} \int_{|s|< t^{1/\alpha} \wedge (2\delta|a|) } |s|^{-\alpha} \left(h_t^{(\varepsilon)} (|x+s|) + h_t^{(\varepsilon)} (|x|)\right) \, ds.
\end{align*}
If $|x|\geq 2t^{1/\alpha}$ then for $|s|<t^{1/\alpha}$ we have $h_t^{(\varepsilon)}(|x+s|)\leq h_t^{(\varepsilon)}\left(\frac{|x|}{2}\right)$, since $h_t^{(\varepsilon)}$ is nonincreasing. If $|x|\leq 2t^{1/\alpha}$ and $|s|<t^{1/\alpha}$ then $h_t^{(\varepsilon)} \left(\frac{|x|}{2}\right) \geq c t^{1/\alpha}$ and $h_t^{(\varepsilon)} (|x+s|)\leq t^{1/\alpha}$, hence we obtain $h_t^{(\varepsilon)} (|x+s|) \leq c h_t^{(\varepsilon)} \left(\frac{|x|}{2}\right)$.
This yields
\begin{displaymath}
	I_1 \leq c  t^{-1/\alpha}  h_t^{(\varepsilon)} \left(\tfrac{|x|}{2}\right) \int_{|s|< t^{1/\alpha} } |s|^{-\alpha}\, ds = \frac{c}{t}  h_t^{(\varepsilon)} \left(\tfrac{|x|}{2}\right).
\end{displaymath}

Now we estimate $I_2$. If $t^{1/\alpha}>2\delta|a|$ then $I_2=0$ so we assume that $t^{1/\alpha}\leq 2\delta|a|$ and using \eqref{gdelta1} we obtain
\begin{displaymath}
	I_2 \leq c \int_{  t^{1/\alpha} \leq|s|\leq  2\delta|a|} \left( h_t^{(\varepsilon)}(|x+s|) + h_t^{(\varepsilon)}(|x|)\right)|s|^{-1-\alpha}\, ds.
\end{displaymath}

Since we have
\begin{displaymath}
	\int_{|s|\geq t^{1/\alpha} } h_t^{(\varepsilon)}(|x|) |s|^{-1-\alpha}\, ds = \frac{c}{t}h_t^{(\varepsilon)}(|x|),
\end{displaymath}
we only need to estimate
\begin{displaymath}
	J = \int_{t^{1/\alpha} \leq|s|\leq  2\delta|a|} h_t^{(\varepsilon)}(|x+s|) |s|^{-1-\alpha}\, ds.
\end{displaymath}

Let $g^{(\infty)}_t$ denote the transition density of the one-dimensional symmetric standard $\alpha$-stable process.
It follows from \cite{BG60} that $g^{(\infty)}_t(x) \approx \frac{t}{(t^{1/\alpha} + |x|)^{1+\alpha}}$, hence
$h_t^{(\eps)}(y) \leq c g^{(\infty)}_t(y),$ for all $t \in (0,\tau]$, $y\in\R$. 
Noting also that $ |s|^{-1-\alpha}\le c \frac{g_t^{(\infty)}(s)}t$ for $|s|\ge t^{1/\alpha},$ and
using Chapman-Kolmogorov equation for $g^{(\infty)}_t$ we get

\begin{align*}
J
&\le 
 \int_{ |s|\ge t^{1/\alpha}} h_t^{(\eps)}(x + s)|s|^{-1-\alpha} \, ds \\
&\le 
 \frac ct \int_{ \R}g_t^{(\infty)}(x + s)g_t^{(\infty)}(s)\, ds
= \frac ct g_{2t}^{(\infty)}(x), 
\end{align*}
which yields $J\leq \tfrac{ch_t^{(\eps)}(x)}{t}$ for $|x|\leq |\eps|.$

Let now $|x|\geq \eps \geq 4\delta|a|\geq 2t^{1/\alpha}.$ Then $|x+s|\geq |x|/2$ for $s\leq 2\delta|a|$, and we obtain
\begin{align*}
  J
	& \leq  \int_{t^{1/\alpha} \leq|s|\leq  2\delta|a|} h_t^{(\varepsilon)}(|x+s|) |s|^{-1-\alpha}\, ds 
	\leq  h_t^{(\varepsilon)}\left(\tfrac{|x|}{2}\right) \int_{t^{1/\alpha} \leq|s| } |s|^{-1-\alpha}\, ds \\
	& \leq  \frac{c}{t} h_t^{(\varepsilon)}\left(\tfrac{|x|}{2}\right).
\end{align*}
\end{proof}

\begin{lemma}
\label{g_derivatives}
For any $\eps \in (0,1]$, $\delta =\varepsilon\min\{ \frac{\alpha}{8(\alpha+d+2)},\frac{1}{4d(\eta_1\vee 1)^2}\}$ and $m,n\in\N$, $n \ge 2$ there exists 
$c = c(m, n, \tau, \alpha, d, \eta_1, \eta_2, \eta_3, \eps, \delta)$ such that for any $t \in (0,\tau]$, $x \in \R$ we have
$$
  \left|\frac{d^m}{dx^m} g_t^{(\delta)}(x)\right| \le c t^{-(1+m)/\alpha} (1 + |x|)^{-n}.
$$
\end{lemma}
\begin{proof} In the proof we assume that constants $c$ may additionally depend on $m$ and $n$.
We use Theorem 3 of \cite{KS2015}. Let $f(s)=\Aa s^{-1-\alpha}$ for $s\le \delta$ and
$f(s)=\Aa \delta^{n-1-\alpha} s^{-n}$ for $s>\delta$. It is then obvious that the assumptions
(1) and (2) of Theorem 3 in \cite{KS2015} hold and it follows that
$$
  \left|\frac{d^m}{dx^m} g_t^{(\delta)}(x)\right| \le c t^{-(1+m)/\alpha}\min\left\{1,t^{1+1/\alpha}f(|x|)+\left(1+\frac{|x|}{t^{1/\alpha}}\right)^{-n}\right\}, \quad x\in\R, t\in (0,\tau].
$$
Clearly, for $|x| \ge 1$ and $t \in (0,\tau]$ we have $f(|x|) \approx |x|^{-n}$ and $\left(1+\frac{|x|}{t^{1/\alpha}}\right)^{-n} \le c |x|^{-n}$. This implies the assertion of the lemma.
\end{proof}

\begin{lemma}
\label{levy1} There is a constant $C=C(\alpha)$ such that for $ a\ge 0$, and any $t>0$, 

$$\int_{\R} \left(\frac{(a+|w|)|w|}{ t^{1/\alpha}}\wedge1\right)\mu^{(\delta)}(w) \, dw\le C \frac {t^{1/2}+a^{\alpha}}t.
$$
\end{lemma}
\begin{proof}

We have $$\left(\frac{(a+|w|)|w|}{ t^{1/\alpha}}\wedge1\right)\le \left(\frac{a|w|}{ t^{1/\alpha}}\wedge1\right)+\left(\frac{w^2}{ t^{1/\alpha}}\wedge1\right).$$
Hence, using $\mu^{(\delta)}(w)\le \frac c{|w|^{1+\alpha}}$, we obtain 

$$\int_{\R} \left(\frac{(a+|w|)|w|}{ t^{1/\alpha}}\wedge1\right)\mu^{(\delta)}(w) \, dw\le c\int_{\R}\left(\frac{a|w|}{ t^{1/\alpha}}\wedge1\right) \frac {dw}{|w|^{1+\alpha}} \,  + c\int_{\R}\left(\frac{w^2}{ t^{1/\alpha}}\wedge1\right)\frac {dw}{|w|^{1+\alpha}}  .
$$
Next,

$$\int_{\R}\left(\frac{a|w|}{ t^{1/\alpha}}\wedge1\right)\frac {dw}{|w|^{1+\alpha}}= 2\int_0^{\frac{ t^{1/\alpha}}a}\frac{aw}{ t^{1/\alpha}} \frac {dw}{w^{1+\alpha}} + 2\int_{\frac{ t^{1/\alpha}}a}^\infty  \frac {dw}{w^{1+\alpha}} = c \frac {a^{\alpha}}t.$$

Similar calculations show that 
$$\int_{\R}\left(\frac{w^2}{ t^{1/\alpha}}\wedge1\right)\frac {dw}{|w|^{1+\alpha}}= c t^{-1/2}.$$

\end{proof}
In the sequel we will need a version of the inverse map theorem for a Lipschitz function $f:\R^n \to \R^n, n\in \N$. The corresponding theorem is the main result  in \cite{C1976}, however it is not formulated in a suitable  way for our purpose. Below, closely following the arguments from  \cite{C1976},   we  provide a version we need.    

It is well known that $y$ almost surely the Jacobi matrix ${\mathcal J}_{f}(y)$ of $f$ exists. For any  $y_0\in \R^n$ we define (see Definition 1 in \cite{C1976}) the generalized Jacobian denoted $\partial f(y_0)$ as the convex hull of the set of matrices which can be obtained as limits of ${\mathcal J}_{f}(y_n)$, when $y_n\to y_0 $.   

We denote by $B(x,r)$ an open  ball of the
center $x\in \R^n$ and radius $r>0$. For any matrix $M$  we denote by $||M ||_{\infty} $ the maximum of its entries. 

\begin{lemma}\label{inverse} Let $f:\R^n \to \R^n$ be a Lipschitz map and $x\in \R^n$.  Suppose that for any $y\in \R^n$, the generalized Jacobian  $\partial f (y)$ consist of the matrices which can be represented as  $M(x)+ R $, where matrices $M(x),  R$ satisfy the following conditions: there are positive $\beta$ and $\eta$ such that   $||R||_{\infty}\le \eta |x-y|$
 and $|v  M(x)^T|\ge 2\beta$ for every $v\in \R^n, |v|=1$. Then  $f$ is injective on $B(x,\beta/(n\eta))$ and we have 
$B(f(x),  \beta^2/(2n \eta))\subset f(B(x,  \beta/(n\eta)))$.
\end{lemma}

\begin{proof}
Let $v$ be an arbitrary unit vector in $\R^n$. Let $M\in \partial f (y)$ and let $z= v  M(x)^T$.
Since $M^T= M(x)^T+ R^T$ the scalar product of $z$ and $w= v M^T= z+vR^T$ can be estimated as follows 
$$zw= z(z+vR^T)= |z|^2 + z(vR^T)\ge |z|^2- n\eta|z| |x-y|. $$
Next, taking $w^*= z/|z|$ we have for $|x-y|\le \beta/(n\eta)$, 
$$w^*(v M^T) \ge |z|- n\eta |x-y|\ge \beta. $$
Using this fact we can apply Lemma 3 and Lemma 4 of \cite{C1976} to claim that 
for every $y_1, y_2\in B(x,\beta/(n\eta))$  we have 
$$|f(y_1)-f(y_2)|\ge \beta|y_1-y_2|,$$
which shows that $f$ is injective in a ball $B(x,  \beta/(n\eta))$.
Next, by similar arguments, we show that 
$$|v M^T| \ge |v M(x)^T| - |v R^T|\ge 2\beta- n\eta |x-y|\ge \beta,\ |y-x|\le  \beta/(n\eta),$$
which proves that all matrices from the set $ \partial f (y)$ are of full rank if $|y-x|\le  \beta/(n\eta)$. 

Finally, we can apply Lemma 5 of \cite{C1976} to show that 
the  $f$ image of the ball $B(x,  \beta/(n\eta))$ contains the ball $B(f(x),  \beta^2/(2n\eta))$.
\end{proof}

\section{Construction and properties of the transition density of the solution of (\ref{main}) driven by the truncated process}

The approach in this section is based on Levi's method (cf. \cite{L1907,F1975,LSU1968}). This method was applied  in the framework of pseudodifferential operators by Kochubei \cite{K1989} to construct a fundamental solution to the related Cauchy problem as well as  transition density for the  corresponding Markow process. In recent years it was used in several papers to study transition densities of L{\'e}vy-type processes see e.g. \cite{CZ2016, KSV2018, CZ2017, J2017, GS2018, BSK2017, KK2017, KK2018, K2017}. Levi's method was also used to study gradient and Schr{\"o}dinger perturbations of fractional Laplacians see e.g. \cite{BJ2007,CKS2012,XZ2014}.

We first introduce the generator of the process $X_t$. We define $\calK f(x)$ by the following formula
$$
\calK f(x) =  \Aa \sum_{i = 1}^d  \lim_{\zeta \to 0^+}  \int_{|w| > \zeta} \left[f(x + a_i(x) w) - f(x)\right] \, \frac{dw}{|w|^{1 + \alpha}},
$$
for any Borel function $f: \R^d \to \R$ and any $x \in \R^d$ such that all the limits on the right hand side exist. Recall that $a_i(x) = (a_{1i}(x),\ldots, a_{di}(x))$. It is well known that $\calK f(x)$ is well defined for any $f \in C_b^2(\Rd)$ and any $x \in \R^d$. By (\cite[Proposition 4.1]{BC2006}) we know that if $f \in C_b^2(\Rd)$ then $f(X_t) - f(X_0) - \int_0^t \calK f(X_s) \, ds$ is a martingale. 

Let us fix $\eps \in (0,1]$ (it will be chosen later). 
Recall that for given $\eps$ the constant $\delta$ is chosen according to Lemma \ref{gtht}. For such fixed $\eps$, $\delta$ we abbreviate $\mu(x) = \mu^{(\delta)}(x)$, $\calG = \calG^{(\delta)}$, $g_t(x) = g_t^{(\delta)}(x)$, $h_t(x) = h_t^{(\eps)}(x)$.

We divide $\calK$ into two parts
\begin{equation}
\label{KLR}
\calK f(x) = \calL f(x) + \calR f(x),
\end{equation}
where
$$
\calL f(x) =  \sum_{i = 1}^d  \lim_{\zeta \to 0^+}  \int_{|w| > \zeta} \left[f(x + a_i(x) w) - f(x)\right] \, \mu(w) \, dw,
$$ 
for any Borel function $f: \R^d \to \R$ and any $x \in \R^d$ such that all the limits on the right hand side exists.
Our first aim will be to construct the heat kernel $u(t,x,y)$ corresponding to the operator $\calL$. This will be done by using the Levi's method. 

For each $z \in \R^d$ we introduce the ``freezing'' operator 
$$
\calL^z f(x) =  \sum_{i = 1}^d  \lim_{\zeta \to 0^+}  \int_{|w| > \zeta} \left[f(x + a_i(z) w) - f(x)\right] \, \mu(w) \, dw.
$$

Let $G_t(x) = g_t(x_1) \ldots g_t(x_d)$ and $H_t(x) = h_t(x_1) \ldots h_t(x_d)$ for $t > 0$ and $x = (x_1,\ldots,x_d) \in \R^d$. We also denote $B(x) = (b_{ij}(x)) = A^{-1}(x)$. Note that the coordinates of $B(x)$ satisfy conditions 
(\ref{bounded}) and  (\ref{Lipschitz}) with possibly different constants $\eta_1^*$ and $\eta_3^*$, but taking maximums we can assume that $\eta_1^*=\eta_1$ and $\eta_3^*=\eta_3$.

  For any $y \in \R^d$, $i = 1, \ldots, d$ we put
\begin{equation*}
b_i(y) = (b_{i1}(y),\ldots, b_{id}(y)).
\end{equation*}

We also denote $\|B\|_{\infty} = \max\{|b_{ij}|: \, i, j \in \{1,\ldots,d\}\}$.

For any $t > 0$, $x, y \in \R^d$ we define
\begin{eqnarray*}
p_y(t,x) &=& \det(B(y)) G_t(x(B(y))^T)\\
&=& \det(B(y)) g_t(b_1(y) x) \ldots g_t(b_d(y) x).
\end{eqnarray*}
It may be easily checked that for each fixed $y \in \R^d$ the function $p_y(t,x)$ is the heat kernel of $\calL^y$ that is 
$$
\frac{\partial}{\partial t} p_y(t,x) = \calL^y p_y(t,\cdot)(x), \quad t > 0, \, x \in \R^d,
$$
$$
\int_{\R^d} p_y(t,x) \, dx = 1, \quad t > 0.
$$
For any $t > 0$, $x, y \in \R^d$ we also define
\begin{eqnarray*}
r_y(t,x) &=&  H_t(x(B(y))^T)\\
&=&  h_t(b_1(y) x) \ldots h_t(b_d(y) x).
\end{eqnarray*}
For $x, y \in \R^d$, $t > 0$, let
$$
q_0(t,x,y) = \calL^{x}p_y(t,\cdot)(x-y) - \calL^{y} p_y(t,\cdot)(x-y),
$$
and for $n \in \N$ let
\begin{equation}
\label{defqn}
q_n(t,x,y) = \int_0^t \int_{\R^d} q_0(t-s,x,z)q_{n-1}(s,z,y) \, dz \, ds.
\end{equation}
For $x, y \in \R^d$, $t > 0$ we define
\begin{equation}
\nonumber
q(t,x,y) = \sum_{n = 0}^{\infty} q_n(t,x,y)
\end{equation}
and 
\begin{equation}
\label{defu}
u(t,x,y) = p_y(t,x-y) + \int_0^t \int_{\R^d} p_z(t-s,x-z)q(s,z,y) \, dz \, ds.
\end{equation}

In this section we will show that $q_n(t,x,y)$, $q(t,x,y)$, $u(t,x,y)$ are well defined and we will obtain estimates of these functions. First, we will get some simple properties of $p_y(t,x)$ and $r_y(t,x)$.

\begin{lemma}
\label{pyholder}
Choose $\gamma \in (0,1]$. For any $t \in (0,\tau]$, $x, x', y \in \R^d$ we have
$$
|p_y(t,x) - p_y(t,x')| \le c (1 \wedge (t^{-\gamma/\alpha} |x - x'|^{\gamma})) (r_y(t,x/2) + r_y(t,x'/2)).
$$
\end{lemma}
\begin{proof}
Of course, we may assume that $\gamma = 1$. We have
$$
p_y(t,x) - p_y(t,x') = \det(B(y)) \left(\prod_{i=1}^d  g_t\left(b_i(y)x\right) - \prod_{i=1}^d  g_t\left(b_i(y)x'\right)\right)
$$
If $|x - x'| \ge t^{1/\alpha}/\|B\|_{\infty}$ or $|x - x'| \ge \delta/\|B\|_{\infty}$ then the assertion clearly holds.
So, we may assume that $|x - x'| \le t^{1/\alpha}/\|B\|_{\infty}$ and $|x - x'| \le \delta/\|B\|_{\infty}$. Then the assertion follows easily from Lemma \ref{product_gt}.
\end{proof}

For  $x, y \in \R^d$ we have
$$
|B(y) x^T|^2 = | xB(y)^T|^2= (b_1(y) x)^2 + \ldots + (b_d(y) x)^2.
$$

\begin{lemma}
\label{b1b2}
For any $x, y \in \R^d$ and $i \in \{1,\ldots,d\}$ we have
\begin{equation*}
\max_{1\le i\le d}(b_i(y) x)^2 \ge \frac1{ \eta^2_1 d^3} |x|^2.
\end{equation*}
\end{lemma}
\begin{proof}
 Indeed,  for any $u,x$ we have   $|uA(y)^T|\le \eta_1 d |u|$. Setting $u= xB(y)^T$ we obtain that 
$$|xB(y)^T|\ge \frac1{\eta_1 d} |x|.$$
 Since $$       |xB(y)^T|^2= |b_1(y) x| ^2 +|b_2(y) x|^2+\dots+ |b_d(y) x|^2, $$ it follows that there is $1\le k\le d$ such that $|b_k(y) x|\ge \frac1{ \eta_1 d\sqrt{d}} |x|$. 
\end{proof}

\begin{corollary}
\label{estimate_pytx} Assume that $ \eps \le \frac1{ \eta_1 d\sqrt{d}}$. 
For any $t \in (0,\tau+1]$, $x, y \in \R^d$, we have
\begin{equation}
\label{htht1}
r_y(t,x-y)%\prod_{i = 1}^d h_t(b_i(y)(x-y))
\le c_1 t^{-d/\alpha} e^{-c |x - y|}.
\end{equation}
For any $t \in (0,\tau+1]$, $x, y \in \R^d$, $|x-y| \ge \eps\eta_1 d^{3/2}$, we have
\begin{equation}
\label{htht2}
r_y(t,x-y)%\prod_{i = 1}^d h_t(b_i(y)(x-y))
\le c_1 t e^{-c |x - y|}.
\end{equation}
\end{corollary}
\begin{proof}
For any $t \in (0,\tau+1]$, $z \in \R$ by definition of $h_t$ we have
\begin{equation}
\label{ht_estimate}
h_t(z) \le c t^{-1/\alpha} e^{-|z|}.
\end{equation}
Fix $x, y \in \R^d$, $t \in (0,\tau+1]$. By Lemma \ref{b1b2} there exists $i \in \{1,\ldots,d\}$ such that 
$|b_i(y)(x-y)| \ge  \frac1{ \eta_1 d\sqrt{d}} |x - y|$. Using this and (\ref{ht_estimate}) we get (\ref{htht1}).
For any $t \in (0,\tau+1]$, $z \in \R$, $|z| \ge \eps$ by definition of $h_t$ we have
\begin{equation*}
h_t(z) \le c t^{1 + (d-1)/\alpha} e^{-|z|}.
\end{equation*}
If $|x-y| \ge \eps\eta_1 d^{3/2}$ then $ \frac1{ \eta_1 d\sqrt{d}} |x - y|\ge \eps$ hence, by the same arguments as above, we get (\ref{htht2}).
\end{proof}

Using the definition of $p_y(t,x)$ and properties of $g_t(x)$ we obtain the following regularity properties of $p_y(t,x)$.
\begin{lemma}
\label{pycontinuity}
The function $(t,x,y) \to p_y(t,x)$ is continuous on $(0,\infty) \times \R^d \times \R^d$. The function $t \to p_y(t,x)$ is in $C^1((0,\infty))$ for each fixed $x, y \in \R^d$. The function $x \to p_y(t,x)$ is in $C^2(\R^d)$ for each fixed $t > 0$, $y \in \R^d$.
\end{lemma}

\begin{lemma}
\label{properties_pytx}
For any $y \in \R^d$ we have
$$
\left|\frac{\partial}{\partial x_i} p_y(t,x-y) \right| \le \frac{c}{t^{(d+1)/\alpha} (1+|x -y|)^{d+1}}, 
\quad i \in \{1,\ldots,d\}, \, t \in (0,\tau], \, x \in \R^d,
$$ 
$$
\left|\frac{\partial^2}{\partial x_i \partial x_j} p_y(t,x-y) \right| \le \frac{c}{t^{(d+2)/\alpha} (1+|x - y|)^{d+1}}, 
\quad i, j \in \{1,\ldots,d\}, \, t \in (0,\tau], \, x \in \R^d.
$$
\end{lemma}
\begin{proof}
The estimates follow from Lemma \ref{g_derivatives} and the same arguments as in the proof of (\ref{htht1}).
\end{proof}

%\begin{proposition} For any $x,y$ we have 
%\begin{equation*}
%|q_0(t,x,y)| \le    C\frac 1{t^{(d-1)/\alpha}} h_t(c_0|x-y|).
%\end{equation*}
%In particular, for  $|y-x|\ge \epsilon  $ we have\begin{equation*}
%|q_0(t,x,y)| \le    C e^{-\eta_5|y-x|},
%\end{equation*}
%where $C$ depends on $\eta_5$ and $\epsilon$.
%\end{proposition}
\begin{lemma}\label{intA}
Let  $ b^*_{i}(x,y), x, y \in \Rd; i=1, \dots, d$, be real functions such that there are positive  $\eta_4, \eta_5$ and 
\begin{equation}
\label{bounded1}
|b^*_{i}(x,y)| \le \eta_4, \quad x, y \in \R^d,
\end{equation}
\begin{equation}
\label{Lipschitz1}
|b^*_{i}(x,y) - b^*_{i}(\overline{x},\overline{y})| \le \eta_5(|x-\overline{x}|+ |y-\overline{y}|), \quad x, y, \overline{x}, \overline{y} \in \R^d.
\end{equation}
%\begin{equation}
%b^*_{1}(x,x) = 1,\ b^*_{i}(x,x) = 1,\quad i=2,\dots,d,  \quad x \in \R^d,
%\end{equation}
Let, for fixed $x\in \R^d$, $\Psi_x$ be a map $\R^{d+1}\mapsto  \R^{d+1}$ given by

$$\Psi_x(w, y)= (w, \xi_1, \dots, \xi_d)\in \R^{d+1}, \quad w \in \R, y\in \R^{d}, $$
where  $\xi_i=b_i(y)(x-y) +b^*_{i}(x,y)w$.

There is a positive  $\eps_0= \eps_0(\eta_1,\eta_3,\eta_4,\eta_5,d)\le \frac1{2\eta_5}$ such that 
the map $\Psi_x$  and its   Jacobian determinant   denoted by $J_{\Psi_x}(w,y)$ has   the property
\begin{eqnarray*}
|\Psi_x(w, y)|&\le& 1,\\
2 |det B(y)| \ge |J_{\Psi_x}(w,y)|&\ge& (1/2)|det B(y)|, \end{eqnarray*}
 for $ |x-y|\le \eps_0, |w|\le \eps_0$, $(w,y)$ almost surely. Moreover the map $\Psi_x$ is injective on the set $\{(w,y)\in \R^{d+1}; |x-y|\le \eps_0, |w|\le \eps_0\}$.

If, for fixed $y\in \R^d$, $\Phi_y$ be a map $\R^{d+1}\mapsto  \R^{d+1}$ given by 
$$\Phi_y(w, x)=\Psi_x(w, y), \quad w \in \R, x\in \R^{d}, $$ then 
the  Jacobian of $\Phi_y$ denoted by $J_{\Phi_y}(w,x)$ has the property
$$2 |det B(y)| \ge |J_{\Phi_y}(w,x)|\ge (1/2)|det B(y)|,$$
for $ |x-y|\le \eps_0, |w|\le \eps_0$, $(w,x)$ almost surely.  Moreover the map $\Phi_y$ is injective on the set $\{(w,x)\in \R^{d+1}; |x-y|\le \eps_0, |w|\le \eps_0\}$.
\end{lemma}
\begin{proof} In the proof we assume that constants $c$ may additionally depend on $\eta_4, \eta_5$. We  prove the statement for the map $\Psi_x$, only.
Since $|\Psi_x(w, y)|\le \sqrt{ d}(1+\eta_1 + \eta_4)(|w|+|x-y|)$ we have 

\begin{equation}\label{upper}|\Psi_x(w, y)|\le 1, \ \text{if} \ \ |w|+|x-y|\le \frac1 {\sqrt{ d}(1+\eta_1 + \eta_4)}.\end{equation}

Next, we observe that $(w,y)$ almost surely
$$\frac{\partial \xi_k} {\partial y_l}= -b_{kl}(y)+ (x-y) \cdot \frac{\partial b_k} {\partial y_l}
+w\frac{\partial b^*_k} {\partial y_l}, \quad 1\le l,k\le d.$$
Since $|(y-x) \cdot\frac{\partial b_k} {\partial y_l}+w\frac{\partial b^*_k} {\partial y_l}|\le (\eta_3+\eta_5) (|y-x|+|w|)$,   it follows that 
$$J_{\Psi_x}(w,y)= (-1)^d det B(y)+ R(x,y,w), \quad   |R(x,y,w)|\le c (|y-x|+|w|), $$
with $c=c(d,\eta_1,\eta_3,\eta_5)$.
Since $|\det B(y)|=\frac1 {|\det A(y)|}\ge \frac1 {d!\eta_1^d}  $, we have for sufficiently small $\eps_1= \eps_1(\eta_1, \eta_3,\eta_4,\eta_5,d)$, $(w,y)$ almost surely
\begin{equation}\label{Jacobian}J_{\Psi_x}(w,y)=  (-1)^d\kappa(x,y,w) det B(y),\ |y-x|\le \varepsilon_1, \ |w|\le \varepsilon_1,\end{equation}
where $\frac12\le \kappa\le 2.$ %Taking $\eps_0= \eps_1 \wedge \frac1 {2\sqrt{ d}(1+\eta_1 + \eta_4)}$ we complete the proof of the bounds of the Jacobian matrix. 

Let ${\mathcal J}_{\Psi_x}(w,y)$ be the Jacobi matrix for the map $\Psi_x$ which is defined $(w,y)$ almost surely. Let 
$\partial \Psi_x(w,y)$ denote the generalized Jacobian  of $\Psi_x$ at the point $(w,y)$. Then from the form of  ${\mathcal J}_{\Psi_x}$ it is clear that every  matrix ${\mathcal M}\in \partial \Psi_x(w,y)$ can be written as
$${\mathcal M}=  {\mathcal B}(x)+ {\mathcal R},  $$
where the coordinates ${\mathcal B}_{kl}(x), 0\le k,l\le d$ of the matrix $ {\mathcal B}(x)$   are
\begin{align*}{\mathcal B}_{kl}(x)&=  -b_{kl}(x), \,\,\, k,l\ge 1 , \\
{\mathcal B}_{00}(x)&=1;\ {\mathcal B}_{0l}(x)=0;\ {\mathcal B}_{l0}(x) = b_l^*(x,x),  \,\,\, 1\le l\le d,\end{align*}  while all the entries of ${\mathcal R}$ satisfy   $|{\mathcal R}_{kl}|\le c\sqrt{ |w|^2 +|x-y|^2}$ with $c= c(\eta_3, \eta_5)$.

%Then it is clear from (\ref{Jacobian}) that every such matrix ${\mathcal M}$ must have the following property
%$$|\det {\mathcal M}|\ge (1/2) det B(y),\ |y-x|\le \varepsilon_1, \ |w|\le \varepsilon_1.$$ Moreover,
Now,  for every $(u,z), u\in \R, z \in \R^{d}: |u|^2 +|z|^2=1 $ we have 
$$|(u,z){\mathcal B}(x)^T|\ge 2\beta>0,$$
with $\beta = \beta(d,\eta_1,\eta_4)$. 
Since $||{\mathcal R}||_\infty\le c\sqrt{ |w|^2 +|x-y|^2}$ we can apply 
Lemma \ref{inverse} with $n=d+1$ to show on the set  $\{(w,y);\sqrt{ |w|^2 +|x-y|^2}\le \beta/(c(d+1))\}$ the map $\Psi_x$ is injective. This fact, combined with (\ref{upper}) and (\ref{Jacobian}), completes the proof if we choose  
$\eps_0= \eps_1 \wedge \frac1 {2\sqrt{ d}(1+\eta_1 + \eta_4)}\wedge \frac{\beta}{2(d+1)c}$.
\end{proof}

\begin{remark} \label{intA1}Let for $x\in \R^d$, $\tilde{\Psi}_x$ be the map $\R^{d}\mapsto  \R^{d}$ given by 
$$\tilde{\Psi}_x( y)= (\xi_1, \dots, \xi_d)\in \R^{d}, \quad  y\in \R^{d}, $$
where  $\xi_i=b_i(y)(x-y)$. Then using the same arguments as in the above proof we can find $\eps_0$ such  that all the assertions of Lemma \ref{intA} are true and additionally
$$2 |det B(y)| \ge |J_{\tilde{\Psi}_x}(y)|\ge (1/2)|det B(y)|,$$
for $ |x-y|\le \eps_0$, $y$ almost surely. Moreover,  the map $\tilde{\Psi}_x$ is injective on $B(x, \eps_0)$. We can also find $\delta_1= \delta_1(\eta_1, \eta_3, \eta_4, \eta_5, d)>0$ and  $\delta_2= \delta_1(\eta_1, \eta_3, \eta_4, \eta_5, d)>0$ such that the $\tilde{\Psi}_x$ image of the ball $B(x, \delta_1)$  contains $B(0, \delta_2)$. To this end we apply the last assertion of Lemma \ref{inverse}.
\end{remark}

Let $b^*_i(x,y)$ be the functions introduced in Lemma \ref{intA}.
We will use the following abbreviations
\begin{eqnarray*}
&& z_i=B_i(x,y) = b_i(y)(x-y) = b_{i1}(y)  (x_1-y_1) + ... + b_{id}(y) (x_d-y_d),\\
%&& B_2(y) = b_{21}(y) y_1 + b_{22}(y) y_2,\\
&&b^*_i= b^*_i(x,y),\\
&&b^*_{i0}= b^*_i(x,x). %&& \tbs = b_{21}(y) a_{11}(0).
\end{eqnarray*}
Let for $l=1,...,d$, 
\begin{equation*}
\text{A}_l = {A}_l(x,y)= \int_{\R}  \prod_{i\ne l} g_t(z_i +b^*_{i}w) \left|g_t(z_{l} +b^*_{l}w) - g_t(z_{l} +b^*_{l0}w)\right|
%\end{equation*}
%\begin{equation*}
  \mu(w) \, dw.
\end{equation*}
 
\begin{corollary} Assume that $2\delta< \eps_0$, where $\eps_0$ is from Lemma \ref{intA}. With the assumptions of Lemma \ref{intA} we have for $t\le \tau$,
\label{intAl}
$$\int_{|y-x|\le \eps_0}\text{A}_ldy\le ct^{-1/2}, \quad x \in \R^d,$$
and 
$$\int_{|y-x|\le \eps_0}\text{A}_ldx\le ct^{-1/2}, \quad y \in \R^d,$$
where $c = c(\tau, \alpha, d, \eta_1, \eta_2, \eta_3, \eta_4, \eta_5, \eps, \delta)$.
\end{corollary}
\begin{proof} In the proof we assume that constants $c$ may additionally depend on $\eta_4, \eta_5$.  For $x,y\in\Rd$ we get  $|b^*_{l}-b^*_{l0}|\le \eta_5 |x-y|$. Hence,  from (\ref{gdelta3}), we have for $w\in \R$, $1\le l\le d$,
\begin{eqnarray*} 
&&\left|g_t(z_{l} +b^*_{l}w) - g_t(z_{l} +b^*_{l0}w)\right|\le   c\left(\frac{|b^*_{l}-b^*_{l0}|| w|}{ t^{1/\alpha}}\wedge1\right)
(h_t(z_{l} +b^*_{l}w)+h_t(z_{l} +b^*_{l0}w))\\
&&  \quad \quad \quad \quad \quad \quad \quad \quad \quad \quad \quad
\le c \left(\frac{|x-y|| w|}{ t^{1/\alpha}}\wedge1\right)
(g^{(\infty)}_t(z_{l} +b^*_{l}w)+g^{(\infty)}_t(z_{l} +b^*_{l0}w)).
\end{eqnarray*}
This implies that 
   $$\text{A}_{l}\le c(\text{A}^1_{l}+\text{A}^2_{l}),$$
where 
$$
\text{A}^1_{l}=   \int_{\R} \left( \prod_{i=1}^{d}g^{(\infty)}_t(z_i + b^*_{i}w)\right) \left(\frac{|x-y|| w|}{ t^{1/\alpha}}\wedge1\right)\mu(w) \, dw $$ and
$$
\text{A}^2_{l}=   \int_{\R} \left( \prod_{i=1}^{d} g^{(\infty)}_t(z_i + \hat{b}_{i}w)\right) \left(\frac{|x-y|| w|}{ t^{1/\alpha}}\wedge1\right)\mu(w) \, dw ,$$  
with $\hat{b}_i= b^*_{i}, i\ne l$ and  $\hat{b}_l= b^*_{l0}$. Note that the functions $\hat{b}_i= \hat{b}_i(x,y) $ have the same properties (\ref{bounded1},  \ref{Lipschitz1}) as $b^*_{i}$.
%Next let $\eps_0$ be as found in Lemma \ref{intA}.
To evaluate the integral $\int_{|x-y|\le \eps_0}\text{A}^1_{l}dy$ we introduce new variables in $\R^{d+1}$, given by  $(w,\xi)=\Psi_x(w, y)$, where $   \xi_i=z_i + b^*_iw, i=1,\dots,d$ ( or $\xi_i=z_i + \hat{b}_iw$ if $\text{A}^2_{l}$ is treated). Note that the vector $\xi= (\xi_1, \dots, \xi_d) $ can be written as
$$\xi= (x-y)B(y)^T+ w b^*,$$
where $b^* = (b^*_1, \dots, b^*_d)$, hence
$$(\xi-w b^*)A(y)^T= x-y.$$ From this we infer that
$$|w||x-y|\le c(|\xi|+|w|)|w|.$$

Let $Q_x=\{(w,y):|y-x|\le \eps_0, \ |w|\le \eps_0\}$.  Due to Lemma \ref{intA}, almost surely on $Q_x$, the absolute value of the  Jacobian determinant  of the map $\Psi_x$ is bounded from below and above by two positive constants and $\Psi_x$ is an injective transformation. Let $V_x= \Psi_x(Q_x)$. Observing that the support of the measure $\mu$ is contained in $[-\eps_0, \eps_0]$ and then  applying the above change of variables, we have 
\begin{eqnarray*}\int_{|y-x|\le \eps_0}\text{A}^1_ldy&\le& 
c \int_{|y-x|\le \eps_0} \int_{\R} \left(\prod_{i=1}^{d} g^{(\infty)}_t(\xi_i )\right)
 \left(\frac{(|\xi|+|w|)|w|}{ t^{1/\alpha}}\wedge1\right)\mu(w) \, dw \, d y \\ 
&\le& c \int_{|y-x|\le \eps_0} \int_{\R} \left(\prod_{i=1}^{d} g^{(\infty)}_t(\xi_i )\right)
 \left(\frac{(|\xi|+|w|)|w|}{ t^{1/\alpha}}\wedge1\right)\\
&& \times \,\, \mu(w)|J_{\Psi_x}(w,y)| \, dw \, d y \\ 
&=& c \int_{V_x} \prod_{i=1}^{d} g^{(\infty)}_t(\xi_i ) \left(\frac{(|\xi|+|w|)|w|}{ t^{1/\alpha}}\wedge1\right)\mu(w) \, dw \, d \xi\end{eqnarray*}
where the last equality follows from the general  change of variable formula  
for injective  Lipschitz maps (see e.g. \cite[Theorem 3]{H1993}). 
Since $|\xi|\le 1$ for $(w,\xi)\in V_x $,   we get 
$$\int_{|y-x|\le \eps_0}\text{A}^1_ldy\le c \int_{|\xi|\le1} \prod_{i=1}^{d} g^{(\infty)}_t(\xi_i )\int_{\R} \left(\frac{(|\xi|+|w|)|w|}{ t^{1/\alpha}}\wedge1\right)\mu(w) \, dw \, d \xi.   $$
Applying Lemma \ref{levy1} we have 
$$\int_{\R} \left(\frac{(|\xi|+|w|)|w|}{ t^{1/\alpha}}\wedge1\right)\mu(w) \, dw\le c \frac {t^{1/2}+|\xi|^{\alpha}}t.
$$
 Finally,

$$\int_{|y-x|\le \eps_0}\text{A}^1_{l}dy\le c  \int_{|\xi|\le1} \prod_{i=1}^{d} g^{(\infty)}_t(\xi_i )\frac {t^{1/2}+|\xi|^{\alpha}}t \, d \xi  \le c t^{-1/2}. $$
Similarly we obtain 
$$\int_{|y-x|\le \eps_0}\text{A}^2_{l}dy  \le c t^{-1/2}, $$
which completes the proof of the first bound. To estimate $\int_{|y-x|\le \eps_0}\text{A}_ldx$ we proceed exactly in the same way.
\end{proof}

For fixed $l \in \{1,\ldots,d\}$ let us consider a family  of functions $b_i^*(x,y)= b_{i}(y) a_{l}(x), i \in \{1,\dots,d\}$. They satisfy the conditions (\ref{bounded1}) and (\ref{Lipschitz1}) with $\eta_4= d\eta_1^2$ and 
$\eta_5= d\eta_1\eta_3$. Let  $\eps_0=\eps_0(\eta_1,\eta_3,\eta_4,\eta_5,d)$ be as found in Lemma \ref{intA} and Remark \ref{intA}. 
Finally we choose $\eps=\eps(\eta_1,\eta_3,d)= \frac{\eps_0}{4d^{3/2}(\eta_1\vee1)} $. From now on we keep 
$\eps_0, \eps$ fixed as above. Recall that if we fixed $\eps$ we fix $\delta$ according to Lemma \ref{gtht}.

\begin{proposition}
\label{integralq0}
%There exists $\sigma = \sigma(\tau, \alpha, \eta_1, \eta_2, \eta_3,\delta,\gamma) > 0$ such that 
 For any $x,y \in \R^d$, $t \in (0,\tau]$  we have 
%\begin{equation}\label{global}
%|q_0(t,x,y)| \le    c\frac 1{t^{1+(d-1)/\alpha}} h_t( (\eps/\eps_0)|x-y|).
%\end{equation}
\begin{equation}\label{global}
|q_0(t,x,y)| \le    c t^{-1-{(d-1)}/\alpha}h_t\left((\eps/\eps_0)|x-y|\mathbf{1}_{[\eps_0, \infty)}(|x-y|)\right).
\end{equation}
In particular for  $x,y \in \R^d$, $t \in (0,\tau]$, $|y-x|\ge \eps_0$ we have\begin{equation*}
|q_0(t,x,y)| \le    c e^{(-\eps/\eps_0)|x-y|}.
\end{equation*}
For any $t \in (0,\tau]$, $x \in \R^d$ we have
\begin{equation}\label{q0int}
\int_{\R^d} |q_0(t,x,y)| \, dy \le c t^{- 1/2}.
\end{equation}
For any $t \in (0,\tau]$, $y \in \R^d$ we have

\begin{equation}\label{q0intx}
\int_{\R^d} |q_0(t,y,x)| \, dx \le c t^{ - 1/2}.
\end{equation}

\end{proposition}
\begin{proof}
We have
\begin{eqnarray*}
 q_0(t,x,y) &=&  \sum_{i = 1}^d  \lim_{\zeta \to 0^+}  \int_{|w| > \zeta} \left[p_y(t,x-y + a_i(x) w) 
- p_y(t,x-y + a_i(y) w)\right] \, \mu(w) \, dw.
\end{eqnarray*}
For $i = 1, \ldots, d$ we put 
\begin{eqnarray} \label{term1}
   R_i &=& \lim_{\zeta \to 0^+}  \int_{|w| > \zeta} \left[p_y(t,x-y + a_i(x) w)- p_y(t,x-y + a_i(y) w)\right] \, \mu(w) \, dw.
\end{eqnarray}
We have $q_0(t,x,y) =  R_1 + \ldots + R_d$. It is clear that it is enough to handle $R_1$ alone.
Note that 
\begin{eqnarray}
\nonumber
&&   R_1 = \det(B(y)) \lim_{\zeta \to 0^+}  \int_{|w| > \zeta} \left[G_t\left((x-y + w e_1 (A(x))^T) (B(y))^T\right) \right.\\
\label{formulaR1}
&& \quad\quad\quad\quad\quad\quad\quad\quad\quad\quad  -\left. G_t\left((x-y + w e_1 (A(y))^T) (B(y))^T\right)\right] \, \mu(w) \, dw.
\end{eqnarray}

We will use the following abbreviations
\begin{eqnarray*}
&& z_i=B_i(x,y) = b_i(y)(x-y) = b_{i1}(y)  (x_1-y_1) + ... + b_{id}(y) (x_d-y_d),\\
%&& B_2(y) = b_{21}(y) y_1 + b_{22}(y) y_2,\\
&&k_i= \tbfs{i} = b_{i}(y) a_{1}(x),\\ %&& \tbs = b_{21}(y) a_{11}(0).
&&k_{i0}=\tilde{b}_{i1}(x,x). \end{eqnarray*}
Note that $k_{10}=1$ and $ k_{i0}=0,\ 2\le i\le d$.

We can rewrite (\ref{formulaR1}) as
\begin{eqnarray*}
   R_1&=& \det(B(y)) \lim_{\zeta \to 0^+}  \int_{|w| > \zeta} \left[\prod_{i=1}^{d} g_t(z_i + k_i w)-\prod_{i=1}^{d} g_t(z_i + k_{i0} w)\right]\, \mu(w) \, dw.
\end{eqnarray*}
%
%We note that $\tilde{b}_{11}(x,x)=1$ and $\tilde{b}_{i1}(x,x)=0$ for $i\neq 1$ and therefore %$|\tbfs{1}|\to 1 $ and $\tbfs{i} \to 0$ for
%$i\neq 1$ as $|x-y|\to 0^+.$ 
Hence,
$$
R_1 = \det(B(y))(\text{I}_1 +...+\text{I}_{d})= \det(B(y))(\text{I}^\prime_1 +...+\text{I}^\prime_{d}),
$$
where
\begin{eqnarray*}
\text{I}_l &=& \int_{\R} \left( \prod_{i=1}^{l-1} g_t(z_i + k_i w)\right) \left[g_t(z_{l} + k_{l} w) - g_t(z_{l}+k_{l0}w)\right]
%\end{equation*}
%\begin{equation*}
 \left( \prod_{i=l+1}^d  g_t(z_i)\right) \mu(w) \, dw,\\
\text{I}^\prime_l &=& \int_{\R} \left( \prod_{i=1}^{l-1} g_t(z_i + k_{i0} w)\right) \left[g_t(z_{l} + k_{l} w) - g_t(z_{l}+k_{l0}w)\right]\\
%\end{equation*}
%\begin{equation*}
&& \times \left( \prod_{i=l+1}^d  g_t(z_i + k_i w)\right) \mu(w) \, dw
\end{eqnarray*}
for $l=1,...,d$ (with convention that $ \prod_{i=1}^{0} =1=\prod_{i=d+1}^d$).

 %$|y|\ge \epsilon$. This argument does not hinge on the assumption that $a_{i1}(0)=0\ i\ge 2$. 

We start with the proof of the bound of $q_0$.
%We may assume that $|z_1|\le \dots\le |z_d|$. 
%
We observe that $|k_l|, |k_{l0}| \le d \eta_1^2 \le \eps/ (4\delta)$, $l \in \{1,\ldots,d\}$.   
 By Lemmas \ref{gtht} and  \ref{integralgt}, we obtain for $l\le d-1$,
\begin{eqnarray*}
|\text{I}_{l}| &\le& g_t(z_d) g^{d-2}_t(0) \int_{\R}  \left|g_t(z_{l} + k_{l} w) - g_t(z_{l}+k_{l0} w)\right| \mu(w)\, dw\\&\le& c (|k_{l}|+|k_{l0}|)^\alpha g_t(z_d) g^{d-2}_t(0) \frac  {h_t(0)}t\\&\le& c t^{-1-(d-1)/\alpha}h_t(z_d).
%\\&\le& c t^{-1-(d-1)/\alpha}h_t(c|x-y|)
\end{eqnarray*}
The same argument leads to 
\begin{eqnarray*}
|\text{I}_{d}| &\le & g^{d-1}_t(0) \int_{\R}  \left|g_t(z_{d} + k_{d} w) - g_t(z_{d})\right| \mu(w)\, dw
\\&\le&
c g^{d-1}_t(0)  |k_{d}|^\alpha  \frac  {h_t(z_d/2)}t
\\&\le& c t^{-1-{(d-1)}/\alpha}h_t(z_d/2).
\end{eqnarray*}
The above inequalities yield 
\begin{equation*}
|R_1| \le c t^{-1-{(d-1)}/\alpha}h_t(z_d/2).
\end{equation*}
Since $R_1$ is invariant with respect to permutations of $z_2, \dots, z_d$ we infer that 

\begin{equation}\label{R1est}
|R_1| \le c t^{-1-{(d-1)}/\alpha}\inf_{2\le i\le d}h_t(z_i/2)= c t^{-1-{(d-1)}/\alpha}h_t\left(\max_{2\le i\le d}|z_i|/2\right) .
\end{equation}
On the other hand,  again by Lemma \ref{integralgt},
\begin{eqnarray*}
|\text{I}^{\prime}_{1}| &\le&  g^{d-1}_t(0) \int_{\R}  \left|g_t(z_{1} + k_{1} w) - g_t(z_{1}+ w)\right| \mu(w)\, dw\\&\le& c (|k_{1}|+1)^\alpha g^{d-1}_t(0) \frac  {h_t(z_1/2)}t\\&\le& c t^{-1-(d-1)/\alpha}h_t(z_1/2).
%\\&\le& c t^{-1-(d-1)/\alpha}h_t(c|x-y|)
\end{eqnarray*}
For $l\ge 2$, a similar argument leads to 
\begin{eqnarray*}
|\text{I}^{\prime}_{l}| &\le& \sup_{|w|\le 2\delta} g_t(z_1+w) g^{d-2}_t(0) \int_{\R}  \left|g_t(z_{l} + k_{l} w) - g_t(z_{l})\right| \mu(w)\, dw\\&\le& c \sup_{|w|\le 2\delta} g_t(z_1+w)  g^{d-2}_t(0) \frac  {h_t(0)}t\\&\le& c t^{-1-(d-1)/\alpha}\sup_{|w|\le 2\delta}g_t(z_1+w).
%\\&\le& c t^{-1-(d-1)/\alpha}h_t(c|x-y|)
\end{eqnarray*}
Observing that $|z_1+w|\ge |z_1|/2$ for $|w|\le 2\delta\le 4\delta\le |z_1|$ we conclude that
\begin{equation}\label{R1est2}
|R_1| \le c t^{-1-{(d-1)}/\alpha}h_t\left(\frac {|z_1|}2\mathbf{1}_{[ 4\delta, \infty)}(|z_1|)\right).
\end{equation}
Combining  (\ref{R1est}) and (\ref{R1est2}) we arrive at
\begin{equation*}
|R_1| \le c t^{-1-{(d-1)}/\alpha}h_t\left(\frac{\max_{1\le i\le d}|z_i|}2\mathbf{1}_{[ 4\delta, \infty)}(\max_{1\le i\le d}|z_i|)\right).
\end{equation*}
By Lemma \ref{b1b2}, $\max_{1\le i\le d}|z_i|\ge \frac1{\eta_1d^{3/2}}|x-y|$, and by the choice of $\eps, \eps_0, \delta$ we have $\frac{1}{2\eta_1d^{3/2}} \ge \eps/\eps_0$ and $4\delta\eta_1d^{3/2}\le \eps_0$,  hence
\begin{eqnarray*}
|R_1| &\le& c t^{-1-{(d-1)}/\alpha}h_t\left(\frac{|x-y|}{2\eta_1d^{3/2}}\mathbf{1}_{[ 4\delta\eta_1d^{3/2}, \infty)}(|x-y|)\right)\\
&\le& c t^{-1-{(d-1)}/\alpha}h_t\left((\eps/\eps_0)|x-y|\mathbf{1}_{[ \eps_0, \infty)} (|x-y|)\right).
\end{eqnarray*}
Finally, for  $|y-x|\ge \eps_0  $ we have 
\begin{equation*}
|R_1| \le  ce^{-(\eps/\eps_0)|x-y|}.
\end{equation*}

Next, we prove the bound of the integral (\ref{q0int}). Let $x \in \R^d$ be fixed. 
Applying Lemma \ref{intAl} with $b^*_i= k_i \mathbf{1}_{\{i\le l\}}$ we have that 

$$\int_{|y-x|\le \eps_0}|\text{I}_{l}|dy \le c t^{-1/2}. $$

Hence $$\int_{|x-y|\le \eps_0} |q_0(t,x,y)| \, dy \le c t^{- 1/2}.$$ 

For $|y-x|\ge \eps_0$ we have $|q_0(t,x,y)| \le    ce^{-(\eps/\eps_0)|x-y|}$ which implies that 
$$\int_{|x-y|\ge \eps_0} |q_0(t,x,y)| \, dy \le c. $$ This completes the proof of (\ref{q0int}). The estimate (\ref{q0intx}) is proved exactly in the same way.
\end{proof}

Using similar arguments as in the proof of Proposition \ref{integralq0} we obtain the following result.
%\begin{lemma}
\begin{proposition}
\label{pyintegral}
For any $t \in (0,\tau]$, $x \in \R^d$ we have
\begin{equation} \label{pyintegral1}
\int_{\R^d} p_y(t,x-y) \, dy \le c,
\end{equation}
\begin{equation} \label{ryintegral1}
\int_{\R^d} r_y(t,(x-y)/2) \, dy \le c.
\end{equation}

%For any  $\delta_1 > 0$ there exists $\tau_1 \in (0,\tau]$ such that 
%$$
 %\sup_{x \in \R^d} \int_{B^c(x,\delta_1)} p_y(t,x-y) \, dy \le \frac t{\delta_1^{1+2\alpha}}.
%$$
For any  $\delta_1 > 0$, 
\begin{equation}
\label{supsup}
\lim_{t\to 0^+} \sup_{x \in \R^d} \int_{B^c(x,\delta_1)} p_y(t,x-y) \, dy =0.
\end{equation}
%\end{lemma}

%\begin{proposition}
%\label{limit_zero}
We have
\begin{equation}
\label{limp_y}
\lim_{t \to 0^+} \int_{\R^d} p_y(t,x-y) \, dy = 1,
\end{equation}
uniformly with respect to $x \in \R^d$.
\end{proposition}
\begin{proof}
For fixed $x \in \R^d$ we introduce new variables $u= \tilde{\Psi}_x(y)$ given by 

$$u=  (x-y)B(y)^T.  $$
Note that 
\begin{equation} \label{normy}\frac1 {d\eta_1}|x-y|\le |u|=  |(x-y)B(y)^T|\le d\eta_1|x-y|.  \end{equation}
For $r>0$, let $V_x(r)$ be the $\tilde{\Psi}_x$ image of the ball  $B(x,r)$. 
 By Remark \ref{intA1} we have almost surely
$$|J_{\tilde{\Psi}_x}(y)|\ge  (1/2) |\det B(y)|\ge c, \quad   |y-x|\le \eps_0, $$
and  $\tilde{\Psi}_x$ is an injective map on $B(x,\eps_0)$.
%$$|J_{\tilde{\Psi}^{-1}_x}(u)|= \frac1 {|J_{\tilde{\Psi}_x}(\tilde{\Psi}_x^{-1}(u))|} \le  \frac2 {|\det B(\tilde{\Psi}_x^{-1}(u))|}, \quad   u\in  V_x(\eps_0). $$
%
Hence, for $0<\delta_1< \eps_0$, by the change of variables formula (see e.g. \cite[Theorem 3]{H1993}), and then by (\ref{normy}) we obtain
\begin{eqnarray*}
 \int_{\delta_1\le |x-y|\le \eps_0} r_y(t,(x-y)/2) \, dy  &\le& c \int_{\delta_1\le |x-y|\le \eps_0 }H_t(u/2) |J_{\tilde{\Psi}_x}(y)|\, dy\\&=&
c\int_{ V_x(\eps_0)\setminus V_x(\delta_1) }H_t(u/2)\, du \\
&\le&
c\int_{|u|\ge \frac {\delta_1} {2d\eta_1}  }H_t(u/2)\, du = I(t,\delta_1) . 
\end{eqnarray*}
It is clear that  $$\lim_{\delta_1\to 0^+}I(t,\delta_1)\le c,\ t\le \tau.$$ 
%$$lim_{t\to 0^+}I(t,\delta_1)=0$$ 
%
%c\int_{ |u|\ge \frac {\delta_1} {d\eta_1} }H(t,u)\, du = I(t)\to 0, \quad t\to 0^+ . 
%
If $|x-y|\ge \eps_0$ then $|x-y|/2\ge \eps\eta_1d^{3/2} $, hence, by (\ref{htht2}),  we obtain
$$
 \int_{|x-y|\ge \eps_0} r_y(t,(x-y)/2) \, dy \le c_1t \int_{|x-y|\ge \eps_0} e^{-c|x-y|} \, dy = c_2 t . 
$$
% \int_{V_x(\eps_0)\setminus V_x(\delta_1) }
The last two inequalities  prove that
$$
 \sup_{x\in \Rd}\sup_{t\le \tau}\int_{\Rd} r_y(t,(x-y)/2) \, dy <\infty. 
$$
Noting that $\lim_{t\to 0^+}I(t,\delta_1)=0$ we obtain
$$
 \lim_{t \to 0^+}\sup_{x\in \Rd}\int_{|x-y|\ge \delta_1} r_y(t,(x-y)/2) \, dy =0 . 
$$
Since  $p_y(t,x-y)\le cr_y(t,(x-y)/2) $ for $t\le \tau,\ x, y \in \Rd$,
the proof of (\ref{pyintegral1}, \ref{ryintegral1}, \ref{supsup}) is completed. 

Note that the coordinates of the matrix $B(y)$ have partial derivatives $y$ almost surely,  bounded uniformly. We can calculate  the absolute value of  Jacobian determinant $J_{\tilde{\Psi}_x}(y)$, $y$ almost surely, as
\begin{equation}
\label{J}
\left|J_{\tilde{\Psi}_x}(y)\right|=  \det B(y)+ R(x,y), \quad    |R(x,y)|\le c |y-x|. \end{equation}
Next, 
\begin{eqnarray*}
  \int_{|x-y|\le \delta_1} p_y(t,x-y) \, dy &=&   \int_{|x-y|\le \delta_1}G_t(u) {\det B(y)} dy \\ 
 &=& \int_{|x-y|\le \delta_1 }G_t(u) \left|J_{\tilde{\Psi}_x}(y)\right| dy- \int_{|x-y|\le \delta_1}G_t(u) R(x,y)dy\\
 %&=& \int_{V_x(\delta_1) }G(t,u) du+ \int_{V_x(\delta_1) }G(t,u) \left({\det B(\Psi_x^{-1}(u))} {J_{\Psi_x^{-1}}(u)}-1\right)du\\
&=&I_1+I_2. 
\end{eqnarray*}
Applying (\ref{J}), (\ref{normy}) and the change of variable formula we obtain 
%
%$$\left|\int_{V_x }G(t,u) \left({\det B(\Psi_x^{-1}(u))} {J_{\Psi_x^{-1}}(u)}-1\right)du\right|\le \int_{V_x }|u|G(t,u) du\le \int_{|u|\le 1 }|u|G(t,u)du \to 0,\ \text{if} t\to 0^+. $$
%
\begin{eqnarray*}|I_2|&\le& c\int_{|x-y|\le \delta_1 }|x-y|G_t(u) dy\\&\le& c\int_{|x-y|\le \delta_1}|u|G_t(u) \left|J_{\tilde{\Psi}_x}(y)\right|dy \\
&=&c\int_{V_x(\delta_1) }|u|G_t(u) du\\
&\le& c\int_{|u|\le d\eta_1 \delta_1 }|u|G_t(u)du \to 0,\ \text{if}\ t\to 0^+. \end{eqnarray*}
%$$|I_2|\le c\int_{V_x(\delta_1) }|u|G(t,u) du\le \int_{|u|\le 1 }|u|G(t,u)du \to 0,\ \text{if}\ t\to 0^+. $$
%
Now we can pick, independenly of $x$, positive $\delta_1 $ and $\delta_2$ such that $ B(0, \delta_2)\subset V_x(\delta_1)$ (see Remark \ref{intA1}). 
Applying again the change of variable formula we obtain  
\begin{eqnarray*}I_1= \int_{V_x(\delta_1) }G_t(u) du&\ge&  \int_{|u|\le \delta_2 }G_t(u)  du \to 1,\ \text{if}\ t\to 0^+. 
\end{eqnarray*}
This completes the proof that  uniformly with respect to $x$,
$$\lim_{t\to 0^+}\int_{|x-y|\le \delta_1} p_y(t,x-y) dy=1,$$
which combined with (\ref{supsup}) proves (\ref{limp_y}).
\end{proof}

In the sequel we will use the following standard estimate. For any $\gamma \in (0,1]$, $\theta_0 > 0$ there exists $c = c(\gamma,\theta_0)$ such that for any $\theta \ge \theta_0$, $t > 0$ we have
\begin{equation}
\label{gammatheta}
\int_0^t (t - s)^{\gamma-1} s^{\theta - 1} \, ds \le \frac{c}{\theta^{\gamma}} t^{(\gamma-1)+(\theta-1)+1}.
\end{equation}

\begin{lemma}
\label{integralqn}
For any $t > 0$, $x \in \R^d$ and $n \in \N$ the kernel $q_n(t,x,y)$ is well defined. For any $t \in (0,\tau]$, $x \in \R^d$ and $n \in \N$ we have
\begin{equation}
\label{integralqn1}
\int_{\R^d} |q_n(t,x,y)| \, dy \le \frac{c_1^{n+1} t^{(n+1)/2 - 1}}{(n!)^{1/2}},
\end{equation}
\begin{equation}
\label{integralqn2}
\int_{\R^d} |q_n(t,y,x)| \, dy \le \frac{c_1^{n+1} t^{(n+1)/2 - 1}}{(n!)^{1/2}}.
\end{equation}
For any $t \in (0,\tau]$, $x, y \in \R^d$ and $n \in \N$ we have
\begin{equation}
\label{qnestimate1}
|q_n(t,x,y)|  \le c_1 \frac{c_2^n t^{n/2 - 1}}{(n!)^{1/2} t^{d/\alpha}}.
\end{equation}
For any $t \in (0,\tau]$, $x, y \in \R^d$ and $n \in \N$, $|x - y| \ge n + 1$ we have
\begin{equation}
\label{qnestimate2}
|q_n(t,x,y)|  \le c_1 \frac{c_2^n t^{n/2}}{(n!)^{1/2}} e^{-\frac{\lambda |x - y|}{n+1}},
\end{equation}
where $\lambda=\eps/\eps_0$.
\end{lemma}
%Here $c_2 \ge c_1$ and $c_4 \ge c_1$.
\begin{proof} By  Proposition \ref{integralq0} there is a constant $c^*$ such that 
for any $x,y \in \R^d$, $t \in (0,\tau]$  we have 
\begin{equation}\label{global1}
|q_0(t,x,y)| \le    c^*\frac 1{t^{1+d/\alpha}},
\end{equation}
\begin{equation}\label{global2}
|q_0(t,x,y)| \le    c^* e^{-\lambda|x-y|}, |x-y|\ge 1 .
\end{equation}
\begin{equation}\label{q0int1}
\int_{\R^d} |q_0(t,x,u)| \, du \le c^* t^{- 1/2},
\end{equation}
\begin{equation}\label{q0intx1}
\int_{\R^d} |q_0(t,u,x)| \, du \le c^* t^{ - 1/2}.
\end{equation}
It follows from (\ref{gammatheta}) there is  $p \ge 1$  such that for $n \in \N$,
\begin{equation*}
\int_0^t (t - s)^{-1/2} s^{n/2-1/2} \, ds \le \frac{p}{(n+1)^{1/2}} t^{n/2}, 
\end{equation*}
\begin{equation*}
\int_{t/2}^t (t - s)^{-1/2} s^{n/2 - 1} \, ds \le \frac{p}{(n+1)^{1/2}} t^{(n+1)/2 -1},
\end{equation*}
\begin{equation*}
\int_0^t (t - s)^{-1/2} s^{n/2} \, ds \le \frac{p}{(n+1)^{1/2}} t^{(n+1)/2}.
\end{equation*}
We define $c_1= pc^*\ge c^*$ and $c_2= 2^{d/\alpha + 1} c_1(2+p)> c_1 $.

We will prove (\ref{integralqn1}), (\ref{integralqn2}), (\ref{qnestimate1}) simultaneously by induction. They are true for $n = 0$ by  (\ref{global1}, \ref{q0int1}, \ref{q0intx1}) and the choice of $c_1$. Assume that (\ref{integralqn1}), (\ref{integralqn2}), (\ref{qnestimate1}) are true for $n \in \N$, we will show them for $n + 1$. By the definition of $q_{n}(t,x,y)$ and the induction hypothesis we obtain
\begin{eqnarray*}
|q_{n+1}(t,x,y)| &\le& 
c_1 \frac{2^{d/\alpha + 1}}{t^{d/\alpha + 1}} \int_0^{t/2} \int_{\R^d} |q_n(s,z,y)| \, dz \, ds\\
&& + c_1 \frac{c_2^n 2^{d/\alpha + 1}}{(n!)^{1/2} t^{d/\alpha}} \int_{t/2}^t \int_{\R^d} |q_0(t - s,x,z)| \, dz s^{n/2 - 1} \, ds\\
&&  \le c_1 \frac{c_1^{n+1} 2^{d/\alpha + 1}}{(n!)^{1/2} t^{d/\alpha + 1}} \int_0^{t/2} s^{(n+1)/2 - 1} \, ds\\
&& + c_1 \frac{c_2^n 2^{d/\alpha + 1} c_1}{(n!)^{1/2} t^{d/\alpha}} \int_{t/2}^t (t-s)^{-1/2} s^{n/2 - 1} \, ds\\
&& \le c_1 \frac{c_2^n t^{(n+1)/2}}{((n+1)!)^{1/2} t^{d/\alpha + 1}} \left(2 c_1 2^{d/\alpha + 1}+ c_1 2^{d/\alpha + 1} 
p \right)\\
&& = c_1 \frac{c_2^{n+1} t^{(n+1)/2}}{((n+1)!)^{1/2} t^{d/\alpha + 1}}. 
\end{eqnarray*}
Hence we get (\ref{qnestimate1}) for $n+1$. In particular this gives that the kernel $q_{n+1}(t,x,y)$ is well defined.

By the definition of $q_{n}(t,x,y)$, (\ref{q0int1}) and the induction hypothesis we obtain 
\begin{eqnarray*}\int_{\R^d} |q_{n+1}(t,x,y)| \, dy &\le&   \int_0^{t} \int_{\R^d}\int_{\R^d} |q_0(t - s,x,z)||q_n(s,z,y)| \, dz \, dy \, ds \\ &\le&
 c^* \frac{c_1^{n+1}}{(n!)^{1/2}} \int_0^{t}(t-s)^{-1/2}  s^{(n+1)/2 - 1}\, ds 
\\ &\le&
 c^* \frac{c_1^{n+1}}{(n!)^{1/2}} \frac{p}{(n+1)^{1/2}} t^{n/2}\\ &=&
  \frac{c_1^{n+2}}{((n+1)!)^{1/2}}  t^{n/2},
\end{eqnarray*}            
which proves (\ref{integralqn1}) for $n+1$. Similarly we get (\ref{integralqn2}).

Now we will show (\ref{qnestimate2}). For $n = 0$ this follows from  (\ref{global2}). Assume that (\ref{qnestimate2}) is true for $n \in \N$, we will show it for $n+1$. 

Using our induction hypothesis, (\ref{integralqn1}) and (\ref{integralqn2}) we get for $|x - y| \ge n+2$
\begin{eqnarray*}
 |q_{n+1}(t,x,y)| &=&
\left|\int_0^t \int_{|x-z| \ge \frac{|x-y|}{n+2}} q_0(t-s,x,z) q_n(s,z,y) \, dz \, ds\right. \\
&& + \left. \int_0^t \int_{|x-z| \le \frac{|x-y|}{n+2}} q_0(t-s,x,z) q_n(s,z,y) \, dz \, ds\right| \\
&& \le c_1 e^{-\frac{\lambda |x - y|}{n+2}} \int_0^t \int_{\R^d} |q_n(s,z,y)| \, dz \, ds\\
&& + c_1 \frac{c_2^n}{(n!)^{1/2}} e^{-\frac{\lambda |x - y|}{n+2}}
\int_0^t \int_{\R^d} |q_0(t-s,x,z)| \, dz s^{n/2} \, ds\\
&& \le c_1 \frac{2 c_1^{n+1} t^{(n + 1)/2}}{((n + 1)!)^{1/2}} e^{-\frac{\lambda |x - y|}{n+2}} 
+
c_1 \frac{c_2^n c^* t^{(n+1)/2} p}{((n+1)!)^{1/2}} e^{-\frac{\lambda |x - y|}{n+2}}\\
&& = (2c_1  c_1^{n+1} +c_2^n c_1^2)\frac{ t^{(n + 1)/2}}{((n + 1)!)^{1/2}} e^{-\frac{\lambda |x - y|}{n+2}} 
,
\end{eqnarray*}
which proves  (\ref{qnestimate2}) for $n+1$ since by the choice of constants $2c_1  c_1^{n+1} +c_2^n c_1^2\le c_1c_2^{n+1}$.

\end{proof}

By standard estimates one easily gets
\begin{equation}
\label{sumsigma}
\sum_{n=k}^{\infty} \frac{C^n}{(n!)^{1/2}} \le \frac{C^k}{(k!)^{1/2}}\sum_{n=k}^{\infty} \frac{C^{n-k}}{((n-k)!)^{1/2}} \le C_1 e^{-k}, \quad k \in \N,
\end{equation}
where $C_1$ depends on $C$.

\begin{proposition}
\label{qestimate}
For any $t \in (0,\infty)$, $x, y \in \R^d$ the kernel $q(t,x,y)$ is well defined. For any $t \in (0,\tau]$, $x,y \in \R^d$ we have
$$
|q(t,x,y)| \le \frac{c}{t^{d/\alpha+1}} e^{-c_3 \sqrt{|x-y|}} \le \frac{c}{t^{d/\alpha + 1} (1 + |x - y|)^{d+1}}.
$$
There exists $a > 0$ ($a$ depends on $\tau, \alpha, d, \eta_1, \eta_2, \eta_3$) such that for any $t \in (0,\tau]$, $x,y \in \R^d$, $|x-y| \ge a$ we have
$$
|q(t,x,y)| \le c e^{-c_3 \sqrt{|x-y|}}.
$$
For any $t \in (0,\tau]$ and $x \in \R^d$ we have
\begin{equation}
\label{integralq1}
\int_{\R^d} |q(t,x,y)| \, dy \le c t^{-1/2},
\end{equation}
\begin{equation}
\label{integralq2}
\int_{\R^d} |q(t,y,x)| \, dy \le c t^{-1/2}.
\end{equation}
\end{proposition}
\begin{proof}
By (\ref{qnestimate1}) we clearly get $\sum_{n = 0}^{\infty} \left| q_n(t,x,y) \right| \le c t^{-d/\alpha -1}$. This gives that $q(t,x,y)$ is well defined and we have $|q(t,x,y)| \le c t^{-d/\alpha -1}$. 

For $|x - y| \ge 1$ by (\ref{qnestimate1}), (\ref{qnestimate2}) and (\ref{sumsigma}) we get
\begin{eqnarray*}
 |q(t,x,y)| &=&
\left|\sum_{n = 0}^{\left[\sqrt{|x-y|}-1\right]} q_n(t,x,y) + \sum_{n = \left[\sqrt{|x-y|}\right]}^{\infty} q_n(t,x,y)\right| \\
&& \le
c_1 \sum_{n = 0}^{\left[\sqrt{|x-y|}-1\right]} \frac{c_2^n \tau^{n/2}}{(n!)^{1/2}} e^{-\lambda \sqrt{|x-y|}} +
c_1 \sum_{n = \left[\sqrt{|x-y|}\right]}^{\infty} \frac{c_2^n \tau^{n/2}}{(n!)^{1/2} t^{d/\alpha + 1}}\\
&& \le \frac{c}{t^{d/\alpha+1}} e^{-c_3 \sqrt{|x-y|}},
\end{eqnarray*}
where $[z]$ denotes the integer part of $[z]$. Take the smallest $n_0 \in \N$ such that $n_0/2 - 1 \ge d/\alpha$ and $a = n_0^2$. For $\sqrt{|x-y|} \ge \sqrt{a} = n_0$ we get
\begin{eqnarray*}
|q(t,x,y)| &\le&
c_1 \sum_{n = 0}^{\left[\sqrt{|x-y|}-1\right]} \frac{c_2^n \tau^{n/2}}{(n!)^{1/2}} e^{-\lambda \sqrt{|x-y|}} +
c_1 \sum_{n = \left[\sqrt{|x-y|}\right]}^{\infty} \frac{c_2^n t^{n/2}}{(n!)^{1/2} t^{d/\alpha + 1}}\\
&\le& c e^{-c_3 \sqrt{|x-y|}}.
\end{eqnarray*}

(\ref{integralq1}) and (\ref{integralq2}) follows easily from (\ref{integralqn1}) and (\ref{integralqn2}).
\end{proof}

By (\ref{defu}), Corollary \ref{estimate_pytx}, Proposition \ref{pyintegral} and Proposition \ref{qestimate} we immediately obtain the following result.
\begin{corollary}
\label{uintegral}
For any $t \in (0,\infty)$, $x, y \in \R^d$ the kernel $u(t,x,y)$ is well defined. For any $t \in (0,\tau]$, $x,y \in \R^d$ we have
$$
|u(t,x,y)| \le \frac{c}{t^{d/\alpha}} e^{-c_1 \sqrt{|x-y|}} \le \frac{c}{t^{d/\alpha} (1 + |x - y|)^{d+1}}.
$$
There exists $a > \eps > 0$ ($a$ depends on $\tau, \alpha, d, \eta_1, \eta_2, \eta_3$) such that for any $t \in (0,\tau]$, $x,y \in \R^d$, $|x-y| \ge a$ we have
$$
|u(t,x,y)| \le c e^{-c_2 \sqrt{|x-y|}}.
$$
For any $t \in (0,\tau]$ and $x \in \R^d$ we have
\begin{equation}
\label{integralu1}
\int_{\R^d} |u(t,x,y)| \, dy \le c,
\end{equation}
\begin{equation}
\label{integralu2}
\int_{\R^d} |u(t,y,x)| \, dy \le c.
\end{equation}
\end{corollary}

\begin{proof} By Corollary \ref{estimate_pytx}, Proposition \ref{pyintegral}, (\ref{defu}), we only need to prove the corresponding bounds for $$I(t,x,y)=\int_0^t \int_{\R^d} p_z(t-s,x-z)|q(s,z,y)| \, dz \, ds.$$
For $0<s<t/2$ we have 
$$p_z(t-s,x-z)\le \frac c{t^{d/\alpha}},$$
and, by Proposition \ref{qestimate}, for $t/2<s<t$,
$$
|q(s,z,y)| \le \frac{c}{t^{d/\alpha+1}}.%e^{-c_8 \sqrt{|x-y|}} \le \frac{c}{t^{d/\alpha + 1} (1 + |x - y|)^{d+1}}.
$$
Hence, 
\begin{eqnarray}
\nonumber
 I(t,x,y)
&=&\int_0^{t/2} \int_{\R^d}p_z(t-s,x-z)|q(s,z,y)| \, dz \, ds\\
&&   + \int_{t/2}^t \int_{\R^d} p_z(t-s,x-z)|q(s,z,y)| \, dz \, ds \nonumber\\
&\le&\frac c{t^{d/\alpha}}\int_0^{t/2} \int_{\R^d} |q(s,z,y)| \, dz \, ds + \frac{c}{t^{d/\alpha+1}}\int_{t/2}^t \int_{\R^d} p_z(t-s,x-z) \, dz \, ds \nonumber\\
&\le&\frac c{t^{d/\alpha}}, \label{ubound1}\end{eqnarray}
where   (\ref{integralq2}) and  Proposition \ref{pyintegral} were applied to estimate the integrals with respect to the space variable. 

Let $a$ the constant found in Proposition \ref{qestimate}.
Assume that  $|x-y|\ge 1+a$. By Corollary
\ref{estimate_pytx} for $0<s<t$ we have 
$$p_z(t-s,x-z)\le  c  e^{-c_1|x-y|},\quad |x-z|> |x-y|>1.$$
Proposition \ref{qestimate} implies that  for $0<s<t$,
$$
|q(s,z,y)| \le  c e^{-c_1\sqrt{|x-y|}}, \quad |y-z|> |x-y|>a%e^{-c_8 \sqrt{|x-y|}} \le \frac{c}{t^{d/\alpha + 1} (1 + |x - y|)^{d+1}}.
$$
Hence,
\begin{eqnarray}
 I(t,x,y)
&\le&\int_0^{t} \int_{|x-z|> |x-y|}\dots  \, dz \, ds + \int_0^{t} \int_{|y-z|> |x-y|}\dots  \, dz \, ds\nonumber\\
&\le& c e^{-c_1|x-y|}\int_0^{t} \int_{\R^d} |q(s,z,y)| \, dz \, ds \nonumber \\&+& c e^{-c_1\sqrt{|x-y|}}\int_{0}^t \int_{\R^d} p_z(t-s,x-z) \, dz \, ds\nonumber\\
&\le& c{t^{1/2}} e^{-c_1|x-y|}+ ct e^{-c_1\sqrt{|x-y|}}\nonumber\\
&\le& c  e^{-c_1\sqrt{|x-y|}}\label{ubound2}.\end{eqnarray}
Combining (\ref{ubound1}) and (\ref{ubound2}) we obtain the desired pointwise estimates of $u(t,x,y)$.

Next, (\ref{integralu1}) and (\ref{integralu2}) immediately follow from (\ref{integralq1}), (\ref{integralq2}) and  Proposition \ref{pyintegral}.
\end{proof}

For any $\zeta > 0$ and $x, y \in \R^d$ we put
$$
\calL_{\zeta} f(x) = 
\sum_{i = 1}^d   \int_{|w| > \zeta} \left[f(x + a_i(x) w) - f(x)\right] \, \mu(w) \, dw,
$$
$$
\calL_{\zeta}^y f(x) = 
\sum_{i = 1}^d  \int_{|w| > \zeta} \left[f(x + a_i(y) w) - f(x)\right] \, \mu(w) \, dw.
$$

\begin{lemma}
\label{Lxpyestimate} 
For any $\xi \in (0,1]$, $\zeta > 0$, $x, y, v \in \R^d$ and $t \in (\xi,\tau+\xi]$ we have
\begin{eqnarray}
\label{est1}
\sum_{i = 1}^d   \int_{\R} \left|p_y(t,x-y + a_{i}(v) w) - p_y(t,x-y)\right| \, \mu(w) \, dw
&\le& c(\xi) e^{-c |x - y|},\\
\label{est2}
\sum_{i = 1}^d   \int_{|w| \le \zeta} \left|p_y(t,x-y + a_{i}(v) w) - p_y(t,x-y)\right| \, \mu(w) \, dw
&\le& c(\xi) \zeta^{1 - \alpha}.
\end{eqnarray}
where $c(\xi)$ is a constant depending on $\xi, \tau, \alpha, d, \eta_1, \eta_2, \eta_3, \eps, \delta$.
\end{lemma}
\begin{proof}
We estimate the term for $i = 1$.  By Lemma \ref{pyholder} for $\gamma = 1$ we get for $w \in \R$
\begin{eqnarray*}
&&\left|p_y(t,x-y + a_{1}(v) w) - p_y(t,x-y)\right|\\
&& \le c t^{-1/\alpha} |w| \left(r_y\left(t,\frac{x-y}{2}\right) + r_y\left(t,\frac{x-y + a_{1}(v) w}{2}\right)\right). 
\end{eqnarray*}
Recall that if $|w| \ge 2 \delta$ then $\mu(w) = 0$. So we may assume that $|w| \le 2 \delta$. By Corollary \ref{estimate_pytx} we get
$$
r_y\left(t,\frac{x-y}{2}\right) + r_y\left(t,\frac{x-y + a_{1}(v) w}{2}\right) \le c_1 t^{-d/\alpha} e^{-c |x-y|}.
$$
Now (\ref{est1}) and (\ref{est2}) follow by the fact that $\mu(w) \le c 1_{[-2\delta,2\delta]}(w) |w|^{-1-\alpha}$.
\end{proof}

\begin{lemma}
\label{pyeps_limit}
Let $\tau_2 > \tau_1 > 0$ and assume that a function $f_t(x)$ is bounded and uniformly continuous on $[\tau_1,\tau_2] \times \R^d$. Then 
$$
\sup_{t \in [\tau_1,\tau_2], \, x \in \R^d} \left|\int_{\R^d} p_y(\eps_1,x-y) f_t(y) \, dy - f_t(x)\right| \to 0 \quad \text{as} \quad \eps_1 \to 0^+.
$$
\end{lemma}
\begin{proof}
The lemma follows easily by Propostion  \ref{pyintegral}.
\end{proof}

For any $t > 0$, $x,y \in \R^d$ we define
$$
\varphi_y(t,x) = \int_0^t \int_{\R^d} p_z(t-s,x-z) q(s,z,y) \, dz  \, ds.
$$
Clearly we have 
$$
u(t,x,y) = p_y(t,x-y) + \varphi_y(t,x).
$$
For any $t > 0$, $x,y \in \R^d$, $f \in \Bb$ we define
$$
\Phi_t f(x) = \int_{\R^d} \varphi_y(t,x) f(y) \, dy,
$$
$$
U_t f(x) = \int_{\R^d} u(t,x,y) f(y) \, dy.
$$
$$
Q_t f(x) = \int_{\R^d} q(t,x,y) f(y) \, dy.
$$

Now following ideas from \cite{KK2018} we will define the so-called approximate solutions. 

For any $t \ge 0$, $\xi \in [0,1]$, $t + \xi > 0$, $x,y \in \R^d$ we define
$$
\varphi_y^{(\xi)}(t,x) = \int_0^t \int_{\R^d} p_z(t-s+\xi,x-z) q(s,z,y) \, dz  \, ds
$$
and
$$
u^{(\xi)}(t,x,y) = p_y(t+\xi,x-y) + \varphi_y^{(\xi)}(t,x).
$$
For any $t \ge 0$, $\xi \in [0,1]$, $t + \xi > 0$, $x,y \in \R^d$, $f \in \Bb$ we define
$$
\Phi_t^{(\xi)} f(x) = \int_{\R^d} \varphi_y^{(\xi)}(t,x) f(y) \, dy,
$$
$$
U_t^{(\xi)} f(x) = \int_{\R^d} u^{(\xi)}(t,x,y) f(y) \, dy,
$$
$$
\Phi_0 f(x) = 0, \quad U_0^{(0)} f(x) = U_0 f(x) = f(x).
$$

By the same arguments as Corollary \ref{uintegral} we obtain the following result.
\begin{corollary}
\label{ueps_estimate}
For any $t \in [0,\infty)$, $\xi \in [0,1]$, $t + \xi > 0$, $x, y \in \R^d$ the kernel $u^{(\xi)}(t,x,y)$ is well defined. For any $t \in (0,\tau]$, $\xi \in [0,1]$, $x,y \in \R^d$ we have
$$
|u^{(\xi)}(t,x,y)| \le \frac{c}{(t + \xi)^{d/\alpha} (1 + |x - y|)^{d+1}}.
$$
For any $t \in (0,\tau]$, $\xi \in [0,1]$ and $x \in \R^d$ we have
\begin{equation*}
\int_{\R^d} |u^{(\xi)}(t,x,y)| \, dy \le c,
\end{equation*}
\begin{equation*}
\int_{\R^d} |u^{(\xi)}(t,y,x)| \, dy \le c.
\end{equation*}
\end{corollary}

\begin{lemma}
\label{qtf_uniform}
Let $f \in C_0(\R^d)$ and $\tau \ge \tau_2 > \tau_1 > 0$. Then $Q_t f(x)$ as a function of $(t,x)$ is uniformly continuous on $[\tau_1,\tau_2] \times \R^d$. We have $\lim_{|x| \to \infty} Q_t f(x) = 0$ uniformly in $t \in [\tau_1,\tau_2]$. For each $t > 0$ we have $Q_t f \in C_0(\R^d)$.
\end{lemma}
\begin{proof}
For any $\zeta > 0$, $y \in \R^d$ by Lemma \ref{pycontinuity} we obtain that
$$
(t,x) \to \calL_{\zeta}^{x}p_y(t,\cdot)(x-y) - \calL_{\zeta}^{y} p_y(t,\cdot)(x-y)
$$
is continuous on $(0,\infty) \times \R^d$. Using this and (\ref{est2}) we obtain that 
\begin{equation}
\label{q0_cont}
(t,x) \to q_0(t,x,y) \quad \text{is continuous on $(0,\infty) \times \R^d$.}
\end{equation}
By Proposition \ref{integralq0} we have
\begin{equation}
\label{q0_est}
|q_0(t,x,y)| \le \frac{c}{t^{1+d/\alpha}} e^{-c_1|x-y|}.
\end{equation}

For any $n \in \N$, $t > 0$, $x \in \R^d$ denote
$$
Q_{n,t} f(x) = \int_{\R^d} q_n(t,x,y) f(y) \, dy.
$$
By (\ref{q0_cont}), (\ref{q0_est}) and the dominated convergence theorem we obtain that $(t,x) \to Q_{0,t} f(x)$ is continuous on $(0,\infty) \times \R^d$. By Lemma \ref{integralqn} for any $t \in (0,\tau]$, $x \in \R^d$, $n \in \N$ we have
\begin{equation}
\label{qnt}
|Q_{n,t} f (x)| \le \frac{c_1^{n+1} t^{(n+1)/2 - 1}}{(n!)^{1/2}} \|f\|_{\infty}.
\end{equation}
Note that for any $t > 0$, $x \in \R^d$, $n \in \N$, $n \ge 1$ we have
$$
Q_{n,t} f (x) = \int_0^t \int_{\R^d} q_0(t-s,x,z) Q_{n-1,s}f(z) \, dz \, ds.
$$
For any $\eps_1 \in (0,\tau_1/2)$ using (\ref{q0_cont}), (\ref{q0_est}) and (\ref{qnt}) we obtain that
$$
(t,x) \to \int_0^{t - \eps_1} \int_{\R^d} q_0(t-s,x,z) Q_{n-1,s}f(z) \, dz \, ds
$$
is continuous on $[\tau_1,\tau_2] \times \R^d$. Note also that for any $\eps_1 \in (0,\tau_1/2)$, $t \in [\tau_1,\tau_2]$, $x \in \R^d$, $n \in \N$, $n \ge 1$ we have by (\ref{q0int})
\begin{eqnarray*}
\left| \int_{t - \eps_1}^t \int_{\R^d} q_0(t-s,x,z) Q_{n-1,s}f(z) \, dz \, ds \right| 
&\le& c \|f\|_{\infty} \int_{t - \eps_1}^t (t-s)^{-1/2} s^{-1/2} \, ds\\ 
&\le& c  \tau_1^{-1/2} \eps_1^{1/2} \|f\|_{\infty}.
\end{eqnarray*}
This implies that $(t,x) \to Q_{n,t}f(x)$ is continuous on $[\tau_1,\tau_2] \times \R^d$. Using this and (\ref{qnt}) we obtain that $(t,x) \to Q_t f(x) = \sum_{n =0}^{\infty} Q_{n,t}f(x)$ is continuous on $[\tau_1,\tau_2] \times \R^d$. By Proposition \ref{qestimate} we obtain that $\lim_{|x| \to \infty} Q_t f(x) = 0$ uniformly in $t \in [\tau_1,\tau_2]$. This implies the assertion of the lemma.
\end{proof}

\begin{proposition}
\label{utfholder}
Choose $\gamma \in (0,\alpha)$. For any $t \in (0,\tau]$, $x, x' \in \R^d$, $f \in \Bb$ we have
$$
|U_tf(x) - U_tf(x')| \le c t^{-\gamma/\alpha} |x - x'|^{\gamma} \|f\|_{\infty}.
$$
\end{proposition}
\begin{proof}
We have
\begin{eqnarray*}
U_tf(x) - U_tf(x') 
&=& \int_{\R^d} (p_y(t,x-y) - p_y(t,x'-y)) f(y) \, dy\\
&& + \int_0^t \int_{\R^d} (p_z(t-s,x-z) - p_z(t-s,x'-z)) Q_s f(z) \, dz \, ds\\
&=& \text{I} + \text{II}.
\end{eqnarray*}
By Lemma \ref{pyholder} and Proposition \ref{pyintegral} we get
$$
|\text{I}| \le c \|f\|_{\infty} |x-x'|^{\gamma} t^{-\gamma/\alpha} \int_{\R^d} (r_y(t,(x-y)/2) + r_y(t,(x'-y)/2)) \, dy
\le c \|f\|_{\infty} |x-x'|^{\gamma} t^{-\gamma/\alpha}.
$$
By Lemma \ref{pyholder} and Propositions \ref{pyintegral}, \ref{qestimate} we obtain
\begin{eqnarray*}
|\text{II}| &\le& 
c \|f\|_{\infty} |x-x'|^{\gamma}
\int_0^t \int_{\R^d} (t - s)^{-\gamma/\alpha} \left(r_z(t-s,(x-z)/2) \right.\\
&& \quad \quad \quad \quad \quad \quad \quad \quad \quad \quad  + \left. r_z(t-s,(x'-z)/2)\right) s^{-1/2} \, dz \, ds\\
&\le& c \|f\|_{\infty} |x-x'|^{\gamma} \int_0^t (t - s)^{-\gamma/\alpha} s^{-1/2} \, ds\\
&\le& c \|f\|_{\infty} |x-x'|^{\gamma} t^{1/2 -\gamma/\alpha}.
\end{eqnarray*}
\end{proof}

Note that by Lemma \ref{Lxpyestimate} for any $\xi \in (0,1]$, $t \in [\xi,\tau+\xi]$, $x, z \in \R^d$ we have
\begin{equation}
\label{tderivative}
\left|\frac{\partial p_z(t,x-z)}{t}\right| = \left|\calL^z p_z(t,\cdot)(x-z)\right|\le c(\xi) e^{-c|x-z|},  
\end{equation}
where $c(\xi)$ is a constant depending on $\xi, \tau, \alpha, d, \eta_1, \eta_2, \eta_3, \eps, \delta$.

The next lemma is similar to \cite[Lemma 4.1]{KK2018}. 
\begin{lemma}
\label{regularity_uteps}
(i) For every $f \in C_0(\R^d)$, $\xi \in (0,1]$ the function $U_t^{(\xi)} f(x)$ belongs to $C^1((0,\infty))$ as a function of $t$ and to $C_0^2(\R^d)$ as a function of $x$. Moreover we have
\begin{equation}
\label{tderivative_estimate}
\left|\frac{\partial}{\partial t} (U_t^{(\xi)} f)(x)\right| \le c({\xi}) t^{-1/2} \|f\|_{\infty},
\end{equation}
for each $f \in C_0(\R^d)$, $t \in (0,\tau]$, $x \in \R^d$, $\xi \in (0,1]$, where $c(\xi)$ depends on $\xi, \tau, \alpha, d, \eta_1, \eta_2, \eta_3$.

(ii) For every $f \in C_0(\R^d)$ we have
\begin{equation*}
\lim_{t,\xi \to 0^+}\|U_t^{(\xi)} f - f\|_{\infty} = 0.
\end{equation*}

(iii) For every $f \in C_0(\R^d)$ we have
\begin{equation*}
U_t^{(\xi)} f(x) \to 0, \quad \text{as} \quad |x| \to \infty,
\end{equation*}
uniformly in $t \in [0,\tau]$, $\xi \in [0,1]$.

(iv) For every $f \in C_0(\R^d)$ we have
\begin{equation*}
\|U_t^{(\xi)} f - U_t f\|_{\infty} \to 0, \quad \text{as} \quad \xi \to 0^+,
\end{equation*}
uniformly in $t \in [0,\tau]$. 
\end{lemma}
\begin{proof}
(i) Let $f \in C_0(\R^d)$, $t \in (0,\tau]$, $\xi \in (0,1]$ and $x \in \R^d$. We have
\begin{eqnarray*}
&& \lim_{h \to 0^+} \frac{\Phi_{t+h}^{(\xi)} f(x) - \Phi_{t}^{(\xi)} f(x)}{h}\\
&=& \lim_{h \to 0^+} \frac{1}{h} \int_t^{t+h} \int_{\R^d} p_z(t+h-s+\xi,x-z) Q_sf(z) \, dz \, ds\\
&& + \lim_{h \to 0^+} \int_0^t \int_{\R^d} \frac{p_z(t+h-s+\xi,x-z) - p_z(t-s+\xi,x-z)}{h}
Q_sf(z) \, dz \, ds\\
&=& \text{I} + \text{II}.
\end{eqnarray*}
By Lemmas \ref{pycontinuity}, \ref{qtf_uniform}, Corollary \ref{estimate_pytx} and Proposition \ref{qestimate} we get
$$
\text{I} = \int_{\R^d} p_z(\xi,x-z) Q_t f(z) \, dz.
$$

By Lemma \ref{pycontinuity}, the dominated convergence theorem, (\ref{tderivative}) and Proposition \ref{qestimate} we get
$$
\text{II} = \int_0^{t} \int_{\R^d}  \frac{\partial p_z(t-s+\xi,x-z)}{\partial t} Q_sf(z) \, dz \, ds.
$$
By similar arguments we get the analogous result for $\lim_{h \to 0^-} \left(\Phi_{t+h}^{(\xi)} f(x) - \Phi_{t}^{(\xi)} f(x)\right)/h$.

By (\ref{tderivative}) we get
$$
\frac{\partial}{\partial t} \int_{\R^d} p_z(t+\xi,x-z) f(z) \, dz 
= \int_{\R^d} \frac{\partial}{\partial t} p_z(t+\xi,x-z) f(z) \, dz.
$$
Hence we have
\begin{eqnarray}
\nonumber
\frac{\partial}{\partial t} (U_t^{(\xi)} f)(x)
&=& \int_{\R^d} \frac{\partial}{\partial t} p_z(t+\xi,x-z) f(z) \, dz\\
\nonumber
&+& \int_{\R^d} p_z(\xi,x-z) Q_t f(z) \, dz\\
\label{tchange}
&+& \int_0^{t} \int_{\R^d}  \frac{\partial p_z(t-s+\xi,x-z)}{\partial t} Q_sf(z) \, dz \, ds
\end{eqnarray}
Using this, (\ref{tderivative}), Propositions \ref{pyintegral} and \ref{qestimate} we obtain (\ref{tderivative_estimate}). We also obtain that for every $f \in C_0(\R^d)$, $\xi \in (0,1]$ the function $U_t^{(\xi)} f(x)$ belongs to $C^1((0,\infty))$ as a function of $t$. 

The fact that $U_t^{(\xi)} f \in C_0^2(\R^d)$ for $\xi \in (0,1]$ follows by Lemmas \ref{pycontinuity}, \ref{properties_pytx}, Proposition \ref{qestimate} and Lemma \ref{qtf_uniform}.

(ii) Fix $f \in C_0(\R^d)$. For any $\xi \in [0,1]$, $t \ge 0$, $t + \xi > 0$, $x \in \R^d$ we have
$$
U_t^{(\xi)} f(x) = \int_{\R^d} p_y(t + \xi, x - y) f(y) \, dy + \Phi_t^{(\xi)} f(x).
$$
For any $\xi \in [0,1]$, $t \in [0,\tau]$, $\xi + t > 0$, $x \in \R^d$ by Proposition \ref{qestimate} and Proposition \ref{pyintegral} we get
\begin{equation}
\label{phitxi}
\left|\Phi_t^{(\xi)} f(x)\right| \le c \|f\|_{\infty} \int_0^t s^{-1/2} \, ds \le c \|f\|_{\infty} t^{1/2}.
\end{equation}
By (\ref{supsup}), Proposition  \ref{pyintegral} and the fact that $f$ is uniformly continuous on $\R^d$ we obtain
$$
\lim_{t,\xi \to 0^+} \int_{\R^d} p_y(t + \xi, x - y) f(y) \, dy - f(x) = 0
$$
uniformly with respect to $x \in \R^d$. This and (\ref{phitxi}) gives (ii).

(iii) This follows easily from (ii) and Corollary \ref{ueps_estimate}.

(iv) Fix $f \in C_0(\R^d)$. By Lemma \ref{pycontinuity}, Corollary \ref{estimate_pytx} and the dominated convergence theorem we obtain that 
$$
(t,x) \to \int_{\R^d} p_y(t,x-y) f(y) \, dy
$$
is continuous on $(0,\tau+1] \times \R^d$. It follows that
$$
(\xi,t,x) \to \int_{\R^d} p_y(t + \xi,x-y) f(y) \, dy
$$
is continuous on $[0,1]\times(0,\tau] \times \R^d$. Using Lemma \ref{pycontinuity}, Corollary \ref{estimate_pytx}, Proposition \ref{qestimate} and the dominated convergence theorem we obtain that for any $s \in (0,\tau)$ 
$$
(\xi,t,x) \to \int_{\R^d} p_z(t + \xi - s,x-z) Q_s f(z) \, dz
$$
is continuous on $[0,1]\times(s,\tau] \times \R^d$. Using this, Corollary \ref{estimate_pytx}, Proposition \ref{qestimate} and the dominated convergence theorem we obtain that
$$
(\xi,t,x) \to \Phi_t^{(\xi)}f(x) = \int_0^{\tau} 1_{(0,t)}(s) \int_{\R^d} p_z(t + \xi - s,x-z) Q_s f(z) \, dz \, ds
$$
is continuous on $[0,1] \times (0,\tau] \times \R^d$. Hence $(\xi,t,x) \to U_t^{(\xi)}f(x)$ is continuous on $[0,1] \times (0,\tau] \times \R^d$. Using (\ref{phitxi}) we obtain that
\begin{equation}
\label{ut_continuity}
(\xi,t,x) \to U_t^{(\xi)}f(x) \quad \text{is continuous on $[0,1]\times [0,\tau] \times \R^d$.}
\end{equation}
This and (iii) implies (iv).
\end{proof}

By the same arguments as in the proof of Lemma \ref{regularity_uteps} (iv) we obtain the following result.
\begin{lemma}
\label{u_continuity}
For any $f \in \Bb$ the function $(t,x) \to U_t f(x)$ is continuous on $(0,\infty) \times \R^d$. For any $\xi \in (0,1]$, $f \in \Bb$ the function $(t,x) \to U_t^{(\xi)} f(x)$ is continuous on $[0,\infty) \times \R^d$.
\end{lemma}

Heuristically, now our aim is to show that if $\xi$ is small then $\frac{\partial}{\partial t} (U_t^{(\xi)} f)(x) - \calL (U_t^{(\xi)} f)(x)$ is small.
For any $t > 0$, $\xi \in (0,1]$, $x \in \R^d$ we put 
$$
\Lambda_t^{(\xi)} f(x) = \frac{\partial}{\partial t} (U_t^{(\xi)} f)(x) - \calL (U_t^{(\xi)} f)(x).
$$

\begin{lemma}
\label{Lambda_teps}
$\calL (U_t^{(\xi)} f)(x)$ is well defined for every $f \in C_0(\R^d)$, $t \in (0,\tau]$, $\xi \in (0,1]$ and $x \in \R^d$ and  we have
\begin{equation}
\label{Lambdat_formula}
\Lambda_t^{(\xi)} f(x) = \int_{\R^d} p_z(\xi,x-z) Q_t f(z) \, dz - Q_{t+\xi} f(x) 
+ \int_t^{t+\xi} \int_{\R^d} q_0(t-s+\xi,x,z) Q_s f(z) \, dz \, ds.
\end{equation}
Moreover we have
\begin{equation}
\label{Lzeta_estimate}
\left|\calL_{\zeta} (U_t^{(\xi)} f)(x)\right| \le c({\xi}) \|f\|_{\infty}, \quad \zeta > 0,
\end{equation}
\begin{equation}
\label{L_estimate}
\left|\calL (U_t^{(\xi)} f)(x)\right| \le c({\xi}) \|f\|_{\infty}.
\end{equation}
for each $f \in C_0(\R^d)$, $x \in \R^d$, $t \in (0,\tau]$, $\xi \in (0,1]$, where $c(\xi)$ is a constant depending on $\xi, \tau, \alpha, d, \eta_1, \eta_2, \eta_3, \eps, \delta$.
\end{lemma}
\begin{proof}
Let $f \in C_0(\R^d)$, $t \in (0,\tau]$, $\xi \in (0,1]$, $x \in \R^d$ and $\zeta > 0$. We have 
\begin{eqnarray}
\nonumber
\calL_{\zeta} (U_t^{(\xi)} f)(x)
&=& \int_{\R^d} \calL_{\zeta}^x p_z(t+\xi,\cdot)(x-z) f(z) \, dz\\
\label{calLzeta}
&+& \int_0^{t} \int_{\R^d}  \calL_{\zeta}^x p_z(t-s+\xi,\cdot)(x-z) Q_sf(z) \, dz \, ds.
\end{eqnarray}
Using this, Lemma \ref{Lxpyestimate} and Proposition \ref{qestimate} we obtain (\ref{Lzeta_estimate}). By (\ref{calLzeta}), the dominated convergence theorem, Lemma \ref{Lxpyestimate} and Proposition \ref{qestimate} one gets
\begin{eqnarray}
\nonumber
\calL (U_t^{(\xi)} f)(x)
&=& \lim_{\zeta \to 0^+} \calL_{\zeta}^x(U_t^{(\xi)} f)(x)\\
\nonumber
&=& \int_{\R^d} \calL^x p_z(t+\xi,\cdot)(x-z) f(z) \, dz\\
\label{Lchange}
&+& \int_0^{t} \int_{\R^d}  \calL^x p_z(t-s+\xi,\cdot)(x-z) Q_sf(z) \, dz \, ds.
\end{eqnarray}
Using this and again Lemma \ref{Lxpyestimate} and Proposition \ref{qestimate} we obtain (\ref{L_estimate}).

Note that for  $s \in [0,t)$, $z \in \R^d$ we have
$$
\frac{\partial p_z(t-s+\xi,x-z)}{\partial t} - \calL^x p_z(t-s+\xi,\cdot)(x-z)
= -q_0(t-s+\xi,x,z).
$$
Using this, (\ref{tchange}) and (\ref{Lchange}) we get
\begin{eqnarray}
\nonumber
\Lambda_t^{(\xi)} f(x) &=& \int_{\R^d} p_z(\xi,x-z) Q_t f(z) \, dz\\
\nonumber
&& -\int_{\R^d} q_0(t +\xi,x,z) f(z) \, dz\\
\label{tLchange}
&&
- \int_0^t \int_{\R^d} q_0(t-s+\xi,x,z) Q_s f(z) \, dz \, ds.
\end{eqnarray}
For $\xi \in (0,1]$, $t \in (0,\tau]$, $x \in \R^d$ by the definition of $q(t,x,y)$ we obtain
\begin{equation*}
\int_{\R^d} q_0(t +\xi,x,z) f(z) \, dz
= Q_{t+\xi} f(x) - 
\int_0^{t+\xi} \int_{\R^d} q_0(t-s+\xi,x,z) Q_s f(z) \, dz \, ds.
\end{equation*}
Using this and (\ref{tLchange}) we obtain (\ref{Lambdat_formula}).
\end{proof}

The next lemma is similar to \cite[Lemma 4.2]{KK2018}.
\begin{lemma}
\label{heat_u}
(i) For any $f \in C_0(\R^d)$ we have
\begin{equation*}
\Lambda_t^{(\xi)} f(x) \to 0, \quad \text{as} \quad \xi \to 0^+,
\end{equation*}
uniformly in $(t,x) \in [\tau_1,\tau_2] \times \R^d$ for every $\tau \ge \tau_2 > \tau_1 > 0$.

(ii) For any $f \in C_0(\R^d)$ we have
\begin{equation}
\label{Lambda_s}
\int_0^t \Lambda_s^{(\xi)} f(x) \, ds \to 0, \quad \text{as} \quad \xi \to 0^+,
\end{equation}
uniformly in $(t,x) \in (0,\tau] \times \R^d$.
\end{lemma}
\begin{proof}
Let $f \in C_0(\R^d)$ and $0 < \tau_1 < \tau_2 \le \tau$. For any $t > 0$, $x \in \R^d$, $\xi \in (0,1]$ we put
\begin{equation*}
\Lambda_t^{(\xi,1)} f(x) = \int_{\R^d} p_z(\xi,x-z) Q_t f(z) \, dz - Q_{t+\xi} f(x).
\end{equation*}
\begin{equation*}
\Lambda_t^{(\xi,2)} f(x) = 
\int_t^{t+\xi} \int_{\R^d} q_0(t-s+\xi,x,z) Q_s f(z) \, dz \, ds.
\end{equation*}
By Lemma \ref{qtf_uniform} we get
$$
\sup_{t \in [\tau_1,\tau_2], \, x \in \R^d} \left|Q_{t+\xi} f(x) -Q_t f(x)\right| \to 0 \quad \text{as} \quad \xi \to 0^+.
$$
By Lemmas \ref{pyeps_limit} and \ref{qtf_uniform} we obtain
$$
\sup_{t \in [\tau_1,\tau_2], \, x \in \R^d} 
\left|\int_{\R^d} p_z(\xi,x-z) Q_t f(z) \, dz - Q_t f(x)\right| 
\to 0 \quad \text{as} \quad \xi \to 0^+.
$$
This gives (i) for $\Lambda_t^{(\xi,1)} f(x)$ instead of $\Lambda_t^{(\xi)} f(x)$. 

By Proposition \ref{qestimate} for any $t \in (0,\tau]$, $x \in \R^d$, $\xi \in (0,1]$ we get
\begin{equation}
\label{lambdatf}
\left|\Lambda_t^{(\xi,1)} f(x)\right| \le c \|f\|_{\infty} t^{-1/2}.
\end{equation}
This allows to use the dominated convergence theorem in the integral (\ref{Lambda_s}) with $\Lambda_t^{(\xi)} f(x)$ replaced by $\Lambda_t^{(\xi,1)} f(x)$. So (ii) for $\Lambda_t^{(\xi,1)} f(x)$ follows from (i) for $\Lambda_t^{(\xi,1)} f(x)$.

For any $t \in (0,\tau]$, $x \in \R^d$, $\xi \in (0,1]$ by Propositions \ref{integralq0} and \ref{qestimate} we get
\begin{equation}
\label{lambdatf2}
\left|\Lambda_t^{(\xi,2)} f(x)\right| \le c \|f\|_{\infty} \int_t^{t+\xi} ((t-s+\xi)s)^{-1/2} \, ds.
\end{equation}
This implies (i) and (ii) for $\Lambda_t^{(\xi,2)} f(x)$.
\end{proof}

\begin{lemma}
\label{lowerbound}
There exist $\eps_1 \in (0,1]$ and $t_1 \in (0,1]$ such that for any $t \in (0,t_1]$, $x, y \in \R^d$, $|x-y| \le \eps_1 t^{1/\alpha}$ we have
$$
u(t,x,y) \ge c_1 t^{-d/\alpha}.
$$
$\eps_1$, $t_1$ depend on $\alpha, d, \eta_1, \eta_2, \eta_3, \eps, \delta$.
\end{lemma}
\begin{proof}
By the weak lower scaling property of the symbol $\Phi_{\delta}^{(1)}$ (see proof of Lemma \ref{gtht}) and by \cite[formula (23)]{BGR2014} we get that $g_t(0) \ge c t^{-1/\alpha}$. Using this and Lemma \ref{gtht} there exist $\eps_2 > 0$, $t_2 > 0$ such that for $|y| \le \eps_2 t^{1/\alpha}$, $t \le t_2$ we have $g_t(y) \ge c t^{-1/\alpha}$. It follows that there exist $\eps_3 > 0$, $t_3 > 0$ such that for $x, y \in \R^d$, $|x - y| \le \eps_3 t^{1/\alpha}$, $t \le t_3$ we have $p_y(t,x-y) \ge c t^{-d/\alpha}$. By Lemma \ref{qestimate} and Proposition \ref{pyintegral} for $t \in (0,\tau]$, $x, y \in \R^d$ and $a \in [1/2,1)$ we get
\begin{eqnarray*}
|\varphi_y(t,x)| &=& \left|\int_0^{t a} \int_{\R^d} p_z(t-s,x-z) q(s,z,y) \, dz \, ds \right. \\
&& + \left. \int_{ta}^{t} \int_{\R^d} p_z(t-s,x-z) q(s,z,y) \, dz \, ds
\right|\\
&\le& c (t-ta)^{-d/\alpha} \int_0^{ta} s^{-1/2} \, ds + c t^{-d/\alpha - 1} t (1-a)\\
&\le& c t^{-d/\alpha} (t^{1/2} (1-a)^{-d/\alpha} + 1-a).
\end{eqnarray*} 
By an appropriate choice of $a$ there exists $\eps_1$ and $t_1 > 0$ such that for any $x, y \in \R^d$, $|x-y| \le \eps_1 t^{1/\alpha}$, $t \le t_1$ we have
$$
u(t,x,y) = p_y(t,x-y) - \varphi_y(t,x)
\ge ct^{-d/\alpha} - c_2 t^{-d/\alpha} (t^{1/2} (1-a)^{-d/\alpha} + 1-a)
\ge c_1 t^{-d/\alpha}.
$$
\end{proof}

\section{Construction and properties of the semigroup of $X_t$}

Let us intruduce the following notation
$$
\nu(x) = \frac{\Aa}{|x|^{1 + \alpha}} - \mu(x), \quad x\in \R,
$$
$$
\lambda = d \int_{\R} \nu(x) \, dx<\infty.
$$

Note that by (\ref{KLR}), for any $x\in \R^d$ and  $f \in \Bb$, we have
$$
\calR f(x) =  \sum_{i = 1}^d   \int_{\R} \left[f(x + a_{i}(x) w) - f(x)\right] \, \nu(w) \, dw.
$$ 
We denote, for any $x\in \R^d$ and  $f \in \Bb$, 
$$
\calN f(x) = \sum_{i = 1}^d  \int_{\R} \left[f(x + a_{i}(x) w)\right] \, \nu(w) \, dw. 
$$
It is clear that

\begin{equation}
\label{N_bound}
||\calN f||_\infty \le  \lambda|| f||_\infty. 
\end{equation}

For any $t \ge 0$, $x \in \R^d$ and $n \in \N$, $n \ge 1$, $f \in \Bb$ we define
\begin{eqnarray}
\label{Psi_nt0}
\Psi_{0,t} f(x) &=& U_t f(x),\\
\label{Psi_nt}
\Psi_{n,t} f(x) &=& \int_0^t  U_{t-s}(\calN(\Psi_{n-1,s}f))(x) \, ds, \quad n \ge 1.
\end{eqnarray}

For any $t \ge 0$, $\xi \in [0,1]$, $x \in \R^d$ and $n \in \N$,  $f \in \Bb$ we define
\begin{eqnarray}
\label{Psi_nteps0}
\Psi_{0,t}^{(\xi)} f(x) &=& U_t^{(\xi)} f(x), \\
\label{Psi_nteps}
\Psi_{n,t}^{(\xi)} f(x) &=& \int_0^t  U_{t-s}^{(\xi)}(\calN(\Psi_{n-1,s}^{(\xi)}f))(x) \, ds, \quad n \ge 1.
\end{eqnarray}
We remark that  $$\Psi_{n,t}=\Psi_{n,t}^{(0)}.$$

\begin{lemma}
\label{psint}
$\Psi_{n,t} f(x)$ and $\Psi_{n,t}^{(\xi)} f(x)$ are well defined for any $t > 0$,  $f \in \Bb$, $x\in \R^d$, $n \in \N$ and $\xi \in [0, 1]$. For any  $f \in \Bb$, $x \in \R^d$, $n \in \N$ we have
\begin{eqnarray}
\label{psint1}
\left|\Psi_{n,t} f(x)\right| &\le& \frac{c_1^{n+1} t^n}{n!} \|f\|_{\infty}, \quad t \in (0,\tau],\\
\label{psint2}
\left|\Psi_{n,t}^{(\xi)} f(x) \right| &\le&  \frac{c_1^{n+1} t^n}{n!} \|f\|_{\infty}, 
\quad \xi \in (0,1],\,\,t \in [0,\tau].
\end{eqnarray}
\end{lemma}
\begin{proof}
We will only  show the result for $\Psi_{n,t} f(x)$ using the induction.  The proof  for $\Psi_{n,t}^{(\xi)} f(x)$ is almost the same.

Let $c$ be the constant from (\ref{integralu1}) and put $c_1 = (\lambda \vee 1) c$. For $n = 0$ (\ref{psint1}) follows from (\ref{integralu1}). %For $n \ge 1$ (\ref{psint1}) follows by induction argument.
 Assume that (\ref{psint1}) holds for $n \ge 0$, we will show it for $n+1$. 
Indeed, applying (\ref{integralu1}) and (\ref{N_bound}),  we get
$$
\left|\Psi_{n+1,t} f(x)\right| \le \int_0^t  \int_{\R^d}  |u(t-s,x,z)| \, dz \frac{\lambda c_1^{n+1} s^n}{n!} \, ds 
\le \frac{c_1^{n+2} t^{n+1}}{(n+1)!}.
$$
\end{proof}

For any $x \in \R^d$ we define
\begin{eqnarray*}
T_t f(x) &=& e^{-\lambda t} \sum_{n = 0}^{\infty} \Psi_{n,t} f(x), \quad t > 0,\\
T_0 f(x) &=& f(x),\\
T_t^{(\xi)} f(x) &=& e^{-\lambda t} \sum_{n = 0}^{\infty}  \Psi_{n,t}^{(\xi)} f(x), \quad t \ge 0, \,\, \xi \in [0,1].
\end{eqnarray*}

Our ultimate aim will be to show that for any $t > 0$ we have $T_t = P_t$, where $P_t$ is given by (\ref{semigroup}).

By Lemma \ref{psint} we obtain
\begin{corollary}
\label{ptxy}
$T_t f(x)$ and  $T_t^{(\xi)} f(x)$ are  well defined for any $t \ge 0$,  $f \in \Bb$, $x \in \R^d$, $\xi \in [0,1]$  and for $t \in [0,\tau]$ we have $\max\{|T_t f(x)|,|T_t^{(\xi)} f(x)|\} \le c \|f\|_{\infty}$.% and $|T_t^{(\xi)} f(x)| \le c \|f\|_{\infty}$ .  
%Moreover $T_t^{(\xi)} f(x)$ is well defined for any $t \ge 0$, $f \in \Bb$, $\xi \in [0,1]$, $x \in \R^d$ and for $t \in [0,\tau]$ we have $|T_t^{(\xi)} f(x)| \le c \|f\|_{\infty}$.
\end{corollary}

Next, we obtain the following regularity results concerning operators $T_t$.
\begin{theorem} 
\label{TtL1Linfty} 
For any $\gamma \in (0,\alpha/d)$, $t \in (0,\tau]$, $x \in \R^d$ and $f \in L^1(\R^d) \cap L^{\infty}(\R^d)$ we have
\begin{equation*}
|T_t f(x)| \le c t^{-\gamma d/\alpha}  \|f\|_\infty^{1-\gamma}  \|f\|_1^{\gamma}.
\end{equation*}
\end{theorem}
\begin{proof}
For any $t \in (0,\tau]$, $x \in \Rd$ by Corollary \ref{uintegral} we get $|U_t f(x)| \le c \|f\|_{\infty}$, $|U_t f(x)| \le c t^{-d/\alpha}\|f\|_1$. Fix $\gamma \in (0,\alpha/d)$. It follows that for any $t \in (0,\tau]$, $x \in \Rd$ we have $|U_t f(x)| \le c t^{-\gamma d/\alpha} \|f\|_\infty^{1-\gamma}  \|f\|_1^{\gamma}$. Hence $|\Psi_{1,t} f(x)| \le c \|f\|_\infty^{1-\gamma} \|f\|_1^{\gamma}$. Using the same arguments as in Lemma \ref{psint} for any $t \in (0,\tau]$, $x \in \Rd$, $n \in \N$, $n \ge 1$ one gets $|\Psi_{n,t} f(x)| \le c^n t^{n-1} \|f\|_\infty^{1-\gamma} \|f\|_1^{\gamma}/(n-1)!$, which implies the assertion of the theorem.
\end{proof}

\begin{theorem}
\label{Holdermain}
Choose $\gamma \in (0,\alpha)$. For any $t \in (0,\tau]$, $x, x' \in \R^d$, $f \in \Bb$ we have
$$
|T_tf(x) - T_tf(x')| \le c t^{-\gamma/\alpha} |x - x'|^{\gamma} \|f\|_{\infty}.
$$
\end{theorem}
\begin{proof}
We have 
$$
T_t f(x) = e^{-\lambda t} U_t f(x) + e^{-\lambda t} \sum_{n = 1}^{\infty}  \Psi_{n,t} f(x).
$$
By Lemma \ref{utfholder} it remains to show 
\begin{equation}
\label{holder_sum}
\left| \sum_{n = 1}^{\infty}  \Psi_{n,t} f(x) - \sum_{n = 1}^{\infty}  \Psi_{n,t} f(x')\right|
\le c t^{-\gamma/\alpha} |x - x'|^{\gamma} \|f\|_{\infty}.
\end{equation}
Let $n \in \N$, $n \ge 1$. By (\ref{Psi_nt}) we get 
\begin{equation}
\label{psintf}
\Psi_{n,t} f(x) - \Psi_{n,t} f(x') = 
\int_0^t  \left( U_{t-s}(\calN(\Psi_{n-1,s}f))(x) - U_{t-s}(\calN(\Psi_{n-1,s}f))(x')\right) \, ds.
\end{equation}
By Lemma \ref{psint} we have 
$$
\left|\Psi_{n,t} f(x)\right| \le \frac{c_1^{n+1} t^n}{n!} \|f\|_{\infty}.
$$
Hence for $s \in (0,\tau]$, by (\ref{N_bound}), we arrive at
$$
\left|\calN(\Psi_{n-1,s}f))(x)\right| \le \frac{\lambda c_1^{n} s^{n-1}}{(n - 1)!} \|f\|_{\infty}.
$$
Using this and Lemma \ref{utfholder} we get
$$
\left| U_{t-s}(\calN(\Psi_{n-1,s}f))(x) - U_{t-s}(\calN(\Psi_{n-1,s}f))(x')\right| 
\le c (t - s)^{-\gamma/\alpha} |x - x'|^{\gamma} \frac{c_1^{n} s^{n-1}}{(n - 1)!} \|f\|_{\infty}.
$$
Combining this with (\ref{psintf}) we obtain
$$
\left| \Psi_{n,t} f(x) - \Psi_{n,t} f(x')\right| \le
c \frac{c_1^{n} }{(n - 1)!} \|f\|_{\infty} |x - x'|^{\gamma} t^{n -\gamma/\alpha}.
$$
This implies (\ref{holder_sum}), which finishes the proof.
\end{proof}

Clearly, we have, by applying (\ref{integralu1}) and (\ref{N_bound}), the following lemma. 
\begin{lemma}
\label{integral_u}
There exists $a \ge 1$ such that for any $f \in \Bb$, $x \in \Rd$, $\xi \in [0,1]$, $t \in (0,\tau]$ we have
$$
\left|\int_0^t  U_{t-s}^{(\xi)}(\calN(f))(x) \, ds\right| \le a \|f\|_{\infty},
$$
where $a$ depends on $\tau, \alpha, d, \eta_1, \eta_2, \eta_3$.
\end{lemma}

\begin{lemma}
\label{eps_q}
Assume that $f \in \Bb$. Put $\beta = 1/ (4(d/\alpha+1))$. For any  $t \in (0,\tau]$, $x \in \R^d$, we have 
$$
|Q_t f(x)| \le \frac{c \|f\|_{\infty}}{t^{3/4} (\dist(x,\supp(f))  +1)^{\beta}}.
$$
\end{lemma}
\begin{proof}
Let $t \in (0,\tau]$ be arbitrary. By Proposition \ref{qestimate} we get for $x \in \Rd$
\begin{eqnarray*}
|Q_t f(x)| &\le& c t^{-1/2} \|f\|_{\infty},\\
|Q_t f(x)| &\le& \frac{c \|f\|_{\infty}}{t^{d/\alpha + 1} (\dist(x,\supp(f))  +1)}.
\end{eqnarray*}
It follows that 
\begin{eqnarray*}
|Q_t f(x)|^{1 -\beta} &\le& c t^{(-1/2)(1-\beta)} \|f\|_{\infty}^{1-\beta} \le c t^{-1/2} \|f\|_{\infty}^{1-\beta},\\
|Q_t f(x)|^{\beta} &\le& \frac{c \|f\|_{\infty}^{\beta}}{t^{1/4} (\dist(x,\supp(f))  +1)^{\beta}}.
\end{eqnarray*}
This implies the assertion of the lemma.
\end{proof}

\begin{lemma}
\label{eps_u}
Assume that $f \in \Bb$. For any $\eps_1 > 0$ there exists $r \ge 1$ (depending on $\eps_1, \tau, \alpha, d, \eta_1, \eta_2, \eta_3$) such that for any $\xi \in [0,1]$, $t \in [0,\tau]$, $x \in \R^d$, if $\dist(x, \supp(f)) \ge r$ then $\left| U_{t}^{(\xi)} f(x) \right| \le \eps_1 \|f\|_{\infty}$.
\end{lemma}
\begin{proof}
Let $\xi \in [0,1]$, $t \in [0,\tau]$ be arbitrary. Assume that $\dist(x, \supp(f)) \ge 1$ and $t + \xi > 0$. By Lemmas \ref{pyintegral}, \ref{eps_q} we get
\begin{eqnarray*}
|\Phi_t^{(\xi)} f(x)| &=& \left|\int_0^t \int_{\R^d} p_z(t-s + \xi,x-z) Q_s f(z) \, dz \, ds\right| \\
&\le& 
\frac{c \|f\|_{\infty}}{(\dist(x,\supp(f))  +1)^{\beta}} \int_0^t \frac{1}{s^{3/4}} \, ds.
\end{eqnarray*}
By Corollary \ref{estimate_pytx} we get
$$
\int_{\Rd} p_z(t + \xi,x-z) f(z) \, dz \le c \|f\|_{\infty} e^{-c_1 \dist(x, \supp(f))}.
$$
This gives the assertion of the lemma.
\end{proof}

The proof of the next lemma is standard and it is omitted.
\begin{lemma}
\label{eps_n}
Assume that $f \in \Bb$. For any $\eps_1 > 0$ there exists $r \ge 1$ (depending on $\eps_1, \tau, \alpha, d, \eta_1, \eta_2, \eta_3$) such that, for any $x \in \R^d$, if $\dist(x, \supp(f)) \ge r$, then $\left| \calN f(x) \right| \le \eps_1 \|f\|_{\infty}$.
\end{lemma}

\begin{lemma}
\label{eps_integral_u}
Assume that $f \in \Bb$. For any $\eps_1 > 0$ there exists $r \ge 1$ (depending on $\eps_1, \tau, \alpha, d, \eta_1, \eta_2, \eta_3$) such that for any $\xi \in [0,1]$, $t \in [0,\tau]$, $x \in \R^d$, if $\dist(x, \supp(f)) \ge r$ then $\left|\int_0^t  U_{t-s}^{(\xi)}(\calN(f))(x) \, ds\right| \le \eps_1 \|f\|_{\infty}$.
\end{lemma}
\begin{proof}
Let $\xi \in [0,1]$, $t \in (0,\tau]$ be arbitrary. Choose $\eps_1 > 0$. There exists $a > 0$ such that for any $g \in \Bb$ we have $\|U_t^{(\xi)} g\|_{\infty} \le a \|g\|_{\infty}$, where $a$ depends on $\tau, \alpha, d, \eta_1, \eta_2, \eta_3$.  
By Lemma \ref{eps_n} there exists $r_1 \ge 1$ such that, if $\dist(x,\supp(f)) \ge r_1$, then
$$
|\calN f(x)| \le \frac{\eps_1 \|f\|_{\infty}}{2 a \tau},
$$ 
where $r_1$ depends on $\eps_1, \tau, \alpha, d, \eta_1, \eta_2, \eta_3$. Put $A = \{x \in \R^d: \, \dist(x,\supp(f)) \ge r_1\}$.
We have, applying also (\ref{N_bound}), 
$$
\left| \calN f(x) \right|
=  \left| \calN f(x) 1_{A^c}(x)
+ \calN f(x) 1_{A}(x)\right|
\le \lambda \|f\|_{\infty} 1_{A^c}(x) + \frac{\eps_1 \|f\|_{\infty}}{2 a \tau} 1_{A}(x).
$$
Put $\tilde{f}(x) = \lambda \|f\|_{\infty} 1_{A^c}(x)$, 
$\tilde{\tilde{f}}(x) =  \frac{\eps_1 \|f\|_{\infty}}{2 a \tau} 1_{A}(x)$.

We have 
\begin{equation}
\label{int_u1}
\left| \int_0^t  U_{t-s}^{(\xi)}(\calN(f))(x) \, ds \right| 
\le \int_0^t  U_{t-s}^{(\xi)}(\tilde{f})(x) \, ds
+ \int_0^t  U_{t-s}^{(\xi)}(\tilde{\tilde{f}})(x) \, ds.
\end{equation}
For any $x \in \R^d$ we get
\begin{equation}
\label{int_u2}
\int_0^t  U_{t-s}^{(\xi)}(\tilde{\tilde{f}})(x) \, ds
\le \frac{\eps_1 \|f\|_{\infty}}{2}.
\end{equation}
By Lemma \ref{eps_u}, there exists $r_{2} \ge r_1$ such that, if $\dist(x,\supp(f)) \ge r_{2}$ and $s \in (0,t)$, then
$$
\left| U_{t-s}^{(\xi)}(\tilde{f})(x) \right| \le \frac{\eps_1 \|f\|_{\infty}}{2 \tau},
$$
where $r_{2}$ depends on $\eps_1, r_1, \tau, \alpha, d, \eta_1, \eta_2, \eta_3$.
It follows that, if $\dist(x,\supp(f)) \ge r_{2}$, then
\begin{equation}
\label{int_u3}
\int_0^t  U_{t-s}^{(\xi)}(\tilde{f})(x) \, ds
\le  \frac{\eps_1 \|f\|_{\infty}}{2}.
\end{equation}
Finally, (\ref{int_u1}-\ref{int_u3}) imply the assertion of the lemma.
\end{proof}

\begin{lemma}
\label{distant_support}
Assume that $f \in \Bb$. For any $\eps_1 > 0$ there exists $r \ge 1$ (depending on $\eps_1, \tau, \alpha, d, \eta_1, \eta_2, \eta_3$), such that for any $\xi \in [0,1]$, $t \in [0,\tau]$, $x \in \R^d$, if $\dist(x, \supp(f)) \ge r$, then 
$|T_t^{(\xi)} f(x)| \le \sum_{n = 0}^{\infty} |\Psi_{n,t}^{(\xi)} f(x)| \le \eps_1 \|f\|_{\infty}$.
\end{lemma}
\begin{proof} Fix $f \in \Bb$.
Put 
$$
M = \sup_{n \in \N} \frac{c_1^{n+1} \tau^n}{n!},
$$
where $c_1$ is a constant from Lemma \ref{psint}.
Let $\xi \in [0,1]$, $t \in [0,\tau]$ be arbitrary such that $t + \xi > 0$. Choose $\eps_2 > 0$. By Lemma \ref{psint} there exists $n_0$ such that for any $x \in \Rd$ we have
\begin{equation}
\label{n0}
\sum_{n = n_0}^{\infty}  \left| \Psi_{n,t}^{(\xi)} f(x)\right| \le \eps_2 \|f\|_{\infty}.
\end{equation}

Put $r_{-1} = 1$. Now we will show that for any $n \in \N$ there exists $r_n \ge r_{n-1}$ such that, if $\dist(x,\supp(f)) \ge r_n$, then 
\begin{equation}
\label{induction_psint}
 \left|\Psi_{n,t}^{(\xi)} f(x) \right| \le 2^n a^n \eps_2 \|f\|_{\infty},
\end{equation}
where $a \ge 1$ is a constant from Lemma \ref{integral_u} and $r_n$ depends on $r_{n - 1},\eps_2, \tau, \alpha, d, \eta_1, \eta_2, \eta_3$.

By Lemma \ref{eps_u} there exists $r_0 \ge r_{-1}$ 
such that, if $\dist(x,\supp(f)) \ge r_0$, then 
\begin{equation}
\left| \Psi_{0,t}^{(\xi)} f(x) \right| =
\left| U_t^{(\xi)} f(x) \right| \le \eps_2 \|f\|_{\infty},
\end{equation}
where $r_0$ depends on $\eps_2, \tau, \alpha, d, \eta_1, \eta_2, \eta_3$.

Assume that (\ref{induction_psint}) holds for $n \in \N$. We will show it for $n + 1$. Put $A_n = \{x \in \R^d: \, \dist(x,\supp(f)) \ge r_n\}$. We have
\begin{eqnarray*}
 \left| \Psi_{n,t}^{(\xi)} f(x) \right|
&=&  \left|  \Psi_{n,t}^{(\xi)} f(x) 1_{A_n^c}(x)
+   \Psi_{n,t}^{(\xi)} f(x) 1_{A_n}(x)\right|\\
&\le& M \|f\|_{\infty} 1_{A_n^c}(x) + 2^n a^n \eps_2 \|f\|_{\infty} 1_{A_n}(x).
\end{eqnarray*}
Put $f_n(x) = M \|f\|_{\infty} 1_{A_n^c}(x)$, $\tilde{f}_n(x) =  2^n a^n \eps_2 \|f\|_{\infty} 1_{A_n}(x)$.

We have 
$$
\Psi_{n+1,t}^{(\xi)} f(x) 
= \int_0^t  U_{t-s}^{(\xi)}(\calN(\Psi_{n,s}^{(\xi)}f))(x) \, ds.
$$
Hence,
\begin{equation}
\label{ind_psi1}
\left| \Psi_{n+1,t}^{(\xi)} f(x) \right| 
\le \int_0^t  U_{t-s}^{(\xi)}(\calN(f_n))(x) \, ds
+ \int_0^t  U_{t-s}^{(\xi)}(\calN(\tilde{f}_n))(x) \, ds.
\end{equation}
By Lemma \ref{integral_u}, for any $x \in \R^d$, we get
\begin{equation}
\label{ind_psi2}
\int_0^t  U_{t-s}^{(\xi)}(\calN(\tilde{f}_n))(x) \, ds
\le a \|\tilde{f}_{n}\|_{\infty}
\le 2^n a^{n+1} \eps_2 \|f\|_{\infty}.
\end{equation}
By Lemma \ref{eps_integral_u}, there exists $r_{n+1} \ge r_n$ such that, if $\dist(x,\supp(f)) \ge r_{n+1}$, then
\begin{equation}
\label{ind_psi3}
\int_0^t  U_{t-s}^{(\xi)}(\calN(f_n))(x) \, ds
\le  \eps_2 \|f\|_{\infty},
\end{equation}
where $r_{n+1}$ depends on $\eps_2, r_n, \tau, \alpha, d, \eta_1, \eta_2, \eta_3$.

By (\ref{ind_psi1}-\ref{ind_psi3}) we obtain (\ref{induction_psint}) for $n+1$.
By (\ref{induction_psint}), we obtain that, if $\dist(x,\supp(f)) \ge r_{n_0}$, then
$$
\sum_{n=0}^{n_0}  \left| \Psi_{n,t}^{(\xi)} f(x) \right|
\le \eps_2  \sum_{n=0}^{n_0} 2^n a^n \|f\|_{\infty}.
$$
Using this and (\ref{n0}) we get the assertion of the lemma.
\end{proof}

By Lemma \ref{distant_support} and Theorem \ref{TtL1Linfty} one easily obtains the following result.
\begin{corollary}
\label{convergence}
Assume that $f \in \Bb$, for any $n \in \N$, $n \ge 1$ we have $f_n \in \Bb$, $\sup_{n \in \N, n \ge 1} \|f_n\|_{\infty} < \infty$ and $\lim_{n \to \infty} f_n(x) = f(x)$ for almost all $x \in \R^d$ with respect to the Lebesgue measure. Then, for any $t > 0$, $x \in \R^d$, we have
$\lim_{n \to \infty} T_t f_n(x) = T_t f(x)$.
\end{corollary}

\begin{lemma}
\label{regularity_pteps}
(i) For every $f \in C_0(\R^d)$ we have
\begin{equation*}
\lim_{t,\xi \to 0^+}\|T_t^{(\xi)} f - f\|_{\infty} = 0.
\end{equation*}

(ii) For every $f \in C_0(\R^d)$ we have
\begin{equation*}
T_t^{(\xi)} f(x) \to 0, \quad \text{as} \quad |x| \to \infty,
\end{equation*}
uniformly in $t \in [0,\tau]$, $\xi \in [0,1]$.

(iii) For every $f \in C_0(\R^d)$ we have
\begin{equation*}
\|T_t^{(\xi)} f - T_t f\|_{\infty} \to 0, \quad \text{as} \quad \xi \to 0^+,
\end{equation*}
uniformly in $t \in [0,\tau]$.
\end{lemma}
\begin{proof}
(i) This follows from Lemma \ref{regularity_uteps} (ii) and Lemma \ref{psint}.

(ii) Note that $T_0^{(\xi)} f = U_0^{(\xi)}$ so (ii) for $t = 0$ follows from Lemma \ref{regularity_uteps}. So we may assume that $t > 0$. Let $t \in (0,\tau]$, $\xi \in [0,1]$ be arbitrary. By Lemma \ref{ptxy} we have
\begin{equation}
\label{ttxi}
\left\|T_t^{(\xi)} f\right\|_{\infty} \le c_1 \|f\|_{\infty}.
\end{equation}
Choose $\eps_1 > 0$. Since $f \in C_0(\R^2)$ there exists $r_1 > 0$ such that if $|x| \ge r_1$ then $|f(x)| \le \eps_1/(2 c_1)$, where $c_1$ is a constant from (\ref{ttxi}). Put $f_1(x) = f(x) 1_{B(0,r_1)}(x)$, $f_2(x) = f(x) 1_{B^c(0,r_1)}(x)$. By Lemma \ref{distant_support}, there exists $r_2 > r_1$ such that, if $|x| \ge r_2$, then $|T_t^{(\xi)} f_1(x)| \le \eps_1/2$. Hence for any $|x| \ge r_2$ we have $|T_t^{(\xi)} f(x)| \le |T_t^{(\xi)} f_1(x) + T_t^{(\xi)} f_2(x)| \le \eps_1/2 + (\eps_1/(2 c_1)) c_1 = \eps_1$. 

(iii) Let $\xi \in (0,1]$, $n \in \N$, $n \ge 1$, $x \in \R^d$, $f \in C_0(\R^d)$.
Note that for any $t \in (0,\tau]$ we have
\begin{eqnarray}
\nonumber
\Psi_{n,t}^{(\xi)} f(x) 
&=& \int_0^t \int_{\R^d} p_y(t-s+\xi,x-y) \calN(\Psi_{n-1,s}^{(\xi)}f)(y) \, dy \, ds\\
\label{formula_psinteps}
&+& \int_0^t \int_0^{t-s} \int_{\R^d} p_z(t-s-r+\xi,x-z) Q_r(\calN(\Psi_{n-1,s}^{(\xi)}f))(z) \, dz \, dr \, ds.
\end{eqnarray}
By the same arguments as in the proof of Lemma \ref{regularity_uteps} (iv) we obtain that $(\xi,t,x) \to \Psi_{n,t}^{(\xi)} f(x)$ is continuous on $[0,1] \times [0,\tau] \times \R^d$ for any $n \in \N$, $n \ge 1$. Using this, (\ref{ut_continuity}) and Lemma \ref{psint} we obtain that $(\xi,t,x) \to T_t^{(\xi)} f(x)$ is continuous on $[0,1] \times [0,\tau] \times \R^d$. This and (ii) implies (iii).
\end{proof}

By the same arguments as in the proof of Lemma \ref{regularity_uteps} (iv) we obtain the following result.
\begin{lemma}
\label{psi_continuity}
For any $f \in \Bb$, $n \in \N$, the function $(t,x) \to \Psi_{n,t} f(x)$ is continuous on $(0,\infty) \times \R^d$. For any 
$\xi \in (0,1]$, $f \in \Bb$, $n \in \N$, the function $(t,x) \to \Psi_{n,t}^{(\xi)} f(x)$ is continuous on $[0,\infty) \times \R^d$.
\end{lemma}

\begin{lemma}
\label{tderivative_psi_eps}
$\frac{\partial}{\partial t} \left(\Psi_{n,t}^{(\xi)} f\right)(x)$, $\calL \left(\Psi_{n,t}^{(\xi)} f\right)(x)$ are well defined for any $t > 0$, $\xi \in (0,1]$, $x \in \R^d$, $n \in \N$, $n \ge 1$ and $f \in C_0(\R^d)$ and we have
\begin{eqnarray*}
 \frac{\partial}{\partial t} \left(\Psi_{n,t}^{(\xi)} f\right)(x) 
- \calL \left(\Psi_{n,t}^{(\xi)} f\right)(x) 
&=& \int_{\R^d} p_z(\xi,x-z) \calN \left(\Psi_{n-1,t}^{(\xi)} f\right)(z) \, dz\\
&& + \int_0^t \Lambda_{t-s}^{(\xi)} \left(\calN \left(\Psi_{n-1,s}^{(\xi)} f\right)\right) (x) \, ds.
\end{eqnarray*}
Moreover, $\frac{\partial}{\partial t} \Psi_{n,t}^{(\xi)} f(x)$ is continuous as a function of $t$ for $t > 0$.
\end{lemma}
\begin{proof} Let $\xi \in (0,1]$, $n \in \N$, $n \ge 1$, $x \in \R^d$, $f \in C_0(\R^d)$.
Note that for any $t \in (0,\tau]$, $s \in (0,t)$, we have
\begin{eqnarray*}
U_{t-s}^{(\xi)}(\calN(\Psi_{n-1,s}^{(\xi)}f))(x)
&=& \int_{\R^d} p_y(t-s+\xi,x-y) \calN(\Psi_{n-1,s}^{(\xi)}f)(y) \, dy\\
&+& \int_0^{t-s} \int_{\R^d} p_z(t-s-r+\xi,x-z) Q_r(\calN(\Psi_{n-1,s}^{(\xi)}f))(z) \, dz \, dr.
\end{eqnarray*}

By similar arguments as in the proof of Lemma \ref{regularity_uteps} (i) we obtain that \\ $\frac{\partial}{\partial t} U_{t-s}^{(\xi)}(\calN(\Psi_{n-1,s}^{(\xi)}f))(x)$ is well defined and continuous as a function of $t$ for $t \in (s,\tau]$. Note that for any $g \in C_0(\R^d)$, and $t \ge 0$ we have $U_t^{(\xi)} g \in C_0(\R^d)$, $\calN g \in C_0(\R^d)$. By (\ref{Psi_nteps}), (\ref{tderivative_estimate}), Lemmas \ref{u_continuity}, \ref{psi_continuity}, \ref{regularity_uteps} (i) and standard arguments we get 
\begin{eqnarray}
\nonumber
&& \frac{\partial}{\partial t} \left(\Psi_{n,t}^{(\xi)} f\right)(x)
=U_0^{(\xi)} \left(\calN\left(\Psi_{n-1,t}^{(\xi)} f\right)\right)(x) + 
\int_0^t \frac{\partial}{\partial t} U_{t-s}^{(\xi)} \left(\calN \left(\Psi_{n-1,s}^{(\xi)} f\right)\right)(x) \, ds\\
\nonumber
&& = \int_{\R^d} p_z(\xi,x-z) \calN(\Psi_{n-1,t}^{(\xi)} f)(z) \, dz
+ \int_0^t \Lambda_{t-s}^{(\xi)}(\calN(\Psi_{n-1,s}^{(\xi)} f))(x) \, ds\\
\label{psintt}
&& + \int_0^t \calL^x(U_{t-s}^{(\xi)}(\calN(\Psi_{n-1,s}^{(\xi)} f)))(x) \, ds
\end{eqnarray}
This implies that $\frac{\partial}{\partial t} \Psi_{n,t}^{(\xi)} f(x)$ is continuous as a function of $t$ for $t \in (0,\tau]$. 

For any $\zeta > 0$ we have
$$
\calL_{\zeta} \left(\Psi_{n,t}^{(\xi)} f\right)(x)
= \int_0^t \calL_{\zeta}^x U_{t-s}^{(\xi)} \left(\calN \left(\Psi_{n-1,s}^{(\xi)} f\right)\right)(x) \, ds.
$$
By the dominated convergence theorem and (\ref{Lzeta_estimate}) we obtain
$$
\calL \left(\Psi_{n,t}^{(\xi)} f\right)(x) = \lim_{\zeta \to 0^+} \calL_{\zeta} \left(\Psi_{n,t}^{(\xi)} f\right)(x)
= \int_0^t \calL^x U_{t-s}^{(\xi)} \left(\calN \left(\Psi_{n-1,s}^{(\xi)} f\right)\right)(x) \, ds.
$$
This and (\ref{psintt}) gives the assertion of the lemma.
\end{proof}

\begin{lemma}
\label{change}
For any $t > 0$, $\xi \in (0,1]$, $x \in \R^d$, $i, j \in \{1,\ldots,d\}$, $k \in \N$ and $f \in C_0(\R^d)$ we have $\Psi_{k,t}^{(\xi)} f(x) \in C^2(\R^d)$ (as a function of $x$) and 
\begin{eqnarray}
\label{sumt}
\frac{\partial}{\partial t} \left(\sum_{n=0}^{\infty}  \Psi_{n,t}^{(\xi)} f \right) (x) &=&
\sum_{n=0}^{\infty}  \frac{\partial}{\partial t} \left(\Psi_{n,t}^{(\xi)} f \right) (x),\\
\label{sumx}
\frac{\partial}{\partial x_i} \left(\sum_{n=0}^{\infty}  \Psi_{n,t}^{(\xi)} f \right) (x) &=&
\sum_{n=0}^{\infty}  \frac{\partial}{\partial x_i} \left(\Psi_{n,t}^{(\xi)} f \right) (x),\\
\label{sumxx}
\frac{\partial^2}{\partial x_i \partial x_j} \left(\sum_{n=0}^{\infty}  \Psi_{n,t}^{(\xi)} f \right) (x) &=&
\sum_{n=0}^{\infty}  \frac{\partial^2}{\partial x_i \partial x_j} \left(\Psi_{n,t}^{(\xi)} f \right) (x),\\
\label{sumL}
\calL \left(\sum_{n=0}^{\infty}  \Psi_{n,t}^{(\xi)} f \right) (x) &=&
\sum_{n=0}^{\infty}  \calL \left(\Psi_{n,t}^{(\xi)} f \right) (x),\\
\label{sumN}
\calN \left(\sum_{n=0}^{\infty}  \Psi_{n,t}^{(\xi)} f \right) (x) &=&
\sum_{n=0}^{\infty}  \calN \left(\Psi_{n,t}^{(\xi)} f \right) (x).
\end{eqnarray}
\end{lemma}
\begin{proof}
Fix $f \in C_0(\R^d)$, $\xi \in (0,1]$. By Lemma \ref{tderivative_psi_eps} we know that $t \to \frac{\partial}{\partial t} (\Psi_{n,t}^{(\xi)} f)(x)$ is continuous on $(0,\tau]$ for each fixed $n \in \N$, $x \in \R^d$. Using this, Lemma \ref{psint}, (\ref{psintt}) and (\ref{tderivative_estimate}) we get (\ref{sumt}).

Let $t \in (0,\tau]$, $n \in \N$, $i, j \in \{1,\ldots,d\}$. The fact that $\frac{\partial}{\partial x_i} \left(\Psi_{n,t}^{(\xi)} f \right) (x)$ is well defined and continuous as a function of $x \in \R^d$ follows from (\ref{formula_psinteps}), Lemmas \ref{pycontinuity}, \ref{properties_pytx}, Proposition \ref{qestimate} and Lemma \ref{regularity_uteps}. By the above arguments we also get 
$$
\left|\frac{\partial}{\partial x_i} \left(\Psi_{n,t}^{(\xi)} f \right) (x)\right| \le c(\xi) \sup_{s \in (0,t]} \left\|\Psi_{n-1,s}^{(\xi)} f \right\|_{\infty}, \quad n \ge 1.
$$
Using this, Lemma \ref{regularity_uteps} and Lemma \ref{psint} we arrive at  (\ref{sumx}). By similar arguments we obtain that $\frac{\partial^2}{\partial x_i \partial x_j} \left(\Psi_{n,t}^{(\xi)} f \right) (x)$ is well defined and continuous as a function of $x \in \R^d$ and
\begin{equation}
\label{psint_xx}
\left|\frac{\partial^2}{\partial x_i \partial x_j} \left(\Psi_{n,t}^{(\xi)} f \right) (x)\right| \le c(\xi) \sup_{s \in (0,t]} \left\|\Psi_{n-1,s}^{(\xi)} f \right\|_{\infty}, \quad n \ge 1.
\end{equation}
Using this and Lemma \ref{psint} we get (\ref{sumxx}).

For any $\zeta > 0$ we have 
$$
\calL_{\zeta} \left(\sum_{n=0}^{\infty}  \Psi_{n,t}^{(\xi)} f \right) (x) =
\sum_{n=0}^{\infty}  \calL_{\zeta} \left(\Psi_{n,t}^{(\xi)} f \right) (x)
$$
This, Lemma \ref{psint}, (\ref{Lzeta_estimate}) and (\ref{L_estimate}) implies (\ref{sumL}). (\ref{sumN}) is easy.
\end{proof}

\begin{corollary}
\label{regularity_pteps1}
For every $f \in C_0(\R^d)$, $\xi \in (0,1]$ the function $T_t^{(\xi)} f(x)$ belongs to $C^1((0,\infty))$ as a function of $t$ and to $C_0^2(\R^d)$ as a function of $x$.
\end{corollary}
\begin{proof}
Fix $f \in C_0(\R^d)$, $\xi \in (0,1]$. The fact that $T_t^{(\xi)} f(x)$ belongs to $C^1((0,\infty))$ as a function of $t$ follows from (\ref{sumt}), Lemma \ref{tderivative_psi_eps}, (\ref{psintt}), (\ref{tderivative_estimate}), Lemma \ref{ueps_estimate} and Lemma \ref{psint}. From the proof of Lemma \ref{change} we know that for each $t > 0$, $n \in \N$, $i, j \in \{1, \ldots, d\}$ the function $\frac{\partial^2}{\partial x_i \partial x_j} \left(\Psi_{n,t}^{(\xi)} f \right) (x)$ is continuous as a function of $x \in \R^d$. The fact that $T_t^{(\xi)} f(x)$ belongs to $C_0^2(\R^d)$ as a function of $x$ follows from (\ref{sumx}), (\ref{sumxx}), (\ref{psint_xx}), Lemma \ref{psint} and Lemma \ref{distant_support}. 
\end{proof}

Heuristically, now our aim is to show that if $\xi$ is small then $\frac{\partial}{\partial t} (T_t^{(\xi)} f)(x) - \calK (T_t^{(\xi)} f)(x)$ is small.
For any $t > 0$, $\xi \in (0,1]$, $x \in \R^d$ and $f \in C_0(\R^d)$ let us denote
$$
\Upsilon_t^{(\xi)} f(x) = 
\frac{\partial}{\partial t} \left(T_{t}^{(\xi)} f\right)(x) 
- \calK \left(T_{t}^{(\xi)} f\right)(x),
$$
$$
\Upsilon_t^{(\xi,1)} f(x) = 
e^{-\lambda t} \sum_{n=1}^{\infty}  
\left[\int_{\R^d} p_z(\xi,x-z) \calN \left(\Psi_{n-1,t}^{(\xi)} f\right)(z) \, dz
- \calN \left(\Psi_{n-1,t}^{(\xi)} f\right)(x) \right],
$$
$$
\Upsilon_t^{(\xi,2)} f(x) = 
e^{-\lambda t} \sum_{n=1}^{\infty} 
\int_0^t \Lambda_{t-s}^{(\xi)} \left(\calN \left(\Psi_{n-1,s}^{(\xi)} f\right)\right) (x) \, ds.
$$
By Lemma \ref{psint}, (\ref{lambdatf}), (\ref{lambdatf2}) and the boundedness of $\calN: L^{\infty}(\R^d) \to L^{\infty}(\R^d)$ the above series are convergent.

\begin{lemma}
\label{upsilon_t}
For any $t > 0$, $\xi \in (0,1]$, $x \in \R^d$ and $f \in C_0(\R^d)$ we have
$$
\Upsilon_t^{(\xi)} f(x) = e^{-\lambda t} \Lambda_{t}^{(\xi)} f(x) + \Upsilon_t^{(\xi,1)} f(x) +
\Upsilon_t^{(\xi,2)} f(x).
$$
\end{lemma}
\begin{proof}
By Lemmas \ref{Lambda_teps}, \ref{tderivative_psi_eps} and \ref{change} we get
\begin{eqnarray*}
&&\frac{\partial}{\partial t} \left(T_{t}^{(\xi)} f\right)(x)
= -\lambda e^{-\lambda t} \Psi_{0,t}^{(\xi)} f(x) 
+ e^{-\lambda t} \calL \left(\Psi_{0,t}^{(\xi)} f\right) (x)
+ e^{-\lambda t} \Lambda_{t}^{(\xi)} f(x)\\
&& \quad -\lambda e^{-\lambda t} \sum_{n=1}^{\infty}  \Psi_{n,t}^{(\xi)} f(x)
+ e^{-\lambda t} \sum_{n=1}^{\infty}  \calL \left(\Psi_{n,t}^{(\xi)} f\right) (x)\\
&& \quad + e^{-\lambda t} \sum_{n=1}^{\infty}  \calN \left(\Psi_{n-1,t}^{(\xi)} f\right) (x) 
+ \Upsilon_t^{(\xi,1)} f(x) + \Upsilon_t^{(\xi,2)} f(x).
\end{eqnarray*}
Again by Lemma \ref{change}, this is equal to
\begin{eqnarray*}
&& \calL \left(e^{-\lambda t} \sum_{n=0}^{\infty}  \Psi_{n,t}^{(\xi)} f \right)(x)
-\lambda \left(e^{-\lambda t} \sum_{n=0}^{\infty}  \Psi_{n,t}^{(\xi)} f(x) \right)\\
&& \quad + \calN \left(e^{-\lambda t} \sum_{n=0}^{\infty}  \Psi_{n,t}^{(\xi)} f \right)(x)
+ e^{-\lambda t} \Lambda_{t}^{(\xi)} f(x)
+ \Upsilon_t^{(\xi,1)} f(x) + \Upsilon_t^{(\xi,2)} f(x).
\end{eqnarray*}
Using the definition of $T_{t}^{(\xi)} f$ and noting that $\calN g(x) - \lambda g(x) = \calR g(x)$ and $\calL g(x) + \calR g(x) = \calK g(x)$ we obtain the assertion of the lemma.
\end{proof}

\begin{lemma}
\label{heat_p}
(i) For any $f \in C_0(\R^d)$ we have
\begin{equation*}
\Upsilon_t^{(\xi)} f(x) \to 0, \quad \text{as} \quad \xi \to 0^+,
\end{equation*}
uniformly in $(t,x) \in [\tau_1,\tau] \times \R^d$ for every $\tau_1 \in (0,\tau)$.

(ii) For any $f \in C_0(\R^d)$ we have
\begin{equation}
\label{Upsilon_s}
\int_0^t \Upsilon_s^{(\xi)} f(x) \, ds \to 0, \quad \text{as} \quad \xi \to 0^+,
\end{equation}
uniformly in $(t,x) \in (0,\tau] \times \R^d$.
\end{lemma}
\begin{proof}
The lemma follows from Lemma \ref{heat_u}, Proposition \ref{pyintegral}, Lemma \ref{distant_support}, Lemma \ref{psint}, (\ref{lambdatf}), (\ref{lambdatf2}) and the boundedness of $\calN: L^{\infty}(\R^d) \to L^{\infty}(\R^d)$.
\end{proof}

The next result (positive maximum principle) is based on the ideas from \cite[Section 4.2]{KK2018}. Its proof is very similar to the proof of \cite[Lemma 4.3]{KK2018} and it is omitted.
\begin{lemma}
\label{max_principle}
Let us consider the function $v: [0,\infty) \times \R^d \to \R$ and the family of functions $v^{(\xi)}: [0,\infty) \times \R^d \to \R$, $\xi \in (0,1]$. Assume that for each $\xi \in (0,1]$ $\sup_{t \in (0,\tau], x \in \R^d} |v^{(\xi)}(x,t)| < \infty$, $v^{(\xi)}$ is $C^1$ in the first variable and $C^2$ in the second variable. We also assume that (for any $\tau > 0$)

(i)
$$
v^{(\xi)}(t,x) \to v(t,x) \quad \text{as} \quad \xi \to 0^+,
$$
uniformly in $t \in [0,\tau]$, $x \in \R^d$;

(ii)
$$
v^{(\xi)}(t,x) \to 0 \quad \text{as} \quad |x| \to \infty,
$$
uniformly in $t \in [0,\tau]$, $\xi \in (0,1]$;

(iii) for any $0 < \tau_1 < \tau_2 \le \tau$
$$
\frac{\partial}{\partial t} v^{(\xi)}(t,x)
- \calK v^{(\xi)}(t,x) \to 0 \quad \text{as} \quad \xi \to 0^+,
$$
uniformly in $t \in [\tau_1,\tau_2]$, $x \in \R^d$;

(iv)
$$
v^{(\xi)}(t,x) \to v(0,x) \quad \text{as} \quad (\xi \to 0^+\,\,\, \text{and}\,\,\, t \to 0^+),
$$
uniformly in $x \in \R^d$;

(v) for any $x \in \R^d$ $v(0,x) \ge 0$.

Then for any $t \ge 0$, $x \in \R^d$ we have $v(t,x) \ge 0$.
\end{lemma}

\begin{proposition}
\label{borelbounded}
$T_t: \Bb \to \Bb$ is a linear, bounded operator for any $t \in (0,\tau]$. For each $t \in (0,\tau]$, $f \in \Bb$ and $R \ge 1$ there exists a sequence $f_k \in C_0(\R^d)$, $k \in \N$ such that $\lim_{k \to \infty} f_k(x) = f(x)$ for almost all $x \in B(0,R)$; for any $k \in \N$ we have $\|f_k\|_{\infty} \le \|f\|_{\infty}$ and for any $x \in B(0,R)$ we have $\lim_{k \to \infty} T_t f_k(x) = T_t f(x)$.
\end{proposition}
\begin{proof}
Fix $t \in (0,\tau]$. The fact that $T_t: \Bb \to \Bb$ is a linear, bounded operator follows by the definition of $T_t$ and Lemma \ref{psint}. 

Fix $f \in \Bb$, $R \ge 1$ and $k \in \N$, $k \ge 1$. By Lemma \ref{distant_support} there exists 
$R_k \ge R$ such that for any $x \in B(0,R)$ we have
\begin{equation}
\label{large}
|T_t(f 1_{B^c(0,R_k)})(x)| \le \frac{1}{k}.
\end{equation}
Put $g_{1,k}(x) = 1_{B(0,R_k)}(x) f(x)$, $g_{2,k}(x) = 1_{B^c(0,R_k)}(x) f(x)$. By standard methods there exists $f_k \in C_0(\R^d)$ such that 
$$
\|f_k - g_{1,k}\|_{1} \le \frac{1}{k}.
$$
and $\supp(f_k) \subset B(0,R_k + 1)$, $\|f_k\|_{\infty} \le \|f\|_{\infty}$. 
By Theorem \ref{TtL1Linfty}, for any $x \in \R^d$, we have
$$
|T_t(f_k - g_{1,k})(x)| \le 
\frac{c \|f\|_{\infty}^{1-\alpha/(2d)}}{k^{\alpha/(2d)} t^{1/2}}.
$$
This and (\ref{large}) imply that for any $x \in B(0,R)$ we have $\lim_{k \to \infty} T_t f_k(x) = T_t f(x)$. We also have $\|f_k 1_{B(0,R)} - f 1_{B(0,R)}\|_1 \le 1/k$. Hence, there exists a subsequence $k_m$ such that $\lim_{m \to \infty} f_{k_m}(x) = f(x)$ for almost all $x \in B(0,R)$.
\end{proof}

\begin{proposition}
\label{martingaleproblem}
For any $t \in (0,\infty)$, $x \in \R^d$ and $f \in C_b^2(\R^d)$ we have
\begin{equation}
\label{martingaleproblem1}
T_t f(x) = f(x) + \int_0^t T_s (\calK f)(x) \, ds.
\end{equation}
\end{proposition}
\begin{proof}
Step 1. $f \in C_0^2(\R^d)$.

For any $t \ge 0$, $x \in \R^d$, $\xi \in (0,1]$ put
\begin{eqnarray*}
v(t,x) &=& T_t f(x) - f(x) - \int_0^t T_s (\calK f)(x) \, ds,\\
v^{(\xi)}(t,x) &=& T_t^{(\xi)} f(x) - f(x) - \int_0^t T_s^{(\xi)} (\calK f)(x) \, ds.
\end{eqnarray*}
Note that $\calK f \in C_0(\R^d)$. By Lemmas \ref{regularity_pteps}, \ref{heat_p}  and Corollary \ref{regularity_pteps1} we obtain that $v(t,x)$, $v^{(\xi)}(t,x)$ satisfy the  assumptions of Lemma \ref{max_principle}. Note that $v(0,x) = 0$ for all $x \in \R^d$. The assertion of the proposition for $f \in C_0^2(\R^d)$  follows from Lemma \ref{max_principle}.

Step 2. $f \in C_b^2(\R^d)$.

By standard methods there exists a sequence $f_n \in C_0^2(\R^d)$, $n = 1, 2, \ldots$ such that 
$$ 
\sup_{n \in \N} \max_{i,j \in \{1,\ldots, d\}} \sup_{x \in \R^d} \left(|f_n(x)| + \left|\frac{\partial f_n}{\partial x_i} (x)\right|
+\left|\frac{\partial^2 f_n}{\partial x_i \partial x_j} (x)\right| \right)< \infty.
$$
and for any $r \ge 1$ we have $\lim_{n \to \infty} (\sup_{|x| \le r} |f_n(x) - f(x)|) = 0$. It follows that $\sup_{n \ge 1, x \in \R^d} |\calK f_n(x)| < \infty$ and for each $x \in \R^d$ we have $\calK f_n(x) \to \calK f(x)$. By Corollary \ref{convergence} it follows that for each $x \in \R^d$, $t > 0$ and $s \in (0,t]$ we have $T_t f_n(x) \to T_t f(x)$ and $T_s (\calK f_n)(x) \to T_s(\calK f)(x)$. By Step 1 and the dominated convergence theorem we obtain the assertion of the proposition.
\end{proof}

The following result shows that $\{T_t\}$ is a Feller semigroup.
\begin{theorem}
\label{FellerT}
We have

(i) $T_t: C_0(\R^d) \to C_0(\R^d)$ for any $ t \in (0,\infty)$,

(ii) $T_t f (x) \ge 0$ for any $t > 0$, $x \in \R^d$ and $f \in C_0(\R^d)$ such that $f(x) \ge 0$ for all $x \in \R^d$,

(iii) $T_t 1_{\R^d}(x) = 1$ for any $t > 0$, $x \in \R^d$,

(iv) $T_{t + s} f(x) = T_t(T_s f)(x)$ for any $s, t > 0$, $x \in \R^d$, $f \in C_0(\R^d)$,

(v) $\lim_{t \to 0^+} ||T_t f - f||_{\infty} = 0$ for any $f \in C_0(\R^d)$.

(vi) there exists a nonnegative function $p(t,x,y)$ in $(t,x,y) \in (0,\infty)\times\Rd\times\Rd$; for each fixed $t > 0$, $x \in \R^d$ the function $y \to p(t,x,y)$ is Lebesgue measurable, $\int_{\R^d} p(t,x,y) \, dy = 1$ and $T_t f(x) = \int_{\R^d} p(t,x,y) f(y) \, dy$ for $f \in C_0(\R^d)$.
\end{theorem}
\begin{proof}
(i) This follows from Theorem \ref{Holdermain} and Lemma \ref{regularity_pteps} (ii).

(ii) Let $f \in C_0(\R^d)$ be such that $f(x) \ge 0$ for all $x \in \R^d$. For $t \ge 0$, $x \in \R^d$, $\xi \in (0,1]$ put $v(t,x) = T_t f(x)$, $v^{(\xi)}(t,x) = T_t^{(\xi)} f(x)$. By Lemmas \ref{regularity_pteps}, \ref{heat_p} and Corollary \ref{regularity_pteps1} we obtain that $v(t,x)$, $v^{(\xi)}(t,x)$ satisfy the  assumptions of Lemma \ref{max_principle}. The assertion of Theorem \ref{FellerT} (ii) follows from Lemma \ref{max_principle}.

(iii) The proof is very similar to the proof of \cite[Lemma 4.5 b]{KK2018}. Let $f \in C_0^2(\R^2)$ be such that $f \equiv 1$ on $B(0,1) \subset \R^d$ and let $f_n(x) = f(x/n)$, $x \in \R^d$, $n \in \N$, $n \ge 1$. For any $x \in \R^d$ we have $\lim_{n \to \infty} f_n(x) = 1$, 
$\lim_{n \to \infty} \calK f_n(x) = 0$ and $\sup_{n \in \N, n\ge 1} (\|f_n\|_{\infty} \vee \|\calK f_n\|_{\infty}) < \infty$. By Corollary \ref{convergence}, for any $s, t > 0$ and $x \in \R^d$, we get
\begin{equation}
\label{convergence1}
\lim_{n \to \infty} T_t f_n(x) = T_t 1_{\R^d}(x), \quad \quad \lim_{n \to \infty} T_s (\calK f_n)(x) = 0.
\end{equation}
Using (\ref{martingaleproblem1}) for $f_n$ and (\ref{convergence1}) we obtain (iii).

(iv) Let $f \in C_0(\R^d)$. For $s, t \ge 0$, $x \in \R^d$, $\xi \in (0,1]$ put 
$v(t,x) = T_{t+s} f(x) - T_t (T_s f)(x)$, 
$v^{(\xi)}(t,x) = T_{t+s}^{(\xi)} f(x) - T_t^{(\xi)} (T_s f)(x)$. By Lemmas \ref{regularity_pteps}, \ref{heat_p}  and Corollary \ref{regularity_pteps1} we obtain that $v(t,x)$, $v^{(\xi)}(t,x)$ satisfy the assumptions of Lemma \ref{max_principle}. Note that $v(0,x) = 0$ for all $x \in \R^d$. The assertion of Theorem \ref{FellerT} (iv) follows from Lemma \ref{max_principle}.

(v) Choose $\eps_1 > 0$. Since $f \in C_0(\R^d)$ there exists $\delta_1 > 0$ such that 
$$
\forall x,y \in \R^d \quad |x - y| < \delta_1 \Rightarrow |f(x) - f(y)| < \eps_1. 
$$
Fix arbitrary $x \in \R^d$, $t \in (0,\tau]$. Put $f_1(y) = 1_{B(x,\delta_1)}(y) (f(y) - f(x))$, $f_2(y) = 1_{B^c(x,\delta_1)}(y) (f(y) - f(x))$, $y \in \R^d$. By (iii) we have 
$$
T_t f(x) - f(x) = T_t f_1(x) + T_t f_2(x).
$$
We  also have
$$
\left|T_t f_1(x)\right| < c \eps_1, 
$$
$$
\left|T_t f_2(x)\right| \le 2 \|f\|_{\infty} T_t 1_{B^c(x,\delta_1)}(x)
$$
and
\begin{eqnarray*}
T_t 1_{B^c(x,\delta_1)}(x) &=&
e^{-\lambda t}\int_{B^c(x,\delta_1)} p_y(t,x-y) \, dy
+ e^{-\lambda t} \Phi_t 1_{B^c(x,\delta_1)}(x)\\
&& + e^{-\lambda t} \sum_{n=1}^{\infty} \Psi_{n,t}1_{B^c(x,\delta_1)}(x).
\end{eqnarray*}
By Proposition \ref{pyintegral} there exists $\tau_1 \in (0,\tau]$ such that 
$$
\forall t \in (0,\tau_1] \quad \int_{B^c(x,\delta_1)} p_y(t,x-y) \, dy < \eps_1.
$$
By Proposition \ref{qestimate} we obtain that
$$
\forall t \in (0,\tau_1]  \quad \left|\Phi_t 1_{B^c(x,\delta_1)}(x)\right| \le c \tau_1^{1/2}.
$$
By Lemma \ref{psint} we obtain that 
$$
\forall t \in (0,\tau_1] \quad \left| e^{-\lambda t} \sum_{n=1}^{\infty} \Psi_{n,t}1_{B^c(x,\delta_1)}(x) \right| \le c t.
$$
This implies (v).

(vi) This follows from (i), (ii), (iii) and Theorem \ref{TtL1Linfty}.
\end{proof}

We are now in a position to provide the proof of Theorems  \ref{mainthm} and \ref{PtL1Linfty}.
\begin{proof}[proof of Theorem \ref{mainthm}] 
From Theorem \ref{FellerT}  we conclude that there is a Feller process $\tilde{X}_t$ with the semigroup $T_t$ on $C_0(\R^d)$. 
Let  $\tilde{\p}^x, \tilde{\E}^x $ be the distribution and expectation of the process $\tilde{X}_t$ starting from $x\in \R^d$. 

By Theorem \ref{FellerT} (vi), Proposition \ref{borelbounded} and Lemma \ref{distant_support} we get
\begin{equation}
\label{semigroup_t}
\tilde{\E}^x f(\tilde{X}_t) = T_t f(x) = \int_{\R^d} p(t,x,y) f(y) \, dy \quad f \in \Bb, \, t > 0, \, x \in \R^d.
\end{equation}
By Proposition \ref{martingaleproblem}, for any function $f \in C_b^2(\R^d)$, 
the process
\begin{equation*} \label{martingale0}
M_t^{\tilde{X},f}=f(\tilde{X}_t)-f(\tilde{X}_0)- \int_0^t \calK f(\tilde{X}_s)ds 
\end{equation*}
is a  $(\tilde{\p}^x, \calF_t)$ martingale, where $\calF_t  $ is a natural filtration. That is  $\tilde{\p}^x$ solves the martingale problem for 
$(\calK, C_b^2(\R^d))$. On the other hand, according to  \cite[Theorem 6.3]{BC2006}, the unique solution $X$ to the stochastic equation
(\ref{main}) has the law which is the unique solution to the  martingale problem for 
$(\calK, C_b^2(\R^d))$. Hence $\tilde{X}$ and $X$ have the same law so for any $t > 0$, $x \in \R^d$  and any Borel bounded set $A \subset \R^d$ we have
$$
\sigma_t(x,A) = \int_{A} p(t,x,y) \, dy,
$$
where $\sigma_t(x,A)$ is defined by (\ref{sigmatx}). Using this, (\ref{semigroup_p}) and (\ref{semigroup_t}) we obtain
\begin{equation}
\label{pttt}
P_t f(x) = T_t f(x), \quad t > 0, \, x \in \R^d, \, f \in \Bb.
\end{equation}
Now the assertion of Theorem \ref{mainthm} follows from Theorem \ref{Holdermain} and (\ref{pttt}).
\end{proof}

\begin{proof}[proof of Theorem \ref{PtL1Linfty}] 
The result follows from Theorem \ref{TtL1Linfty} and (\ref{pttt}).
\end{proof}

\begin{remark}
\label{example1}
For any $\alpha \in (0,1)$, $d \ge 2$ there exist $A(x)$ satisfying (\ref{bounded}-\ref{Lipschitz}) and $t > 0$ such that $P_t: L^1(\R^d) \to L^{\infty}(\R^d)$ is not bounded. For simplicity we will present an example for $d = 2$ but similar examples can be constructed for $d > 2$.
\end{remark}
\begin{proof}
First we define $A(x_1,x_2)$. Let $\kappa(r): [0,\infty) \to [0,\infty)$ be defined by $\kappa(r) = 0$ for $r \in [0,1]$, $\kappa(r) = r-1$ for $r \in (1,1+\pi/4]$, $\kappa(r) = \pi/4$ for $r > 1+\pi/4$. It is easy to check that $\kappa(r) = ((r-1)\vee 0)\wedge(\pi/4)$ and that it is a Lipschitz function. Now let us introduce standard polar coordinates $(r,\varphi)$, $r \in [0,\infty)$, $\varphi \in [0,2\pi)$ by $x_1 = r \cos \varphi$, $x_2 = r \sin \varphi$. 

We put
\begin{math}
A(x_1,x_2) = \tilde{A}(r,\varphi) = \left[\begin{array}{cc}      
\cos(\tilde{\theta}(r,\varphi)) & - \sin(\tilde{\theta}(r,\varphi))\\ 
\sin(\tilde{\theta}(r,\varphi)) & \cos(\tilde{\theta}(r,\varphi))  
 \end{array}\right],
\end{math}
where $\tilde{\theta}(r,\varphi)$ is defined in the following way. $\tilde{\theta}(r,\varphi) = 0$ for $r \in [0,1]$, $\varphi \in [0,2\pi)$, $\tilde{\theta}(r,\varphi) = \varphi$ for $r \in (1,1+\pi/4]$, $\varphi \in [0,\kappa(r)]$, $\tilde{\theta}(r,\varphi) = 2 \kappa(r) - \varphi$ for $r \in (1,1+\pi/4]$, $\varphi \in (\kappa(r),2\kappa(r)]$, $\tilde{\theta}(r,\varphi) = 0$ for $r \in (1,1+\pi/4]$, $\varphi \in [2\kappa(r),2\pi)$, $\tilde{\theta}(r,\varphi) = \varphi$ for $r > 1+\pi/4$, 
$\varphi \in [0,\pi/4]$, $\tilde{\theta}(r,\varphi) = \pi/2 - \varphi$ for $r > 1+\pi/4$, 
$\varphi \in (\pi/4,\pi/2]$, $\tilde{\theta}(r,\varphi) = 0$ for $r > 1+\pi/4$, 
$\varphi \in (\pi/2,2\pi)$.

One can check that \begin{math}
A(x_1,x_2) =  \left[\begin{array}{cc}      
\cos({\theta}(x_1,x_2)) & - \sin({\theta}(x_1,x_2))\\ 
\sin({\theta}(x_1,x_2)) & \cos({\theta}(x_1,x_2))  
 \end{array}\right],
\end{math}
where $\theta(0,0)=0$ and for $(x_1,x_2) \ne (0,0)$
$$
\theta(x_1,x_2) = \left(\kappa\left(\sqrt{x_1^2+x_2^2}\right) - 
\left|\left(\left(\Arg(x_1+ix_2) \vee 0\right) \wedge \frac{\pi}{2} \right) - 
\kappa\left(\sqrt{x_1^2+x_2^2}\right)\right|\right) \vee 0.
$$
It is clear that $A(x)$ satisfies (\ref{bounded}-\ref{Lipschitz}).

Put $D = B((3,1),1)$. Note that for any $x \in \R^2$ such that $x_2 \in [0,x_1]$ and $|x| \ge \pi/4 + 1$ we have 
\begin{math}
A(x) = A(x_1,x_2) = |x|^{-1} \left[\begin{array}{cc}      
x_1 & - x_2 \\ 
x_2  & x_1  
 \end{array}\right].
\end{math}
In particular, this holds for $x \in D$. 

For any $f \in \Bb$, $t > 0$, $x \in \R^d$ we have $P_t f(x) = T_t f(x) = e^{-\lambda t} \sum_{n = 0}^{\infty} \Psi_{n,t} f(x)$. For our purposes it is enough to study $\Psi_{1,t}$. We have 
\begin{eqnarray}
\nonumber
&&\Psi_{1,t} f(x) = \int_0^t U_{t-s}(\calN(U_s f))(x) \, ds = \\
\label{psi1t}
&& \sum_{i= 1}^2 \int_0^t \int_{\R^2} u(t-s,x,z) \int_{\R} \int_{\R^2} u(s,z+a_i(z)w,y) f(y) dy
 \nu(w)  dw dz ds.
\end{eqnarray}
By arguments similar to the proof of Theorem \ref{FellerT} one can show that for any $t> 0$, $x \in \R^d$ and almost all $y \in \R^d$ we have $u(t,x,y) \ge 0$ and for any $s, t > 0$, $x \in \R^d$, $f \in \Bb$ we have $U_{t +s} f(x) = U_t(U_s f)(x)$, (we omit the details here). Put 
\begin{equation}
\label{u1txy}
u_1(t,x,y) = \sum_{i= 1}^2 \int_0^t \int_{\R^2} u(t-s,x,z) \int_{\R} u(s,z+a_i(z)w,y) \nu(w) \, dw \, dz \, ds.
\end{equation}
By (\ref{psi1t}) we have
$$
\Psi_{1,t} f(x) = \int_{\R^2} u_1(t,x,y) f(y) \, dy.
$$
By Lemma \ref{lowerbound} and the semigroup property of $U_t$ one can easily obtain that there exists $t_2 > 1$ such that for any $t \in [t_2 - 1,t_2]$
\begin{equation}
\label{positivebound}
\int_D u(t,0,z) \, dz \ge c.
\end{equation}

Now our aim will be to estimate from below $u_1(t_2,0,y)$ for $y$ which are sufficiently close to $0$. Let $c_1,\eps_1,t_1$ be the constants from Lemma \ref{lowerbound}. First, note that for $z \in D$ we have $a_1(z) = (a_{11}(z),a_{21}(z)) = |z|^{-1} (z_1,z_2) = z/|z|$. 
Hence $|z+a_1(z)w| = |z| + w$ for any $z \in D$, $w \in \R$. It follows that for $z \in D$ we have
$$
w \in (-|z|-\eps_1 s^{1/\alpha}/2, -|z|+\eps_1 s^{1/\alpha}/2) \iff |z+a_1(z)w| < \eps_1 s^{1/\alpha}/2.
$$
Therefore, for $z \in D$, $|y| \le \eps_1 s^{1/\alpha}/2$, $w \in (-|z|-\eps_1 s^{1/\alpha}/2, -|z|+\eps_1 s^{1/\alpha}/2)$, we have
$|z+a_i(z)w -y| \le |z+a_i(z)w| + |y| \le \eps_1 s^{1/\alpha}$. Note also that 
$$
|y| \le \eps_1 s^{1/\alpha}/2 \iff (2|y|/\eps_1)^{\alpha} \le s.
$$
Hence, by Lemma \ref{lowerbound}, for $(2|y|/\eps_1)^{\alpha} \le t_1/2$, $s \in [(2|y|/\eps_1)^{\alpha},t_1]$, $z \in D$ we have
\begin{equation}
\label{uwintegral}
\int_{-|z|-\eps_1 s^{1/\alpha}/2}^{-|z|+\eps_1 s^{1/\alpha}/2} u(s,z+a_1(z)w,y) \nu(w) \, dw
\ge c \eps_1 s^{1/\alpha} \frac{c_1}{s^{2/\alpha}} = \frac{c c_1 \eps_1}{s^{1/\alpha}}.
\end{equation}
Recall that $t_1 \in (0,1]$ and $t_2 > 1 \ge t_1$. By (\ref{positivebound}) for any $s \in (0,t_1]$ we have
\begin{equation}
\label{positivebound1}
\int_{D} u(t_2 - s,0,z) \, dz \ge c.
\end{equation}
Therefore, by (\ref{u1txy}), nonnegativity of $u(\cdot,\cdot,\cdot)$, (\ref{uwintegral}), (\ref{positivebound1}) for $(2|y|/\eps_1)^{\alpha} \le t_1/2$ we have
\begin{eqnarray}
\nonumber
u_1(t_2,0,y)
&\ge&
\int_{(2 |y|/\eps_1)^{\alpha}}^{t_1} \int_D u(t_2-s,0,z) 
\int_{-|z|-\eps_1 s^{1/\alpha}/2}^{-|z|+\eps_1 s^{1/\alpha}/2} u(s,z+a_1(z)w,y) \\
\nonumber
&& \times \,\, \nu(w) \, dw \, dz \, ds\\
\nonumber 
&\ge& c \int_{(2 |y|/\eps_1)^{\alpha}}^{t_1} s^{-1/\alpha} \,ds\\
\label{u1estimate}
&\ge& c |y|^{\alpha - 1}.
\end{eqnarray}
(One can show that in similar examples for $d > 2$ we have $u_1(t_2,0,y) \ge c |y|^{ \alpha + 1 - d}$.)

Observe that $(2 |y|/\eps_1)^{\alpha} \le t_1/2 \iff |y| \le t_1^{1/\alpha} \eps_1 2^{-1-1/\alpha}$. For $r \in (0,t_1^{1/\alpha} \eps_1 2^{-1-1/\alpha})$ we get by (\ref{u1estimate})
\begin{equation}
\label{Ttrestimate}
T_{t_2} 1_{B(0,r)}(0) \ge e^{-\lambda t_2} \Psi_{1,t_2} 1_{B(0,r)}(0) = 
e^{-\lambda t_2} \int_{B(0,r)} u_1(t_2,0,y) \, dy
\ge c r^{\alpha+1}.
\end{equation}
By (\ref{pttt}) we have $T_t = P_t$. By Theorem \ref{mainthm} $x \to P_t 1_{B(0,r)}(x)$ is continuous so $\|P_t 1_{B(0,r)}\|_{\infty} \ge P_t 1_{B(0,r)}(0)$. Using this and (\ref{Ttrestimate}) for $r \in (0,t_1^{1/\alpha} \eps_1 2^{-1-1/\alpha})$ we get
$$
\frac{\|P_{t_2} 1_{B(0,r)}\|_{\infty}}{\|1_{B(0,r)}\|_{1}} \ge 
\frac{P_{t_2} 1_{B(0,r)}(0)}{\|1_{B(0,r)}\|_{1}} \ge c r^{\alpha -1},
$$
which implies the assertion of the remark. (One can show that in similar examples for $d > 2$ we have $\|P_{t_2} 1_{B(0,r)}\|_{\infty}/\|1_{B(0,r)}\|_{1}\ge c r^{\alpha+1-d}$.)
\end{proof}

\begin{remark}
\label{example1Uwaga} From Theorem \ref{FellerT} (vi) and (\ref{pttt}) we infer that transition densities $p(t,x,y)$ for $X_t$ exist. We point out that the existence of transition densities is already well known, see \cite{DF2013}. In the above example (in $\R^2$) we showed that the transition density $p(t,0,y)$, for some $t >0$, is an unbounded function. In fact, the following estimate holds $y$ almost surely
$$p(t,0,y)\ge c |y|^{\alpha -1},\quad  |y|\le \eps_1,$$
where  $ c, \eps_1$ are some positive constants possibly dependent on $t$. 

Hence, we can not expect a general result saying that, with our assumptions, we have the standard estimates for $p(t,x,y)$ of the form 
$$p(t,x,y)\le C t^{-d/\alpha },$$ 
as for example in the case of diagonal matrices \cite{KR2017}, or matrices satisfying some further  regularity assumptions \cite{P1997}. 
On the other hand, the assumption $\alpha<1$ plays an important role (in $\R^2$), since for $\alpha>1$, by the results of \cite{BSK2017},
the transition density is bounded. 

\end{remark}

\textbf{ Acknowledgements.} 
We thank prof.  J. Zabczyk for communicating to us the problem of the strong Feller property for solutions of SDEs driven by cylindrical $\alpha$-stable processes. We also thank A. Kulik for discussions on the problem treated in the paper.


\begin{thebibliography}{99}
\bibliographystyle{plain}

\bibitem{BC2006} R. Bass, Z.-Q. Chen, \emph{Systems of equations driven by stable processes}, 
Probab. Theory Related Fields 134, no. 2 (2006), 175-214.

\bibitem{BG60} R.M. Blumenthal, R. K. Getoor, \emph{Some theorems on stable processes}, Trans. Amer. Math. Soc.
95 (1960), 263–273.

\bibitem{BGR2014} K. Bogdan, T. Grzywny, M. Ryznar, \emph{Density and tails of unimodal convolution semigroups}, J. Funct. Anal. 266 (2014), 3543-3571.

\bibitem{BJ2007} K. Bogdan, T. Jakubowski, \emph{Estimates of heat kernel of fractional Laplacian perturbed by gradient operators},
Comm. Math. Phys. 271 (1) (2007), 179-198.

\bibitem{BSK2017} K. Bogdan, P. Sztonyk, V. Knopova, \emph{Heat kernel of anisotropic nonlocal operators}, arXiv:1704.03705 (2017).

\bibitem{CKS2012} Z.-Q. Chen, P. Kim, R. Song, \emph{Dirichlet  heat  kernel  estimates  for  fractional Laplacian  with  gradient perturbation}, Ann. Probab. 40 (2012), 2483-2538.

\bibitem{CZ2016} Z.-Q. Chen, X. Zhang, \emph{Heat kernels and analyticity of non-symmetric jump diffusion semigroups}, 
Probab. Theory Relat. Fields 165, no. 1-2 (2016), 267-312.

\bibitem{CZ2017} Z.-Q. Chen, X. Zhang, \emph{Heat kernels for time-dependent non-symmetric stable-like operators},  arXiv:1709.04614 (2017).

\bibitem{C1976} F. H. Clarke, \emph {On the inverse function theorem},  Pacific Journal of Mathematics
Vol. 64, No 1 (1976), 97-102. 

\bibitem{DF2013} A. Debussche, N. Fournier, \emph{Existence of
densities for stable-like driven SDE’s with H{\"o}lder
continuous coefficients}, J. Funct. Anal. 264(8) (2013), 1757-1778.

\bibitem{DPSZ2016} Z. Dong, X. Peng, Y. Song, X. Zhang, \emph{Strong Feller properties for degenerate SDEs with jumps}, Ann. Inst. H. Poincaré Probab. Statist. 52 (2016), 888-897.

\bibitem{F1975} A. Friedman, Partial Differential Equations of Parabolic Type, Prentice-Hall, Englewood Cliffs, N.J., (1975).

\bibitem{GS2018} T. Grzywny, K. Szczypkowski, \emph{Heat kernels of non-symmetric L{\'e}vy-type operators}, arXiv:1804.01313 (2018).

\bibitem{H1993} P. Haj{\l}asz, \emph{ Change of variables formula under minimal assumptions},
Colloquium Mathematicum 64(1) (1993), 93-101.

\bibitem{J2017} P. Jin, \emph{Heat kernel estimates for non-symmetric stable-like processes},  arXiv:1709.02836 (2017).

\bibitem{KS2015} K. Kaleta, P. Sztonyk, \emph{Estimates of transition densities and their derivatives for jump L{\'e}vy processes}, J. Math. Anal. Appl. 431 (2015), 260–282.

\bibitem{KSV2018} P. Kim, R. Song, Z. Vondra{\v{c}}ek, \emph{Heat kernels of non-symmetric jump processes: beyond the stable case}, Potential Anal. 49(1) (2018), 37-90.

\bibitem{KK2017} V. Knopova, A. Kulik, \emph{Intrinsic compound kernel estimates for the transition probability density of L{\'e}vy-type processes and their applications}, Probab. Math. Statist. 37 (2017), no. 1, 53-100.

\bibitem{KK2018} V. Knopowa, A. Kulik, \emph{Parametrix construction of the transition probability density of the solution to an SDE driven by $\alpha$-stable noise}, Ann. Inst. H. Poincaré Probab. Statist. 54 (2018), 100-140.

\bibitem{K1989} A. N. Kochubei, \emph{Parabolic pseudodifferential equations, hypersingular integrals and Markov processes},
Math. USSR Izv. 33 (1989), 233-259; translation from Izv. Akad. Nauk SSSR, Ser. Mat. 52 (1988), 909-934.

\bibitem{K2017} F. K{\"u}hn, \emph{Transition probabilities of L{\'e}vy-type processes:  Parametrix construction}, arXiv:1702.00778 (2017).

\bibitem{KR2018} T. Kulczycki, M. Ryznar, \emph{Gradient estimates of Dirichlet heat kernels for unimodal L{\'e}vy processes} ,  Math. Nachr. 291 (2018), 374-397.

\bibitem{KR2016} T. Kulczycki, M. Ryznar, \emph{Gradient estimates of harmonic functions and transition densities for L{\'e}vy processes}, Trans. Amer. Math. Soc. 368, no. 1 (2016), 281-318.

\bibitem{KR2017} T. Kulczycki, M. Ryznar, \emph{Transition density estimates for diagonal systems of SDEs driven by cylindrical $\alpha$-stable processes},  ALEA Lat. Am. J. Probab. Math. Stat. 15 (2018), 1335-1375.

\bibitem{KM2014} S. Kusuoka, C. Marinelli, \emph{On smoothing properties of transition semigroups associated to a class of SDEs with jumps}, Ann. Inst. H. Poincaré Probab. Statist. 50 (2014), 1347-1370. 

\bibitem{LSU1968} O. A. Ladyzenskaja, V. A. Solonnikov, N. N. Ural’ceva, Linear and Quasi-linear Equations of Parabolic Type (translated from the Russian by S.Smith), American Mathematical Society, Providence, (1968).

\bibitem{L1907} E. E. Levi, \emph{Sulle equazioni    lineari    totalmente    ellittiche alle    derivate parziali},
Rend. Circ. Mat. Palermo 24 (1907), 275–317.

\bibitem{M1982} M. M{\'e}tivier, \emph{Semimartingales. A Course on Stochastic Processes}, Walter de Gruyter, Berlin (1982).

\bibitem{P1997} J. Picard, \emph{
Density in small time at accessible points for jump processes}, Stochastic Process. Appl. 67,
no.2 (1997), 251-279.


\bibitem{PZ2011} E. Priola, J. Zabczyk, \emph{Structural properties of semilinear SPDEs driven by cylindrical stable processes}, Probab. Theory Related Fields 149 , no. 1-2 (2011), 97-137.

\bibitem{SS2010} R. Schilling, A. Schnurr, \emph{The symbol associated with the solution of a stochastic differential equation}, Electron. J. Probab. 15 (2010), 1369-1393.

\bibitem{SSW2012} R. Schilling, P. Sztonyk and J. Wang, \emph{Coupling property and gradient estimates for L{\'e}vy processes via the symbol},
Bernoulli 18 (2012), 1128-1149.

\bibitem{S2017} P. Sztonyk, \emph{Estimates of densities for L{\'e}vy processes with lower intensity of large jumps}, Math. Nachr. 290, no. 1 (2017), 120-141.

\bibitem{T2010} A. Takeuchi \emph{The Bismut-Elworthy-Li-type formulae for stochastic differential equations with jumps}, J. Theoret. Probab. 23 (2010), 576-604.

\bibitem{WXZ2015} F.-Y. Wang, L. Xu, X. Zhang, \emph{Gradient estimates for SDEs driven by multiplicative Lévy noise}, J. Funct. Anal. 269 (2015), 3195-3219.

\bibitem{WZ2015} L. Wang, X. Zhang, \emph{Harnack Inequalities for SDEs Driven by Cylindrical $\alpha$-Stable Processes}, Potential Anal. 42(3) (2015), 657–669.

\bibitem{XZ2017} L. Xie, X. Zhang, \emph{Ergodicity of stochastic differential equations with jumps and singular coefficients},  	arXiv:1705.07402 (2017).

\bibitem{XZ2014} L. Xie, X. Zhang, \emph{Heat    kernel    estimates for    critical    fractional    diffusion}, Studia Math. 224(3) (2014), 221-263.

\bibitem{Z2013} X. Zhang, \emph{Derivative formulas and gradient estimates for SDEs driven by $\alpha$-stable processes},
Stochastic Process. Appl. 123 (2013) 1213-1228.

\end{thebibliography}
\end{document}